\documentclass[11pt]{amsart}
\allowdisplaybreaks
\usepackage[unicode]{hyperref}
\usepackage{cleveref}
\usepackage{mathtools}
\usepackage[utf8]{inputenc}
\usepackage{xcolor}
\usepackage{graphicx}
\usepackage{url}
\usepackage{tikz}
\usetikzlibrary{intersections}
\usetikzlibrary {decorations.pathmorphing, decorations.pathreplacing, decorations.shapes}
\usepackage[all]{xy}
\usepackage{array}
\usepackage{bbm}  
\usepackage{amssymb}
\usepackage{xcolor}
\makeatletter
\let\save@mathaccent\mathaccent
\newcommand*\if@single[3]{%
  \setbox0\hbox{${\mathaccent"0362{#1}}^H$}%
  \setbox2\hbox{${\mathaccent"0362{\kern0pt#1}}^H$}%
  \ifdim\ht0=\ht2 #3\else #2\fi
  }
\newcommand*\rel@kern[1]{\kern#1\dimexpr\macc@kerna}
\newcommand*\wideaccent[2]{\@ifnextchar^{{\wide@accent{#1}{#2}{0}}}{\wide@accent{#1}{#2}{1}}}
\newcommand*\wide@accent[3]{\if@single{#2}{\wide@accent@{#1}{#2}{#3}{1}}{\wide@accent@{#1}{#2}{#3}{2}}}
\newcommand*\wide@accent@[4]{%
  \begingroup
  \def\mathaccent##1##2{%
    \let\mathaccent\save@mathaccent
    \if#42 \let\macc@nucleus\first@char \fi
    \setbox\z@\hbox{$\macc@style{\macc@nucleus}_{}$}%
    \setbox\tw@\hbox{$\macc@style{\macc@nucleus}{}_{}$}%
    \dimen@\wd\tw@
    \advance\dimen@-\wd\z@
    \divide\dimen@ 3
    \@tempdima\wd\tw@
    \advance\@tempdima-\scriptspace
    \divide\@tempdima 10
    \advance\dimen@-\@tempdima
    \ifdim\dimen@>\z@ \dimen@0pt\fi
    \rel@kern{0.6}\kern-\dimen@
    \if#41
      #1{\rel@kern{-0.6}\kern\dimen@\macc@nucleus\rel@kern{0.4}\kern\dimen@}%
      \advance\dimen@0.4\dimexpr\macc@kerna
      \let\final@kern#3%
      \ifdim\dimen@<\z@ \let\final@kern1\fi
      \if\final@kern1 \kern-\dimen@\fi
    \else
      #1{\rel@kern{-0.6}\kern\dimen@#2}%
    \fi
  }%
  \macc@depth\@ne
  \let\math@bgroup\@empty \let\math@egroup\macc@set@skewchar
  \mathsurround\z@ \frozen@everymath{\mathgroup\macc@group\relax}%
  \macc@set@skewchar\relax
  \let\mathaccentV\macc@nested@a
  \if#41
    \macc@nested@a\relax111{#2}%
  \else
    \def\gobble@till@marker##1\endmarker{}%
    \futurelet\first@char\gobble@till@marker#2\endmarker
    \ifcat\noexpand\first@char A\else
      \def\first@char{}%
    \fi
    \macc@nested@a\relax111{\first@char}%
  \fi
  \endgroup
}
\makeatother

\newcommand\doubleoverline[1]{\overline{\overline{#1}}}
\newcommand\widebar{\wideaccent\overline}
\newcommand\widebarbar{\wideaccent\doubleoverline}

\newtheorem{Thm}{Theorem}[section]
\newtheorem{Prop}[Thm]{Proposition}
\newtheorem{Cor}[Thm]{Corollary}
\newtheorem{Lem}[Thm]{Lemma}
\theoremstyle{definition}
\newtheorem{Rem}[Thm]{Remark}
\newtheorem{Def}[Thm]{Definition}
\newtheorem{Expl}[Thm]{Example}
\numberwithin{equation}{section} 

\newcommand{\blux}[1]{\textcolor{blue}{#1}} 
\newcommand{\cyax}[1]{\textcolor{cyan}{#1}}
\newcommand{\teax}[1]{\textcolor{teal}{#1}}
\newcommand{\redx}[1]{\textcolor{red}{#1}}
\newcommand{\orax}[1]{\textcolor{orange}{#1}}
\newcommand{\purx}[1]{\textcolor{purple}{#1}}
\newcommand{\grex}[1]{\textcolor{gray}{#1}}

\newcommand{\ebrace}[1]{\langle #1\rangle}
\newcommand{\abs}[1]{\lvert #1\rvert}

\newcommand{\bd}{\mathbf{d}}
 
\newcommand{\bg}{\mathbf{g}}

\newcommand{\bp}{\mathbf{p}}
\newcommand{\bq}{\mathbf{q}}

\newcommand{\bs}{\mathbf{s}}
\newcommand{\bu}{\mathbf{u}}
\newcommand{\bv}{\mathbf{v}}
\newcommand{\bw}{\mathbf{w}}
\newcommand{\bx}{\mathbf{x}}
\newcommand{\by}{\mathbf{y}}
\newcommand{\bz}{\mathbf{z}}

\newcommand{\hbx}{\widehat{\bx}}
\newcommand{\hby}{\widehat{\by}}

\newcommand{\bbx}{\widebar{\bx}}
\newcommand{\bby}{\widebar{\by}}
\newcommand{\bbz}{\widebar{\bz}}

\newcommand{\bbbx}{\widebarbar{\bx}}
\newcommand{\bbby}{\widebarbar{\by}}

\newcommand{\tbp}{\tilde{\bp}}
\newcommand{\tbq}{\tilde{\bq}}
\newcommand{\tbs}{\tilde{\bs}}

\newcommand{\tH}{\widetilde{\text{h}}}

\newcommand{\cH}{\mathcal{H}}
\newcommand{\cK}{\mathcal{K}}
\newcommand{\cL}{\mathcal{L}}
\newcommand{\cC}{\mathcal{C}}

\newcommand{\cO}{\mathcal{O}}
\newcommand{\cP}{\mathcal{P}}
\newcommand{\cQ}{\mathcal{Q}}
\newcommand{\cR}{\mathcal{R}}

\newcommand{\cT}{\mathcal{T}}

\newcommand{\CC}{\mathbb{C}}

\newcommand{\NN}{\mathbb{N}}
\newcommand{\ZZ}{\mathbb{Z}}
\newcommand{\II}{\mathbbm{1}}

\newcommand{\alp}{\alpha}     
\newcommand{\bet}{\beta}
\newcommand{\gam}{\gamma}     
\newcommand{\del}{\delta}
\newcommand{\eps}{\varepsilon}
\newcommand{\kap}{\kappa}
\newcommand{\lam}{\lambda}    
\newcommand{\vph}{\varphi}
\newcommand{\sig}{\sigma}
\newcommand{\Del}{\Delta}

\newcommand{\Sig}{\Sigma}     
\newcommand{\bSig}{\mathbf{\Sig}}
\newcommand{\ome}{\omega}

\newcommand{\df}{\colon}      
\newcommand{\ra}{\rightarrow}
\newcommand{\Iff}{\Longleftrightarrow}

\newcommand{\tikzcircle}[2][red,fill=red]{\tikz[baseline=-0.5ex]\draw[#1,radius=#2, very thick] (0,0) circle ;}%
\newcommand{\PT}{\bullet}
\newcommand{\bPT}{\blux{\bullet}}
\newcommand{\pbPT}{\tikzcircle[purple, fill=blue]{2.5pt}}
\newcommand{\tbPT}{\tikzcircle[teal, fill=blue]{2.5pt}}
\newcommand{\obPT}{\tikzcircle[orange, fill=blue]{2.5pt}}

\newcommand{\rPT}{\redx{\bullet}}
\newcommand{\crPT}{\tikzcircle[cyan, fill=red]{2.5pt}}
\newcommand{\orPT}{\tikzcircle[orange, fill=red]{2.5pt}}
\newcommand{\prPT}{\tikzcircle[purple, fill=red]{2.5pt}}
\newcommand{\trPT}{\tikzcircle[teal, fill=red]{2.5pt}}
\newcommand{\cPT}{\tikzcircle[cyan, fill=white]{2.5pt}}
\newcommand{\tPT}{\tikzcircle[teal, fill=white]{2.5pt}}
\newcommand{\oPT}{\tikzcircle[orange, fill=white]{2.5pt}}
\newcommand{\pPT}{\tikzcircle[purple, fill=white]{2.5pt}}

\newcommand{\sfA}{\mathsf{A}}
\newcommand{\sfD}{\mathsf{D}}
\newcommand{\sftA}{\widetilde{\sfA}}
\newcommand{\sftD}{\widetilde{\sfD}}
\newcommand{\sfP}{\mathsf{P}}
\newcommand{\sfT}{\mathsf{T}}

\newcommand{\us}{\underline{s}}
\newcommand{\ut}{\underline{t}}

\newcommand{\tw}{{\tilde{w}}}

\newcommand{\tbw}{{\widetilde{\bw}}}
\newcommand{\tbx}{{\widetilde{\bx}}}
\newcommand{\tby}{{\widetilde{\by}}}
\newcommand{\tbz}{{\widetilde{\bz}}}

\newcommand{\tQ}{\widetilde{Q}}
\newcommand{\tcR}{\widetilde{\cR}}
\newcommand{\uQ}{{\underline{Q}}}
\newcommand{\ord}{{\mathrm{ord}}}
\newcommand{\spe}{{\mathrm{sp}}}
\newcommand{\Qsp}{{Q_1^\spe}}
\newcommand{\Qord}{{Q_1^\ord}}
\newcommand{\Qspv}{{Q_0^\spe}}
\newcommand{\Qordv}{{Q_0^\ord}}
\newcommand{\huQ}{{\widehat{\uQ}}}
\newcommand{\hQ}{\widehat{Q}}
\newcommand{\hQord}{\hQ_1^{\mathrm{ord}}}
\newcommand{\hQsp}{\hQ_1^{\mathrm{sp}}}
\newcommand{\tuQ}{\widetilde{\uQ}}
\newcommand{\uH}{\underline{H}}
\newcommand{\Hord}{{H_1^{\mathrm{ord}}}}
\newcommand{\Hsp}{{H_1^{\mathrm{sp}}}}
\newcommand{\Hred}{{H_0^{\mathrm{red}}}}
\newcommand{\Hora}{{H_0^{\mathrm{ora}}}}
\newcommand{\Hpur}{{H_0^{\mathrm{pur}}}}
\newcommand{\Hblu}{{H_0^{\mathrm{blu}}}}
\newcommand{\Hcya}{{H_0^{\mathrm{cya}}}}
\newcommand{\Htea}{{H_0^{\mathrm{tea}}}}

\newcommand{\St}{\operatorname{St}}
\newcommand{\Ba}{\operatorname{Ba}}
\newcommand{\pBa}{\operatorname{pBa}}
\newcommand{\Adm}{\operatorname{Adm}}
\newcommand{\AdmSt}{\operatorname{AdmSt}}
\newcommand{\AdmBa}{\operatorname{AdmBa}}
\newcommand{\MAdm}{{\Adm}^{(\Ka)}}
\newcommand{\type}{\operatorname{type}}
\newcommand{\val}{\operatorname{val}}
\newcommand{\wt}{\operatorname{wt}}
\newcommand{\Diag}{\operatorname{Diag}}

\newcommand{\Ka}{k}

\newcommand{\End}{\operatorname{End}}
\newcommand{\Hom}{\operatorname{Hom}}

\newcommand{\GL}{\operatorname{GL}}
\newcommand{\kar}{\operatorname{char}}
      
\newcommand{\id}{\operatorname{id}}

\newcommand{\Irr}{\operatorname{Irr}}
\newcommand{\tIrr}{\Irr^{\tau}}

\newcommand{\tIrrind}{\operatorname{Irr}_{\mathrm{ind}}^\tau}
\newcommand{\codim}{\operatorname{codim}}

\newcommand{\dimv}{\underline{\dim}}         
\newcommand{\lmd}{\operatorname{-mod}}
\newcommand{\ind}{\operatorname{-ind}}
\newcommand{\rep}{\operatorname{rep}}

\newcommand{\Pu}{\mathbb{P}}   
\newcommand{\Ma}{\mathbb{M}}   

\title[Homomorphisms for skewed-gentle algebras]{On Homomorphisms and Generically $\tau$-regular Components for Skewed-Gentle Algebras}

\author{Christof Geiß}
\address{Instituto de Matemáticas, UNAM, Ciudad de México, C.P. 04510, MEXICO} 
\email{christof.geiss@im.unam.mx}
\date{October 14,  2025}

\begin{document}
\begin{center}
\emph{Dedicated to the Memory of María del Rocío Patiño Maya, 1965 -- 2021}
\end{center}
\begin{abstract}
  Let $\Ka$ be an algebraically closed field with $\kar(\Ka)\neq 2$,
  and $A$ a skewed-gentle
  $\Ka$-algebra.   In this case,    Crawley-Boevey's description of
  the indecomposable $A$-modules becomes particularly easy. This allows us
  to provide an explicit basis for the homomorphisms between
  any two indecomposable representations in terms of the corresponding
  admissible words in the sense of Qiu and Zhou~\cite{QZ17}.
  Previously~\cite{Ge99}, such a basis was only available when no asymmetric band modules were involved.
  We also extend a relaxed   version of fringing and kisses
  from Brüstle et al.~\cite{BY20}
  to the setting of skewed-gentle algebras.  With this at hand, we obtain
  convenient  formulae for the E-invariant and g-vector for indecomposable
  $A$-modules,  similar to the known expressions for
  gentle algebras, see also~\cite{PPP21}. Note however, that we allow in
  our context also band-modules. As an application, we describe the
  indecomposable, generically $\tau$-regular irreducible components of the representation varieties of $A$ as well as the generic values of the
  E-invariant between them in terms of tagged admissible words.
\end{abstract}
\maketitle
\section{Introduction}
\subsection{Context}
W. Crawley-Boevey~\cite{CB89b} introduced clans and clannish algebras and gave
a parametrization of their indecomposable representations
over any field with at least three elements.  Later, V. Bondarenko~\cite{Bo91}
and B. Deng~\cite{Deng00} gave a description that works over any field. 
This description of the indecomposable representations,
in terms of symmetric resp.~asymmetric strings and bands, 
is in general notoriously more involved than the well-known
classification of modules for string algebras~\cite{BR87}.  Motivated by
the works~\cite{CB89a} and~\cite{Kr91} we gave in~\cite{Ge99}
a partial answer to the description of homomorphisms between representations
of clans in the description of Deng. This was good enough to describe their
Auslander-Reiten theory.
However, the description of homomorphisms between indecomposable modules,
which involve asymmetric bands in~\cite{Ge99} is incomplete, due to
combinatorial difficulties.

As pointed out by Crawley-Boevey~\cite{CB89b},  the indecomposable
representations of clannish algebras can be described easily in terms
of certain clans.
Skewed-gentle  algebras are a special case of clannish algebras.
Motivated by the properties of gentle algebras, they  were introduced
in~\cite{GP99} because their derived category can be described in terms
of the representations of a  clannish algebra. So, skewed-gentle algebras are
in particular derived tame.  However, unlike gentle algebras, the class
of skewed-gentle algebras is not closed under derived equivalence.

Let $\bSig=(\Sig,\Ma,\Pu)$ a marked surface in the sense of~\cite{FST08}
with marked points $\Ma\subset\partial\Sig$ and punctures
$\Pu\subset\Sig\setminus\partial\Sig$. It is straightforward to construct in case $\Ma\neq\emptyset$ a tagged triangulation
$T$ of signature $0$.  With the non-degenerate Labardini potential~\cite{LF09}
$W(T)$ for the quiver $Q(T)$ of $T$ the Jacobian algebra
$A(T):=\cP_\CC(Q(T), W(T))$ is easily identified as a skewed-gentle algebra. 
In modern language, Qiu and Zhu outlined in~\cite{QZ17} that the tagged arcs for $\bSig$ correspond naturally to the decorated, indecomposable
$\tau$-rigid  representations of $A(T)$, and the E-invariant between two
such modules is $0$ if and only if the (tagged) intersection number between
the corresponding arcs vanishes.
See also~\cite{AP17} and~\cite{GLF23} for more details.

The main motivation for the present work is to show in a series of
follow-up papers~\cite{GLFW23a}, \cite{GLFW23b}
in collaboration with Daniel Labardini-Fragoso and Jon Wilson, that the generic basis of the cluster algebra of such a marked surface $\bSig$ with
$\abs{\Ma}\geq 2$ can be identified with the bangle basis from~\cite{MSW13}.
In order to attack this problem,
we need as a starting point a convenient description of the indecomposable,
generically $\tau$-regular irreducible components of the representation
varieties of a skewed-gentle algebra, and we also need to understand how
those components add up to arbitrary $\tau$-regular irreducible components
(see~\cite{BobS25} and Section~\ref{ssec:tau-reg} for the history of the name).

In other words, we want for each skewed-gentle algebra a combinatorial
description of its component graph of generically $\tau$-regular components
in the sense of~\cite[Sec.~6]{CLS15}.
To this end,  we need to identify 
 \emph{all} indecomposable, generically $\tau$-regular irreducible
 components and calculate the generic value of the $E$-invariant between them.  Recall that 
 $E_A(X, Y):= \dim\Hom_A(X,\tau Y)+\dim\Hom_A(Y,\tau X)$.
Because of the above-mentioned incomplete
description of the homomorphisms between representations of skewed-gentle algebras in~\cite{Ge99}, some extra work is needed.
In the present manuscript we overcome those difficulties by describing the homomorphisms for a particularly nice special case of Crawley-Boevey's description of the indecomposable modules of skewed-gentle algebras, which works only over fields with characteristic different from 2.

We take here  the opportunity
to extend the notions of fringing and kisses for gentle algebras
from~\cite{BY20} to skewed-gentle algebras. This  allows us to calculate
the (generic) E-invariant between (irreducible components of) representations
of skewed-gentle algebras efficiently.  This description will
turn out to be convenient for the interpretation of the generic E-invariant
as a kind of intersection number.

\subsection{Overview} 
In Section~\ref{sec:PolQ} we introduce (admissible) polarized quivers and, as a special case, skewed-gentle polarized quivers.  The latter ones are the combinatorial avatars of the 
above-mentioned  skewed-gentle algebras. 
Skewed-gentle polarized quivers contain a bit of redundant information if one
just wants to recover skewed-gentle algebras. However, the involved choices are convenient for our combinatorial study and can be traced back at least to the work of Butler and Ringel~\cite{BR87}. 
In the case of the Jacobian algebras coming from a tagged triangulation $T=(\tau_1,\ldots,\tau_n)$ of signature $0$,
the corresponding skewed-polarized quiver $\uQ(T)$ retains the information about the orientation of each of the arcs $\tau_i$ for $i=1,2,\ldots, n$.  

The more general admissible polarized
quivers turn out to be convenient for the description of the indecomposable
representation of skewed-gentle algebras over fields with $\kar{\Ka}\neq 2$. 
Section~\ref{sec:PolQ} serves mainly to fix our notation and to introduce asymmetric resp. symmetric strings and bands for skewed-gentle polarized quivers. 
Moreover, we introduce the combinatorial aspects of the
Auslander-Reiten translation for representations of skewed-gentle algebras.

In Section~\ref{sec:adm} we study \emph{admissible words} for a 
skewed-gentle polarized quiver $\uQ$.  These are a class of strings and bands
for an auxiliary gentle polarized quiver $\huQ$, where all special loops are replaced by ordinary loops, which square to zero in the gentle
algebra $\Ka\uQ$.  However, the orientation of the letters for $\huQ$, corresponding to those loops,
has to fulfill a certain consistency rule, which involves possibly the
orientation of other letters of this kind.  Our Proposition~\ref{prp:Adm}
states that this orientation problem can be solved uniquely by the rule that was used, for example, by Crawley-Boevey~\cite[Sec.~3]{CB89b} in his description of the indecomposable  representations of clannish algebras.

While  finishing the present manuscript, we became aware of
U.~Hansper's PhD thesis~\cite{Ha22}.
She showed, in Section~3.3 of her work, with a  different terminology, our Proposition~\ref{prp:Adm}.  
In her terminology, the words of the form $A(\bw)$ with $\bw\in\St(\uQ)\cup\pBa'(\uQ)$ have a weakly consistent
orientation, while the elements of $\Adm(\uQ)$ have a consistent orientation.
We include, anyway, our (shorter) proof of this result in
Appendix~\ref{sec:PfAdm}.

Qiu and Zhou introduced admissible strings in~\cite[Sec.~2.5]{QZ17}.
They use them, without any further comment, in order to parametrize a large class of (indecomposable) representations for skewed-gentle algebras.
Qiu and Zhou observed (without proof) in~\emph{loc.~\!cit.\!} that admissible
strings are very convenient to describe the homotopy classes of
curves without kinks on punctured surfaces with $\Ma\neq\emptyset$.
Apparently,  Qiu and Zhou silently assumed the validity of our Proposition~\ref{prp:Adm}.
We believe that this is a quite non-trivial fact.

In Section~\ref{sec:Quiv} we introduce, following
Crawley-Boevey~\cite[Sec.~3]{CB89b}, a quiver $H(\bx)$ for each admissible
word $\bx$  of a skewed-gentle quiver. The  quivers $H(\bx)$
should be seen as blueprints for the indecomposable modules over the corresponding skewed-gentle algebra.  
We also introduce, inspired by~\cite{Ge99}, a decorated quiver $H_\uQ(\bx,\by)$ for each pair of admissible words $\bx$ and $\bx$. 
We will see, that from the quiver $H_\uQ(\bx,\by)$ it is
possible to read off a basis of the homomorphisms between the representations corresponding to $\bx$ and $\by$.  
More precisely,  we showed in~\cite{Ge99} that in many situations the \emph{real h-lines}  give rise to such a basis.
Our first main result, Proposition~\ref{prp:bij-hl}, states that real 
h-lines are in bijection with \emph{long h-lines}.  
The point is that real h-lines are, for example, convenient to describe intersections between curves on punctured
surfaces, as observed  by Qiu and Zhou in~\cite{QZ17}, and they seem to be the natural objects for the study of the homomorphisms between representations of clans
(or skewed-gentle algebras) over an arbitrary field. 
On the other hand,  long h-lines \emph{do not} depend on the orientation of the special letters.
For this reason they are much easier to deal with.
We will see in Theorem~\ref{thm:hom} that they parametrize a basis for the space of homomorphisms between \emph{any} two indecomposable 
representations of skewed-gentle algebras over a field
with characteristic different from 2.
The notion of fringing and kisses, 
which was introduced in~\cite{BY20} in order
to calculate in a convenient way the $E$-invariant
$\dim\Hom_A(M,\tau N)+ \dim\Hom_A(M,\tau N)$
between string modules for gentle algebras, see also~\cite{PPP21}.
We extend this combinatorial
device  to strings and bands for skewed-gentle polarized quivers;
see Section~\ref{ssec:kissH}. 
We use the opportunity to relax by  the conditions on a fringing, which will be convenient for the study of skewed-gentle polarized quivers, which come from admissible
triangulations of punctured surfaces.

In Section~\ref{sec:mod} we explain how to obtain  a skewed-gentle algebra $\Ka\uQ$ from a skewed-gentle polarized quiver $\uQ$ and a field $\Ka$ with $\kar(\Ka)\neq 2$. 
We also describe the algebra $\Ka\uQ$ in terms of its
 Gabriel quiver $\tQ$ and an admissible ideal.
Then   we apply our constructions to the module category
of the skewed-gentle algebra $\Ka\uQ$. 
Under the hypothesis $\kar(\Ka)\neq 2$, Crawley-Boevey's description of the indecomposable $\Ka\uQ$-modules  
becomes quite easy; see Theorem~\ref{Thm:classif}.
In particular, one does not have to deal with the orientation of the special letters.
This allows us to describe the homomorphisms between the indecomposable representations in terms of the above-mentioned long resp.~real h-lines;
see Theorem~\ref{thm:hom}. With this at hand, we 
deduce in Section~\ref{ssec:Einv} a
convenient formula for the $E$-invariant between indecomposable
representations of skewed-gentle algebras.
The same theorem, together with our theory of fringing,
also allows us to deduce in Section~\ref{ssec:gvec} a formula for the g-vector of any indecomposable
module over a skewed-gentle algebra, which is inspired by the shear
coordinates.

Finally,  in Section~\ref{sec:gentred}, let $\Ka$ an algebraically closed
field with $\kar(\Ka)\neq 2$ and let $\uQ$ be a skewed-gentle polarized quiver. 
We give a combinatorial description of the set $\tIrrind{\Ka\uQ}$ of
indecomposable, generically $\tau$-regular irreducible components for a
skewed-gentle algebra $\Ka\uQ$.  To this end, we introduce the set of
\emph{tagged} admissible words $\Adm^*(\uQ)$,  and construct for each
$(\bx,s)\in\Adm^*(\uQ)$ a natural closed subset
$\widebar{M}^o_{(\bx,s)}\subset\rep_{\Ka\uQ}^{\bd(\bx,s)}(\Ka)$
of the corresponding
representation variety of $\Ka\uQ$.  Each $\widebar{M}^o_{(\bx,s)}$ contains
a dense subset of indecomposable representations.  Moreover, we introduce functions
\[
  e_\uQ\df \Adm^*(\uQ)\times\Adm^*(\uQ)\ra\NN \text{ and }
  \bg_\uQ\df \Adm^*(\uQ)\ra\ZZ^{\tQ_0},
\]
combinatorially defined in terms of a fringing $\uQ^f$ and our
notion of kisses.  Our main result, Theorem~\ref{thm:tirrind}, states that
$e_\uQ((\bx,s), (\by,t))$ gives the generic value of the $E$-invariant
$E_{\Ka\uQ}$ on $\widebar{M}^o_{(\bx,s)}\times\widebar{M}^o_{(\by,t)}$, whilst
$\bg_\uQ(\bx,s)$ gives the generic value of the g-vector $\bg_{\Ka\uQ}$
on $\widebar{M}^o_{(\bx,s)}$.  Most importantly, with
\[
\Adm^*_\tau(\uQ):=\{(\bx,s)\in\Adm^*(\uQ)\mid e_\uQ((\bx,s), (\bx,s))=0\},
\]
our construction $\widebar{M}^o_?$ induces a bijection
\[
  ([\Adm^*_\tau(\uQ)]/\simeq)\ \ra\tIrrind(\Ka\uQ),\quad
  [(\bx,s)]\mapsto\widebar{M}^o_{(\bx,s)}
\]
for an adequate equivalence relation on $\Adm^*(\uQ)$.  

In order to fix notation, and for the convenience of the reader, we discuss in Appendix~\ref{ssec:infdih} in detail the indecomposable representations of the infinite dihedral group over an algebraically closed field $\Ka$ with $\kar(K)\neq 2$.  
This is, in our context, fundamental for the description of the indecomposable representations attached to symmetric bands.

\section{Polarized Quivers and their combinatorics} \label{sec:PolQ}
\subsection{Polarized quivers} \label{ssec:polar}
We introduce polarized quivers as a tool to study the combinatorics of representations of skewed-gentle algebras.

\begin{Def}
A \emph{polarized quiver} is a tuple
$\uQ=(Q_0, \Qsp\coprod \Qord,\us,\ut)$,
where $Q_0$ is the (finite) set of vertices, $Q_1=\Qsp\cup \Qord$ is
the set of arrows, which contains a (possibly empty) subset of ``special''
arrows $\Qsp$,  and 
  \[
    \us, \ut\df Q_1\ra Q_0\times\{-1,1\}
  \]
are a pair of injective maps, which we write as
$\us=(s,s_1)$ and $\ut=(t,t_1)$.
Moreover, for each $\eps\in\Qsp$ there exists, by definition,  a unique $\eps'\in\Qsp$ with $\us(\eps')=\ut(\eps)$ and $\ut(\eps')=\us(\eps)$, 
and we request   $s_1(\eps)=-1=t_1(\eps)$.  
  
An \emph{admissible path} for $\uQ$ is a sequence of arrows
$\alp_1\alp_2\cdots\alp_l$ such that $\us(\alp_i)=-\ut(\alp_{i+1})$
for $i=1,2,\ldots, l-1$, where $-\ut(\alp):=(t(\alp),-t_1(\alp))$.
$\uQ$ is \emph{admissible} if it has only finitely many admissible paths.
$\uQ$ is \emph{skewed-gentle} resp.~\emph{gentle} if moreover $\Qsp$ is contained in the set of loops of $\uQ$, resp.~$\Qsp=\emptyset$. 
A \emph{homomorphism of polarized quivers} $F\df\uQ\ra\uQ'$ consists of a pair of maps $F_0\df Q_0\ra Q'_0$ and $F_1\df Q_1\ra Q'_1$ such that  for each arrow $\alp\in Q_1$ we have
\[
  \us'(F_1(\alp))=(F_0(s(\alp)),s_1(\alp)),\quad
   \ut'(F_1(\alp))=(F_0(t(\alp)),t_1(\alp)) \text{ and }
   F(\Qsp)\subset (Q')^{\mathrm{sp}}.
 \]
If moreover $F(\Qord)\subset (Q'_1)^\ord$, we say that $F$ is \emph{strict}.
\end{Def}
We will see later, that the polarizations are a (somehow redundant) way to codify the relations of a skewed-gentle algebra.

\begin{Rem} \label{rem:polarized}
We observe that for a polarized quiver $\uQ$, in the underlying quiver
$Q:=(Q_0,Q_1,s,t)$ at each given vertex, there start at most two arrows and there end at most two arrows. Let $\eps\in\Qsp$ and $\eps'$ the corresponding
``dual'' special arrow, then $(\eps')'=\eps$. 
Moreover, we note that if a special arrow $\eps$ is a
loop, then $\eps=\eps'$. If $\phi\df\uQ\ra\uQ'$ is a homomorphism of
polarized quivers, 
and $\alp,\bet\in Q_1$ are two different arrows with
$s(\alp)=s(\bet)$ or with $t(\alp)=t(\bet)$, then $\phi(\alp)\neq\phi(\bet)$.
Thus, homomorphisms between polarized quivers automatically fulfill the property (W1) of a winding from~\cite{Kr91}.
\end{Rem}

If $\uQ$ is skewed-gentle polarized quiver, there is at each vertex 
$i\in Q_0$ at most one special loop $\eps_i$.  It will be sometimes convenient to set
\[
  \Qspv:=\{i\in Q_0\mid \text{ there is a special loops at } i\}
  \quad\text{and}\quad \Qordv:=Q_0\setminus\Qspv.
\]

\subsection{Letters} \label{ssec:let}
Associated to a polarized quiver $\uQ$ we have a set $\cL=\cL(\uQ)$ of letters, which come in 5 types.
We describe them in the  Table~\ref{tab:letters}.
In particular, we extend the functions $\us$ and $\ut$ to functions
$\cL(\uQ)\ra (Q_0\times\{-1, 1\}) \cup\{\ast\}$ with  a new  “virtual vertex” $\ast$.
 
\begin{table}[h]
\begin{tabular}{|m{2cm} |l |m{4cm}|c|c|l|}\hline
  type     & symbol & condition       & $\us$      & $\ut$      & inverse    \\ \hline\hline
  ordinary & $\alp$ & $\alp\in\Qord$& $\us(\alp)$& $\ut(\alp)$& $\alp^{-1}$ \\ \hline
inverse &$\alp^{-1}$ & $\alp\in\Qord$& $\ut(\alp)$& $\us(\alp)$& $\alp$     \\ \hline
special &$\eps^*$   & $\eps\in\Qsp$ & $(s(\eps),-1)$ & $(t(\eps),-1)$& $\eps'^*$  \\ \hline
trivial &$\II_{i,\rho}$ & $(i,\rho)\in Q_0\times\{-1,1\}\setminus\{\us(\eps)\mid\eps\in\Qsp\}$ &
* &$(i,\rho)$  & $\II_{i,\rho}^{-1}$\\ \hline                                                         trivial inverse &$\II_{i,\rho}^{-1}$ & $(i,\rho)\in Q_0\times\{-1,1\}\setminus\{\us(\eps)\mid\eps\in\Qsp\}$ &$(i,\rho)$  & *& $ \II_{i,\rho}$\\\hline
\end{tabular}\vspace*{2ex}
\caption{Letters for a polarized quiver}
\label{tab:letters}
\end{table}
Occasionally we will write, for typographic reasons, 
$\II^{\pm 1}_{i,\pm}$
in place of the more formal $\II^{\pm 1}_{i,\pm 1}$.
We note that for each $(i,\rho)\in Q_0\times\{-1,1\}$ the set
\[
  \cL_{(i,\rho)}(\uQ):=\{x\in\cL(\uQ)\mid \ut(x)=(i,\rho)\}
\]
consists either of a single special letter, or else it is, by definition,
a linearly ordered set of  one of the following forms:
\[
\{\II_{i,\rho}\}, \{\alp<\II_{i,\rho}\},\{\II_{i,\rho}<\bet^{-1}\}, \{\alp<\II_{i,\rho}<\bet^{-1}\}
\]
for certain ordinary arrows $\alp,\bet\in\Qord$.  
This yields a partial order $<$ on $\cL$, 
where two letters $w_1\neq w_2$ are comparable if and only if 
$\ut(w_1)=\ut(w_2)\in Q_0\times\{-1,1\}$.

In particular, if $\uQ$ is a skewed-gentle polarized quiver and $i\in\Qspv$,
then there is a unique special letter $\eps_i^*$ with
$\us(\eps_i^*)=(i,-1)=\ut(\eps_i^*)$,
and there is only a trivial letter $\II_{i,+1}$ but \emph{no} letter
$\II_{i,-1}$.

\subsection{Words}
A \emph{word} for $\uQ$ is a sequence of letters 
$\bw=w_1w_2\cdots w_m$ with the $w_i\in\cL_\uQ$ such that
$\us(w_i)=-\ut(w_{i+1})\in Q_0\times\{-1,1\}$ for $i=1,2,\ldots,m-1$. In this case, $w_1$ is called the \emph{first letter}, and $w_m$  is called the \emph{last letter} of $\bw$. 
Moreover, we set
\begin{align*}
\bw_{[k]} &:=w_1w_2\ldots w_{k}, &  \us(\bw)  &:=\us(w_m),\\
\bw^{[k]} &:=w_kw_{k+1}\cdots w_m,& \ut(\bw)  &:=\ut(w_1),\\
\bw^{-1}  &:=w_m^{-1}w_{m-1}^{-1}\cdots w_1^{-1},& l(\bw)    &:=m.  
\end{align*}
\
If $\bv=v_1v_2\cdots v_l$ is another word with 
$\ut(\bv)=-\us(\bw)$, we can form the
concatenated word $\bw\bv=w_1\cdots w_mv_1\cdots v_l$. Moreover, $\bw$ is \emph{right inextensible} if $w_m=\II_{i,\rho}$ for some
$(i,\rho)\in (Q_0\times\{-1,1\})\setminus\{\us(\eps)\mid \eps\in\Qsp\}$,
or equivalently if $\us(w_m)=*=\us(\bw)$. Finally, $\bw$ is \emph{left inextensible} if $\bw^{-1}$ is right inextensible.  
Note, that the above defined admissible paths are words in our sense. 
In view of the construction of the
skewed-gentle algebra $\Ka\uQ$ (see Section~\ref{ssec:skga}) associated to a skewed-gentle polarized quiver, the (classes of the) admissible paths together with the trivial letters form a $\Ka$-basis of $\Ka\uQ$.

\subsubsection*{Example} For each $(i,\rho)\in (Q_0\times\{-1, 1\}) \setminus\{\us(\eps)\mid\eps\in\Qsp\}$ we have the simple string
\[
\tbs_{(i,\rho)}:= \begin{cases}
\II_{(i,\rho)}^{-1}\II_{(i,\rho)} &\text{ if } i\in\Qordv,\\
\II_{(i,1)}^{-1}\eps_i^*\II_{(i,1)} &\text{ if } i\in\Qspv,
\end{cases}
\quad\text{with}\quad
\tbs_{(i,\rho)}^{-1}=\begin{cases}
    \tbs_{(i,-\rho)} &\text{ if } i\in\Qordv,\\
    \tbs_{(i,\rho)}  &\text{ if } i\in\Qspv.
\end{cases}
\]
In contrast, $\II_{(i,\rho)}\II_{(i,\rho')}^{-1}$ is \emph{not} even a word.

\subsection{Lexicographic order} \label{ssec:lex}
The lexicographic order on all words, which is induced by the
partial order on the letters, gives a poset with poor properties.
For example, if $\bv=\bv'\bv''$ is a non-trivial factorization of a word $\bv$, then $\bv'$ and $\bv$ are \emph{not} comparable. 

Note, however, that two right-inextensible words with comparable first letters are comparable under this lexicographical order.
The same is true for two words of the same length: they are comparable if their respective first letters are comparable. 
In fact, if two words are comparable, say $\bv<\bw$,  then there exists a unique word $\bv'$ and comparable letters
$v<w$ such that $\bv=\bv' v\bv''$ and $\bw=\bv'w\bw''$.  
In this case we write $\Del(\bv,\bw)=l(\bv')$.  
In particular, $v$ is an ordinary letter, or $w$ is an inverse ordinary letter. 
In the first case, $w$ can be trivial or
inverse, whilst in the second case, $v$ can be direct or trivial. Note that either of the words $\bv''$ and $\bw''$ may be empty. 

Similarly, if we have factorizations $\bv=\bv'\bv''$ and $\bw=\bw'\bw''$, then $\bv'<\bw'$ implies $\bv<\bw$.

\subsection{Strings and Bands}\label{ssec:StBa}
A \emph{string} for $\uQ$ is a word which is left- and right inextensible.
Thus, the first letter of a string is always a trivial inverse letter.
We denote the set of strings for $\uQ$ by $\St(\uQ)$, and $\St_{(i,\rho)}(\uQ)$
denotes the subset of strings that have as their first letter $\II_{i,\rho}^{-1}$.
This set is linearly ordered by the lexicographic order, see
Section~\ref{ssec:lex}.  Note, that 
$\St(\uQ)=\cup_{(i,\rho)\in Q_0^+}\St_{(i,\rho)}(\uQ)$, where 
$Q_0^+:= (Q_0\times\{-1,1\})\setminus\{\us(\eps)\mid \eps\in\Qsp\}$. 
A string $\bw$ is \emph{symmetric} if $\bw=\bw^{-1}$.
It is easy to see that in this case
there exists a unique left inextensible word $\bv$ and a special letter
$\eps^{*}=(\eps^*)^{-1}$
such that $\bw=\bv\eps^*\bv^{-1}$.

A \emph{band} is a word $\bw$ such that $\bw\bw$ is also a word.
A band is \emph{primitive} if there is no  word $\bv$ such that
$\bw=\bv^{n}$ for some $n\geq 2$.
If $\bw$ is a (primitive) band, then so are all its
\emph{rotations} $\bw^{[k+1]}\bw_{[k]}$.
A primitive band is \emph{symmetric} if $\bw^{-1}=\bw^{[k+1]}\bw_{[k]}$
for some $k$.
This is the case if and only if some rotation of $\bw$ is of the form 
$\eps^*\bv\zeta^*\bv^{-1}$ for some word $\bv$ and
special letters $\eps^*$ and $\zeta^*$. These remarks and definitions go back to~\cite[Sec.~3]{CB89b}, see~\cite[Lemma~2.14]{BTCB24} for proofs.
Note, that a special letter $\eps^*$ alone is not a symmetric band, since $\eps^*\eps^*$ is not a word.

We denote by $\pBa(\uQ)$ resp.~$\Ba(\uQ)$ the set of (primitive) bands for $\uQ$.
Two strings are \emph{equivalent}, in symbols $\bv\sim\bw$ if $\bv\in\{\bw,\bw^{-1}\}$.  
Two bands are \emph{equivalent}, in symbols $\bv\sim\bw$ if
$\bv^\rho=\bw^{[k+1]}\bw_{[k]}$ for some 
$\rho\in\{-1,+1\}$ and $k=0,1,\ldots,l-1$.
We denote by $\pBa'(\uQ)$ the set of primitive bands in \emph{standard form}.
By definition, $\pBa'(\uQ)\subset\pBa(\uQ)$ consists of
all asymmetric primitive bands, and of the symmetric (primitive) bands, which are of the form $\eps^*\bv\zeta^*\bv^{-1}$, as above.

Sometimes we will also consider the equivalence classes of (primitive) bands modulo rotations, which we will denote by $[\pBa(\uQ)]$ resp. $[\Ba(\uQ)]$.

\subsection{Projective strings, injective strings, and the
  operator  \texorpdfstring{$\tau_\uQ$}{τ}} \label{ssec:proj}
Let $\uQ$ be a skewed-gentle polarized quiver.  Then, for each
$(i,\rho)\in Q_0\times\{-1,1\}\setminus\{\us(\eps^*)\mid\eps\in\Qsp\}$,
the linearly ordered set $\St_{(i,\rho)}(\uQ)$ has a (unique)
maximal element $\II_{i,\rho}^{-1}\bw^{\max}_{i,\rho}$ and a
minimal element $\II_{i,\rho}^{-1}\bw^{\min}_{i,\rho}$.
For example, $\bw_{i,\rho}^{\min}$ is the unique right inextensible word of maximal length, which contains no inverse letter, and with
$\ut(\bw_{i,\rho}^{\min})=(i,-\rho)$. 
Note, that for $i\in\Qspv$ we have 
$\bw_{i,1}^{\max}=\eps_i^*{\bw'}_{i,1}^{\max}$ and 
$\bw_{i,1}^{\min}=\eps_i^*{\bw'}_{i,1}^{\min}$ for certain right inextensible
words ${\bw'}_{i,1}^{\max}$ and ${\bw'}_{i,1}^{\min}$.

This allows to define the following projective, resp.~injective strings:
 \begin{alignat*}{3}
  \tbp_{i,\rho} &:= (\bw_{i,-\rho}^{\max})^{-1}\bw_{i,\rho}^{\max},
                 &&\text{ for } (i,\rho)\in \Qordv\times \{-1,1\}, \\
   \tbp_{i} &:=({\bw'}_{i,1}^{\max})^{-1}\eps_i^*{\bw'}_{i,1}^{\max} &&\text{ for } i\in\Qspv,\\ 
   \tbq_{i,\rho} &:= (\bw_{i,-\rho}^{\min})^{-1}\bw_{i,\rho}^{\min}
   &&\text{ for } (i,\rho)\in \Qordv\times \{-1,1\}, \\
   \tbq_{i} &:=({\bw'}_{i,1}^{\min})^{-1}\eps_i^*{\bw'}_{i,1}^{\min} &&\text{ for } i\in\Qspv. 
 \end{alignat*}
 We note that $\tbp_{i,\rho}^{-1}=\tbp_{i,-\rho}$ and $\tbq_{i,\rho}^{-1}=\tbq_{i,\rho}$,
 whilst $\tbp_i$ and $\tbq_i$ for $i\in\Qspv$ are symmetric strings.

Now, let $\bw=\II_{i,\rho}^{-1}\bw'\II_{j,\sig}$  a string, where, by definition,
$(j,\sig)=-\us(\bw')$.
If there exists an inverse letter $\alp^{-1}$ with $\ut(\alp^{-1})=(j,\sig)$, the
string defined by \emph{adding a co-hook on the right} to $\bw$ is
$\II_{i,\rho}\bw'\alp^{-1}\bw^{\min}_{-\us(\alp^{-1})}$.  
In this case it is the immediate successor of $\bw$ in $\St_{i,\rho}(\uQ)$ . 
If there is a direct letter $\bet$ with 
$\ut(\bet)=(j,\sig)$, the string defined by \emph{adding a hook to the right} to $\bw$ is  $\II_{i,\rho}^{-1}\bw'\bet\bw^{\max}_{-\us(\bet)}$.  
In this case it is the immediate predecessor of $\bw$.
Furthermore, if one cannot add a co-hook on the right to 
$\bw\neq \II_{i,\rho}^{-1}\bw^{\max}_{i,\rho}$, then there exists a minimal $k$ with
$1<k\leq l(\bw)$ such that $\bw^{[k]}=\bw^{\max}_{h,\lam}$ with 
$(h,\lam)=\ut(\bw^{[k]})$. In this case $\bw$ is obtained by adding to the string 
$\bw_{[k-1]}\II_{-\us(\bw_{[k-1]})}$ a hook to the right. So, in this case 
$\bw_{[k-1]}\II_{-\us(\bw_{[k-1]})}$ is the immediate successor of $\bw$. Altogether, the above describes how to compute the immediate successor $\bw[1]$
of $\bw$, if $\bw'\II_{j,\rho}\neq\bw^{\max}_{i,\rho}$.

This allows us to define a bijection
 \begin{multline*}
   \tau_\uQ\df \St(\uQ)\setminus\cP(\uQ)\ra \St(\uQ)\setminus\cQ(\uQ),\\
   \bw\mapsto
   \begin{cases}
     [1](\bw[1]) & \text{if } \bw\phantom{^{-1}} \neq \II^{-1}_{i,\rho}\bw_{i,\rho}^{\max} \text{ for some }
     (i,\rho)\in Q_0\times \{-1,+1\},\\
     ([1]\bw)[1] & \text{if } \bw^{-1}\neq\II^{-1}_{i,\rho}\bw_{i,\rho}^{\max}
     \text{ for some } (i,\rho)\in Q_0\times \{-1,+1\}.
   \end{cases}  
 \end{multline*}
where we set
\begin{align*}
 \cP(\uQ) &:=
 \{\tbp_{i,\rho}\mid (i,\rho)\in\Qordv\times\{-1,+1\}\}\cup\{\tbp_i\mid i\in\Qspv\},\\
 \cQ(\uQ) &:=
 \{\tbq_{i,\rho}\mid (i,\rho)\in\Qordv\times\{-1,+1\}\}\cup\{\tbq_i\mid i\in\Qspv\},\\
 [1]\bw& :=(\bw^{-1}[1])^{-1}.
\end{align*}
It is not hard to see that this is a well-defined bijection. In particular,
$([1]\bw)[1]=[1](\bw[1])$ if neither $\bw$ nor $\bw^{-1}$ is of the form
$\bw_{i,\rho}^{\max}$ for some $(i,\rho)\in Q_0\times\{-1,+1\}$.
We agree that for each band $\bw\in\pBa(\uQ)$ we have $\tau_\uQ(\bw)=\bw$.

Our constructions and terminology here are inspired by the classical constructions of Butler and Ringel~\cite[pp. 166--172]{BR87}. 
It was adapted in~\cite[Sec.~5]{Ge99} to the representation theory of finite dimensional clans.

\subsection{Example} \label{ssec:ExSt}
In the following display, we show a skewed-gentle polarized quiver $\uQ$ with
$\Qsp=\{\eps\}$ by indicating for each non-special arrow the values of the polarization close to the respective and of the arrow.  
Moreover, we display, as an example, the elements of the linearly ordered set $\St_{1,-1}(\uQ)$ in descending order.
\[
\uQ=\quad\vcenter{\xymatrix{&\ar[ld]^(.45)\bet_(.8)-_(.25)+ 2\ar@(ul,ur)[]_{\eps}^(.2)-^(.8)- \\
    3\ar[rr]^\gam_(.3)+_(.7)+ 
    && 1\ar[lu]^\alp_(.2)+_(.75)+\\
  }}\qquad
\St_{1,-1}(\uQ):\quad
  \setlength\arraycolsep{-.5pt}
  \begin{matrix}
\tbp_{1,-1}&=&    \II^{-1}_{1,-} &\alp^{-1} &\eps^* &\bet^{-1} &\gam^{-1} & \II_{1,-}\\
\tbq_{3,-1}&=&    \II^{-1}_{1,-} &\alp^{-1} &\eps^* &\bet^{-1} & \II_{3,+}\!\!\\
&&    \II^{-1}_{1,-} &\alp^{-1} &\eps^* & \II_{2,+}\!\!\\
\tbq_2&=&    \II^{-1}_{1,-} &\alp^{-1} &\eps^* &\alp    & \II_{1,-}\!\!\\
\tbs_1&=&    \II^{-1}_{1,-} &\II_{1,+}\!\!\!\!\\
\tbp_{3,+1}&=&    \II^{-1}_{1,-} &\gam & \II_{3,-}\!\!\!\!\\
\tbp_2&=&    \II^{-1}_{1,-} &\gam & \bet & \eps^*& \bet^{-1}& \gam^{-1} &\II_{1,-}\\
&&    \II^{-1}_{1,-} &\gam & \bet & \eps^*& \bet^{-1}&  \II_{3,+}\!\!\\
&&   \II^{-1}_{1,-} &\gam & \bet & \eps^*& \II_{2,+}\!\!\\
\tbp_{1,+1}&=&    \II^{-1}_{1,-} &\gam & \bet & \eps^*& \alp &\II_{1,-}\!\!\\
  \end{matrix}  
\]  
Note, that $\St_{1,-1}(\uQ)$ contains precisely two symmetric
strings, and that $\Ba(\uQ)=\emptyset$.

\section{Admissible words for \texorpdfstring{$\uQ$}{Q}} \label{sec:adm}
\subsection{Orders within strings and bands}\label{ssec:OrdStBa}
Let $\uQ$ be a polarized quiver for which all special letters
are loops. 
For  a string $\bw=w_0w_1\cdots w_l\in\St(\uQ)$  with $w_k=\eps_j^*$ a special letter, 
we have  $\ut(\bw_{[k-1]}^{-1})=\ut(\bw^{[k+1]})=(i,1)$.
Since, moreover, the words $\bw_{[k-1]}$ and $\bw^{[k+1]}$ are right inextensible, they are comparable in the induced lexicographical order.  
We write then
$\Del_\bw(k):=\Del(\bw_{[k-1]}^{-1},\bw^{[k+1]})$, 
see Section~\ref{ssec:lex}.
In this situation $\bw_{[k-1]}^{-1}=\bw^{[k+1]}$ occurs if and only if $\bw$ is a
symmetric string and $\bw=\bw_{[k-1]}\eps_j^*\bw_{[k-1]}^{-1}$.

Similarly, for a primitive band $\bw=w_0w_1\cdots w_l\in\pBa$  with
$w_k=\eps_j^*$ a special letter,  
$\bw_{[k-1]}^{-1}(\bw^{[k+1]})^{-1}=(\bw^{[k+1]}\bw_{[k-1]})^{-1}$ and
$\bw^{[k+1]}\bw_{[k+1]}$ are comparable because then
$\ut(\bw^{[k+1]}\bw_{[k-1]})=\ut(\bw_{[k-1]}^{-1}(\bw^{[k+1]})^{-1})=(i,1)$ and
both words have the same length $l$. We write then also
$\Del_\bw(k):=\Del(\bw^{[k+1]}\bw_{[k-1]}, \bw_{[k-1]}^{-1}(\bw^{[k+1]})^{-1})$.
Note  that in this situation
$\bw_{[k-1]}^{-1}(\bw^{[k+1]})^{-1}=\bw_{[k+1]}\bw^{[k-1]}$ implies that
$\bw_{[k+1]}\bw^{[k-1]}=\bx\eps_j^*\bx^{-1}$ for a unique word $\bx$ and a special
letter $\eps_j^*$.  Thus, $\bw$ is a symmetric band.

\subsection{\texorpdfstring{$\huQ$}{Q} and construction of
  \texorpdfstring{$A(\bw)$}{A(w)} for strings and bands}
\label{ssec:huQA}
Following~\cite[Sec.~3.2]{Ge99},
the \emph{associated gentle polarized quiver} $\huQ=(\hQ_0,\hQ_1,\us,\ut)$
of a skewed-gentle polarized quiver $\uQ$ 
is obtained  by taking $\hQ_0=Q_0$, 
$\hQord=\Qord\cup\Qsp$, $\hQsp=\emptyset$ and otherwise the same functions
$\us,\ut$. Thus, we  have a morphism
\[
F\df\huQ\ra\uQ
\]
of polarized quivers, which is bijective on vertices and arrows.
We will apply our constructions from
Section~\ref{ssec:let} to  $\huQ$. In particular, we now have a linearly ordered set 
\[   
\cL_{(i,-1)}(\huQ)=\{\eps_i\prec\II_{i,-1}\prec\eps_i^{-1}\} 
\] 
of letters for each special arrow $\eps_i\in\Qsp$.

We can extend $F$ to a surjective morphism  between the corresponding sets of  words, which is compatible with the concatenation of words.
In particular,   
$F^{-1}(\eps_i^*)=\cL_{(i,-1)}(\huQ)$ for each special letter $\eps_i^*$ and
$F^{-1}(l)=\{l\}$ for each letter $l$ that is not special.
Following~\cite[Sec.~2.3]{QZ17}, we call the trivial letters for $\huQ$ which
are of the form 
  $\II^{\pm 1}_{i,-1}$ with $i\in\Qspv$ \emph{punctured letters}.
We say that a string $\bx=x_0x_1\cdots x_l\in\St(\huQ)$ 
is of \emph{type} $(a,b)\in\{u,p\}^2$, with $a=p$ if and only if $x_0$ is
punctured, and $b=p$ if and only if $x_l$ is punctured.
Similarly, we say that the elements of $\pBa(\huQ)$ are of \emph{type} $(b)$.

Note that we use the symbol $\prec$ for the partial order on
the letters for $\huQ$ as well as on the induced lexicographic order on the
corresponding words, rather than the symbol $\leq$ which we reserve for
the strings and bands of $\uQ$.

We construct now, closely inspired by~\cite[Sec.~3]{CB89b}, for each
$\bw\in\St(\uQ)\cup\pBa'(\uQ)$ an element
$A(\bw)\in\St(\huQ)\cup\pBa(\huQ)$.

For $\bw=w_0w_1\cdots w_n$  an \textbf{asymmetric string},
$A(w):=a_0a_1\cdots a_n$ with $a_i=w_i$ whenever $w_i$ is not a special letter.
If however $w_i=\eps^*$ is a special letter, then
$a_i=\eps$ if $(\bw_{[i-1]})^{-1} >\bw^{[i+1]}$,  and $a_i=\eps^{-1}$ if
$(\bw_{[i-1]})^{-1} <\bw^{[i+1]}$.  

For $\bw=v_0v_1\cdots v_{m-1}\eps_b^* v_{m-1}^{-1}\cdots v_0^{-1}$  with
$b\in\Qspv$, a \textbf{symmetric string},
$A(\bw)=a_0a_1\cdots a_{m-1} \II_{b,-1}$ with $a_i=w_i$ if $w_i$ is not a special
letter. If, however $w_i=\eps^*$ is a special letter, then
$a_i=\eps$ if $(\bw_{[i-1]})^{-1} >\bw^{[i+1]}$,  and $a_i=\eps^{-1}$ if
$(\bw_{[i-1]})^{-1} <\bw^{[i+1]}$.

For $\bw=w_0w_1\cdots w_n$ an \textbf{asymmetric band},
$A(\bw)=a_0a_1\cdots a_n$ with $a_i=w_i$ if $w_i$ is not a special letter.
If, however $w_i=\eps^*$ is a special letter, then $a_i:=\eps$ if
$\bw_{[i-1]}^{-1}(\bw^{[i-1]})^{-1}>\bw^{[i+1]}\bw_{[i-1]}$ and $a_i:=\eps^{-1}$ 
if $\bw_{[i-1]}^{-1}(\bw^{[i-1]})^{-1}<\bw^{[i+1]}\bw_{[i-1]}$.

For $\bw=\eps^*_av_1v_2\cdots v_{n-1}\eps^*_bv_{n-1}^{-1}\cdots v_1^{-1}$ a
\textbf{symmetric band} in standard form,
$A(\bw)=\II^{-1}_{a,-1}a_1a_2\cdots a_{n-1}\II_{b,-1}$ with
$a_i=v_i$ if $v_i$ is not a special letter.  Else,  for $v_i=\eps^*$  special,
we set $a_i=\eps$ if $\bw_{[i-1]}^{-1}(\bw^{[i+1]})^{-1}>\bw^{[i+1]}\bw_{[i-1]}$ and
$a_i=\eps^{-1}$ if $\bw_{[i-1]}^{-1}(\bw^{[i+1]})^{-1}<\bw^{[i+1]}\bw_{[i-1]}$.

\subsubsection*{Remark added in proof} After this manuscript was completed, we became aware of the constructions in~\cite[Sec.~2 \& 3]{BTCB24}:  
Bennett-Tennenhaus and Crawley-Boevey use in place of our set
$\St(\uQ)\cup\pBa'(\uQ)$  their equivalent notion of 
\emph{end-admissible words}. In this setting, our bands correspond to certain (infinite) periodic words, whilst our strings correspond to certain finite words without trivial letters. Next, they associate to each end-admissible
word $\bw$ a \emph{canonically oriented walk} $C_\bw$ of the same length
as $\bw$.  This walk $C_\bw$ would be in our notation a (possibly infinite)
word for $\huQ$.  Following the discussion in Sections 3.1--3.4 of 
\emph{loc.\!~cit.}, our word $A(\bw)$ would be then be (up to possible trivial
letters) a subword of  $C_\bw$ corresponding to a subset $J_\bw$ of the
index set $I_\bw$ of the letters in the walk $C_\bw$.

\subsection{Admissible words vs.~strings and bands}
\label{ssec:adm}
Let $\uQ$ be a skewed-gentle polarized quiver, and $\huQ$ the associated
gentle polarized quiver from Section~\ref{ssec:huQA}. We recall
that $F\df\huQ\ra\uQ$ induces a surjective map between the respective
sets of words, which we also denote by $F$.  However, $F$ does not
restrict directly to a map between the respective sets of strings and
(primitive) bands. For example, for $i\in\Qspv$ we have
$\II_{i,-}^{-1}\II_{i,+}^{\phantom{-1}}\in\St(\huQ)$, however
$F(\II_{i,-}^{-1}\II_{i,+}^{\phantom{-1}})=\eps_i^*\II_{i,+}\not\in\St(\uQ)$. 

Thus, we define the completion
$\overline{?}\df \St(\huQ)\cup\Ba(\huQ)\ra\St(\uQ)\cup\Ba(\uQ)$ by
\begin{equation} \label{eq:cpl}
  \bx=x_0x_1\cdots x_l \mapsto\begin{cases}
    F(\bx) &\text{if } \bx\in\St(\huQ) \text{ is of type } (u,u),\\
    F(\bx)F(x_{l-1}^{-1}\cdots x_1^{-1} x_0^{-1})
    &\text{if } \bx\in\St(\huQ) \text{ is of type } (u,p),\\
    F(x_{l}^{-1}\cdots x_2^{-1}x_1^{-1})F(\bx)
    &\text{if } \bx\in\St(\huQ) \text{ is of type } (p,u),\\
    F(\bx) F(x_{l-1}^{-1}\cdots x_2^{-1}x_1^{-1}) 
    &\text{if } \bx\in\St(\huQ) \text{ is of type } (p,p),\\
    F(\bx) &\text{if } \bx\in\Ba(\huQ).
  \end{cases}
\end{equation}
In the first three cases $\widebar{\bx}\in\St(\uQ)$, whilst
$\widebar{\bx}\in \Ba(\uQ)$ in the last two cases. Note, that for a primitive
band $\bx\in\pBa(\huQ)$, possibly $\widebar{\bx}\in\Ba(\uQ)$ is not primitive.

Following Qiu and Zhou~\cite[Sec.~2.5]{QZ17} a string
$\bx=x_0x_1\cdots x_n$ for $\huQ$ is
\emph{admissible} (with respect to $F$) if it fulfills the following two
conditions:
\begin{itemize}
\item
  For any $k\in\{1,2,\ldots, n-1\}$ with  $F(x_k)\in\Qsp$ we have
  $\bx_{[k-1]}^{-1}\neq\bx^{[k]}$ and in this situation $x_k$ is 
  direct if and only if $\bx_{[k-1]}^{-1}\succ \bx_{[k+1]}$.
\item
  If $\bx$ is of type $(p,p)$ then $\widebar{\bx}$ is a primitive band.
\end{itemize}

Observe that a string $\bx$ is admissible if and only if
$\bx^{-1}$ is admissible.

Similarly, we say that a band $\bx=x_0x_1\cdots x_n$ for $\huQ$ is
\emph{admissible} (with respect to $F$), if the following two conditions
are fulfilled:
\begin{itemize}
  \item  For any $k\in\{1,2,\ldots, n\}$ with  $F(x_k)\in\Qsp$ we have
    $(\bx^{[k-1]})^{-1}\bx_{[k-1]}^{-1}\neq\bx^{[k+1]}\bx_{[k-1]}$,
      and in this situation $x_k$ is 
      direct if and only if
      $(\bx^{[k-1]})^{-1}\bx_{[k-1]}^{-1}\succ\bx^{[k+1]}\bx_{[k-1]}$,
    \item  $F(\bx)$ is a primitive asymmetric band for $\uQ$.    
\end{itemize}   

\begin{Prop} \label{prp:Adm}
Let $\uQ$ be a skewed-gentle polarized quiver and $F\df\huQ\ra\uQ$
the canonical morphism starting in the corresponding gentle polarized
quiver $\huQ$.  Then $\bx\in\St(\huQ)\cup\pBa(\huQ)$ is admissible
if and only if $\bx=A(\bw)^{\pm 1}$ for some $\bw\in\St(\uQ)\cup\pBa'(\uQ)$.

In particular, all the words $A(\bw)$ are admissible strings or bands, and we
have obviously $\widebar{A(\bw)}=\bw$.  
\end{Prop}

We will prove Proposition~\ref{prp:Adm} in Appendix~\ref{sec:PfAdm}.
As mentioned in the Introduction, a quite similar result, though in a different language, can be found in the recent PhD thesis of U.~Hansper~\cite[Sec.~3.3]{Ha22}.  More precisely, Theorem 3.61
in \emph{{loc.\,cit.}} corresponds to our Proposition~\ref{prp:Adm}.  
Our motivation for this result was the apparent  discrepancy between
Sections~2.5 and~2.6 of~\cite{QZ17} and the main result of~\cite{CB89b}.
We became only
aware of Hansper's Ph.D. Thesis after a first version of the present manuscript was completed. For this reason, we don't claim any credit for Proposition~\ref{prp:Adm}.

\subsubsection*{Remark} 
For $\bx\in\St(\huQ)$  it follows from the
definitions that $F(\bx)\in\St(\uQ)$ if and only if $F(\bx)=\widebar{\bx}$.
In this case $\bx$ must be of type $(u,u)$. If the string $\bx$ is moreover admissible,
we may assume by Proposition~\ref{prp:Adm} that $\bx=A(\bw)$ for some
$\bw\in\St(\uQ)$, and thus $F(\bx)=\widebar{\bx}=\bw$. In particular, 
$l(\bw)=l(\bx)$.  Now, if $\bw$ were symmetric, we would have 
$l(A(\bw))=(l(\bw)+1)/2=l(\bx)=l(\bw)$ and thus $l(\bw)=1$, which is impossible
for a string. We conclude, that for an admissible string $\bx$, the word
$F(\bx)$ is never a symmetric string, and thus it is an asymmetric string 
if and only if $\bx$ is of type $(u,u)$.

\begin{Def} \label{rem:Adm}
Let us denote by $\AdmSt(\uQ)\subset\St(\huQ)$ resp. 
$\AdmBa(\uQ)\subset\pBa(\huQ)$ the set of strings and
bands for $\huQ$ which are admissible with respect to $F$, and abbreviate
$\Adm(\uQ):=\AdmSt(\uQ)\cup\AdmBa(\uQ)$.  It is clear that $\AdmBa(\uQ)$
is stable under rotations, see~\ref{ssec:StBa} for the definition.  Thus,
we may consider the equivalence classes $[\AdmBa(\uQ)]$ of $\AdmBa(\uQ)$
under rotations.  We will use frequently the following abbreviations:
\begin{align*}
\Adm_s(\uQ) &:=\{\bx\in\AdmSt(\uQ)\mid \bx \text{ is not of type } (p,p)\},\\
[\Adm_b(\uQ)] &:=\{\bx\in\AdmSt(\uQ)\mid \bx \text{ is  of type } (p,p)\}\cup[\AdmBa(\uQ)].
\end{align*}
Thus $\widebar{?}$ induces a map
\[
  (\widebar{?})_s \df\Adm_s(\uQ)\ra\St(\uQ), \bx\mapsto\widebar{\bx}.
\]
By  Proposition~\ref{prp:Adm}, this map is surjective and
$(\widebar{\bx})_s=(\widebar{\by})_s$ implies either $\bx=\by$, or else
$\bx=\by^{-1}$ maps to  a symmetric string.  
Similarly, $\widebar{?}$ induces a map
\[
  [\widebar{?}]\df [\Adm_b(\uQ)]\ra [\pBa(\uQ)],
\]
which is also surjective.  Moreover, $[\widebar{\bx}]=[\widebar{\by}]$ implies
$\bx=\by$, or $\bx=\by^{-1}$ maps to a symmetric band in $\pBa'(\uQ)$, again
by Proposition~\ref{prp:Adm}. The situation may be visualized as follows:\\[.5cm]
\hspace*{2cm}\begin{tikzpicture} 
\node at (0,0)  {$\St(\huQ)$} ;
\node at (1,.05)  {$\supset$} ;
\node at (2.2,0)  {$\AdmSt(\uQ)$} ;
  \matrix at (4,-.3)
  {\node(l1){$(u,u)$};\\ \node{$(u,p)$}; \\ \node(l3){$(p,u)$};\\
    \node(l4){$(p,p)$};\\ \node(l5){$\phantom{(x,x)}$};\\};
\draw[decorate,decoration=brace] (l1.east) -- (l3.east);
\draw[decorate,decoration=brace] (l4.north east) -- (l5.south east);
\draw[decorate,decoration=brace] (l4.south west) -- (l1.north west);

\node at (0,-1.8)   {$\pBa(\huQ)$};
\node at (1,-1.75)  {$\supset$} ;
\node at (2.2,-1.8) {$\AdmBa(\uQ)$} ;
\path (5.8,.3) node[outer xsep=1](As) {$\Adm_s(\uQ)$} -- +(3, 0) node(St){$\St(\uQ)$} --  ++(0,-1.6)  node[outer xsep=1](Ab) {$\Adm_b(\uQ)$} -- ++(3,0) node(pB){$\pBa'(\uQ)$};
\draw[thick, ->] (As) -- node[above]{$\overline{?}$} (St);
\draw[thick, ->] (Ab) -- node[above]{$\overline{?}$} (pB);
\draw[rounded corners=4pt] (1.2,1.4) rectangle (6.6,-2.1);
\node at (3.9,-2.6) {$\Adm(\uQ)$};
\end{tikzpicture}
\end{Def}

The slightly unintuitive name $\Adm(\uQ)$ for a collection of strings
and bands of $\huQ$ stems from the fact, that one has to know the set
$\Qsp$ in order to determine which strings are admissible. This information, however, is lost in $\huQ$.

\subsection{AR-translate for admissible words} \label{ssec:AR-adm}
We say that $\bx\in\Adm(\uQ)$ is \emph{projective} resp.~\emph{injective}
if $\bx^{\pm 1}=A(\bp)$ for some $\bp\in\cP(\uQ)$ resp.
$\bx^{\pm 1}=A(\bq)$ for some $\bq\in\cQ(\uQ)$, with the notion of projective
resp.~injective strings from Section~\ref{ssec:proj}.
  
In view of Proposition~\ref{prp:Adm} it makes sense to define
also for  $\bx\in\Adm(\uQ)$ non-projective
the operator $\tau_{\uQ}$ as follows: 
  \[
    \tau_{\uQ}(\bx):=\begin{cases} A(\tau_{\uQ}(\bw)) &\text{ if } \bx=A(\bw)
      \text{ for some }\bw\in\St(\uQ)\cup\pBa'(\uQ),\\
 A(\tau_{\uQ}(\bw))^{-1} &\text{ if } \bx=A(\bw)^{-1}
      \text{ for some }\bw\in\St(\uQ),
    \end{cases}
  \]
where we used the AR-operator $\tau_{\uQ}$ from Section~\ref{ssec:proj}.

\begin{Rem} \label{rem:Admtau}
This is well-defined, since the map $A\df\St(\uQ)\cup\pBa'(\uQ)\ra\Adm(\uQ)$
is by construction injective. Note that if $\bx\in\Adm(\uQ)$ is a \emph{string}
of type $(p,p)$, we have $\tau_{\uQ}(\bx)=\bx$. If $\bx$ is a string of type
$(a,b)\in\{u,p\}^2$, then $\tau_\uQ(\bx)$ is also of type $(a,b)$. Note also,
that if $\bx\in\Adm(\uQ)$ is not projective and of type $(p,u)$, we have
$\tau_\uQ(\bx)=\bx[1]$ and  similarly $\tau_\uQ(\bx)=[1]\bx$ if $\bx$ is
of type $(u,p)$; If $\bx\in\Adm(\uQ)$ is non-projective of type $(u,u)$,
at least one of $([1]\bx)[1]$ and $[1](\bx)[1])$ is defined, and 
accordingly $\tau_\uQ(\bx)=([1]\bx)[1]$ or $\tau_\uQ(\bx)=[1](\bx[1])$.
In all these cases, the operator $?[1]$ corresponds to taking the successor
in the corresponding linearly ordered set $\St_{i,\rho}(\huQ)$. 
\end{Rem}

\subsection{Reading admissible strings and bands} \label{ssec:read}
For the combinatorial description of homomorphisms between representations
of skewed-gentle algebras we need the following, slightly unwieldy, definitions.
For $\bx=x_0x_1\cdots x_l\in\Adm(\uQ)$, recall that $\widebar{\bx}$
was defined by~\eqref{eq:cpl}. We agree on the following labeling  of its letters:
\[
\widebar{\bx} =
\begin{cases}
\bar{x}_{-l}\cdots\bar{x}_{-1}\bar{x}_0\cdots\bar{x}_l &\text{if }
\bx\text{ is of type } (p,u),\\
\bar{x}_0\bar{x}_1\cdots\bar{x}_m &\text{else,}
\end{cases}
\text{ and }
m:=\begin{cases}
 l    &\text{if } \bx\text{ is of type } (u,u),\\
 2l-1 &\text{if } \bx\text{ is of type } (u,p),\\
 2l-2 &\text{if } \bx\text{ is of type } (p,p),\\
 l    &\text{if } \bx\text{ is of type }  (b).\\
 \end{cases}
\]
Next, we set
\[
  \widebarbar{\bx}:=\begin{cases} \qquad\widebar{\bx}\qquad &\text{ if }
    \bx\in\Adm_s(\uQ),\\
    \cdots\widebar{\bx}\widebar{\bx}\widebar{\bx}\cdots
    &\text{ else,}
  \end{cases}
\]
where we agree that 
$\cdots\widebar{\bx}\widebar{\bx}\widebar{\bx}\cdots = 
(\widebarbar{x}_i)_{i\in\ZZ}$ with $\widebarbar{x}_{k(m+1)+i}=\widebar{x}_i$ for $0\leq i\leq m$ and $k\in\ZZ$. Thus, we have in particular 
$\widebarbar{x}_i=F(x_i)$ for $0\leq i\leq l$, in any case.  


For $0\leq i\leq l$  we define accordingly
\begin{align*}
\widebarbar{\bx}^{[i]}&:=\begin{cases}
\bar{x}_i\bar{x}_{i+1}\cdots\bar{x}_l &\text{if } \bx\text{ is of type } (u,u),\\
\bar{x}_i\bar{x}_{i+1}\cdots\bar{x}_m &\text{if } \bx\text{ is of type } (u,p),\\
\bar{x}_i\bar{x}_{i+1}\cdots\bar{x}_l &\text{if } \bx\text{ is of type } (p,u),\\
\widebarbar{x}_i\widebarbar{x}_{i+1}\widebarbar{x}_{i+2}\cdots &\text{else,}
\end{cases}\\
\widebarbar{\bx}_{[i]}^{-1}&:=\begin{cases}
\bar{x}_i^{-1}\bar{x}_{i-1}^{-1}\cdots\bar{x}_0^{-1} &\text{if } \bx\text{ is of type } (u,u),\\
\bar{x}_i^{-1}\bar{x}_{i-1}^{-1}\cdots\bar{x}_0^{-1} &\text{if } \bx\text{ is of type } (u,p),\\
\bar{x}_i^{-1}\bar{x}_{i-1}^{-1}\cdots\bar{x}_{-l}^{-1} &\text{if } \bx\text{ is of type } (p,u),\\
\widebarbar{x}_i^{-1}\widebarbar{x}_{i-1}^{-1}\widebarbar{x}_{i-2}^{-1}\cdots &\text{else.}
\end{cases}
\end{align*}
With this at hand,  for $(i,\rho)\in\{1,\ldots,l\}\times\{-1,+1\}$ we set
\[
  \widebarbar{\bx}(i,\rho):=
  \begin{cases}
    \widebarbar{\bx}^{[i]} &\text{ if } t_1(\widebarbar{\bx}_i)=\rho,\\
    \widebarbar{\bx}_{[i-1]}^{-1} &\text{ if } t_1(\widebarbar{\bx}_i)=-\rho.
  \end{cases}
\]
Note, that for $\bx$ of type $(u,u)$  or of type $(u,p)$ we have 
$\widebarbar{x}_0=\II_{i,\sig}^{-1}$. Thus, since $t_1(\widebarbar{x}_0)=*$,
we can not define $\widebarbar{\bx}(0,\rho)$ for $\rho\in\{-1,+1\}$.
It makes however perfect sense to define, along the same lines 
$\widebarbar{\bw}(0,\rho)$, in case $\bw\in\pBa(\huQ)$.

Similarly, we define, 
for $\del\in\{-,+\}$ and $\bx=x_0x_1\cdots x_l\in\Adm(\uQ)$,
the possibly infinite, words
\[
\widehat{\bx}^\del:=\begin{cases}
\phantom{\cdots (\bx')^{-1}\eps_{s(x_0)}^{\del 1}}\bx
    &\text{if }\bx  \text{ is of type } (u,u),\\
\phantom{\cdots (\bx')^{-1}\eps_{s(x_0)}^{\del 1}}\bx'\eps_{t(x_l)}^{\del 1}(\bx')^{-1}
    &\text{if } \bx \text{ is of type } (u,p),\\
\phantom{\cdots} (\bx')^{-1}\eps_{s(x_0)}^{\del 1} \bx'
    &\text{if } \bx \text{ is of type } (p,u),\\
    \cdots (\bx')^{-1}\eps_{s(x_0)}^{\del 1}\bx'\eps_{t(x_l)}^{\del 1}(\bx')^{-1}\cdots 
    &\text{if } \bx \text{ is of type } (p,p),\\
    \qquad\qquad\;\;\cdots\bx\bx\bx\cdots &\text{if }\bx\text{ is of type } (b)
  \end{cases}
\]
where $\bx'$ is obtained from $\bx$ by removing the punctured letters. 
Again, by insisting  that 
$(\widehat{\bx}^\del)_i=x_i$ for $i=1,2,\ldots,l-1$,
we obtain in all cases a well-defined word.
Next, for $(i,\rho)\in\{1,\ldots,l\}\times\{-1,+1\}$, we set
\[
  \widehat{\bx}^\del(i,\rho):=
  \begin{cases}
    (\widehat{\bx}^\del)^{[i]} &\text{ if } t_1(\widehat{\bx}^\del_i)=\rho,\\
    (\widehat{\bx}^\del_{[i-1]})^{-1} &\text{ if } t_1(\widehat{\bx}^\del_i)=-\rho,
  \end{cases}
\]
where $(\widehat{\bx}^{\del})^{[i]}$ and $(\widehat{\bx}^{\del}_{[i]})^{-1}$
are defined similarly as $\widebarbar{\bx}^{[i]}$ and 
$\widebarbar{\bx}^{-1}_{[i]}$ above. 

We will need  also $\widehat{\bx}^\del(0,\rho)$ in case $\bx\in\pBa(\huQ)$,
which is then defined correspondingly.  We observe, that we have trivially
\begin{equation} \label{eq:obspm}
  \widehat{\bx}^+(i,\rho)\preceq\widehat{\bx}^-(i,\rho),
\end{equation}
whenever this is defined, since $\eps_j\prec \eps_j^{-1}$ in $\cL(\huQ)$
for all loops.

\section{Quivers for modules and homomorphisms} \label{sec:Quiv}
\subsection{Quivers for admissible words}
\label{ssec:Hw}
Let from now on $\uQ$ be a skewed-gentle quiver. Thus, in particular
all special arrows are loops. Following closely~\cite[Sec.~3]{CB89b},
we construct for each admissible $\bx\in\Adm(\uQ)$ a morphism
\[
  G_\bx\df \uH(\bx) \ra\uQ 
\]
of polarized quivers. Recall, that 
the elements of $\Adm(\uQ)$ are precisely those strings or bands for the
associated gentle quiver $\huQ$ which are  of the form
$A(\bw)^{\pm 1}$ for $\bw\in\St(\uQ)\cup\pBa(\uQ)$,
see Proposition~\ref{prp:Adm}.

\begin{table}[h]
\begin{align*}
  \def\objectstyle{\scriptstyle}
  \sfA_n\colon&  \xymatrix{\qquad 1\ar@{-}[r]^>>>>{\nu_1}&2\ar@{-}[r]^{\nu_2} &\cdots \ar@{-}[r]^{\nu_{n-2}} &{n\!\!-\!1\!}\ar@{-}[r]^>>>>>{\nu_{n-1}}& n}\\
  \sfD'_n\colon&  \xymatrix{\qquad 1\ar@{-}[r]^>>>>{\nu_1}&2\ar@{-}[r]^{\nu_2} &\cdots \ar@{-}[r]^{\nu_{n-2}} &{n\!\!-\!\!1}\ar@{-}[r]^>>>>>{\nu_{n-1}}& n \ar@(ur,dr)[]^{\eta_1}}\\
     &\!\!\!\xymatrix{\ar@(ul,dl)[]_{\eta_0} 1\ar@{-}[r]^{\nu_1}&2\ar@{-}[r]^{\nu_2} &\cdots \ar@{-}[r]^{\nu_{n-2}} &{n\!\!-\!\!1}\ar@{-}[r]^>>>>>{\nu_{n-1}}& n} \\     
  \sftA_n\colon&  \xymatrix{\qquad 1\ar@{-}[r]^>>>>{\nu_1}&2\ar@{-}[r]^{\nu_2} &\cdots \ar@{-}[r]^{\nu_{n-2}} &{n\!\!-\!\!1}\ar@{-}[r]^>>>>>{\nu_{n-1}}&\ar@{-}[lld]^{\nu_n} n\\
              &&\ar@{-}[llu]^{\nu_0}0}\\
    \sftD'_n\colon&  \xymatrix{\ar@(ul,dl)[]_{\eta_0} 1\ar@{-}[r]^{\nu_1}&2\ar@{-}[r]^{\nu_2} &\cdots \ar@{-}[r]^{\nu_{n-2}} &{n\!\!-\!\!1}\ar@{-}[r]^>>>>>{\nu_{n-1}}& n \ar@(ur,dr)[]^{\eta_1}}\\   
\end{align*}
\caption{Underlying graphs of the quivers $\uH(\bx)$}
\label{tab:AD-graphs}
\end{table}

We consider first the case when $\bx=x_0x_1\cdots x_n$ is a string (for $\huQ$).  
Then the underlying graph of $H(\bx)$ is a subgraph of the diagram
$\sftD'_n$ from Table~\ref{tab:AD-graphs}.  If $x_0$ is unpunctured,
we remove the loop $\eta_0$, else $G(\eta_0)=\eps_{s(x_0)}$.  
Similarly, if $x_n$ is unpunctured, the loop $\eta_1$ is removed and else
$G_\bx(\eta_1)=\eps_{t(x_n)}$.
For $i=1,2,\ldots, n-1$ the edge $\nu_i$ is oriented to the left if and only if $x_i$ is a direct letter and we set $G_\bx(\nu_i)=x_i$.
Else, $\nu_i$ is oriented to the right and we set $G_\bx(\nu_i)=x_i^{-1}$.
Finally, $G_\bx(i):=t(x_i)$ for $i=1,2,\ldots, n$.

If $\bx=x_0x_1\cdots x_n$ is a band for $\huQ$, the underlying graph of
$H(\bx)$ is of type $\sftA_n$ from Table~\ref{tab:AD-graphs}. In this
case, all letters of $\bx$ are direct or inverse ordinary letters. For
$i=0,1, \ldots, n$ we orient $\nu_i$ anti-clockwise if $x_i$ is a direct letter
and set $G_\bx(\nu_i)=x_i$. Else, $\nu_i$ is oriented clockwise and we set
$G_\bx(\nu_i)=x_i^{-1}$. As for the vertices, we set $G_\bx(i):=t(x_i)$
for $i=0, 1, \ldots, n$.

We can upgrade $H(\bx)$ to a polarized quiver $\uH(\bx)$ by declaring the
loops to be the only special arrows, and taking $s_1(\alp)=s_1(G_\bx(\alp))$
as well as $t_1(\alp)=t_1(G(\alp))$ for all arrows $\alp\in H_1(\bx)$. 

By definition,  the \emph{boundary vertices} of $\uH(\bx)$ are 
precisely the vertices $j\in H_0(\bx)$ with valency 
$\operatorname{Val}_{\uH(\bx)}(j)\leq 1$. Here,
\[
\operatorname{Val}_{\uH}(j):=\abs{\{\alp\in H_1\mid s(\alp)=j\}\cup
\{\bet\in H_1\mid t(\bet)=j\}}.
\]
Thus, the hereditary quivers of type $\sftA_n$ and of type $\sftD_n$ have no boundary vertices. 
The quivers of type $\sfA_1$ and of type $\sfD'_n$ have exactly one boundary vertex, 
and the quivers of type $\sfA_n$ with $n\geq 2$ have exactly two boundary vertices.

\subsection{The quiver
  \texorpdfstring{$H_\uQ(\bx,\by)$}{H(x,y)}} \label{ssec:Hvw}
To each pair of elements $\bx,\by\in\Adm(\uQ)$ we construct, starting
from $G_\bx\df \uH(\bx)\ra\uQ$ and $G_\by\df \uH(\by)\ra\uQ$ a decorated quiver
$H_\uQ(\bx,\by)$ with \emph{three} kinds of arrows. It is closely related to
the construction of $H^{\by,\bx}$ in~\cite[Sec.~3.3]{Ge99}.

\begin{align*}
H_0(\bx,\by) &:=\{(j,i)\in H_0(\by)\times H_0(\bx)\mid G_\by(j)=G_\bx(i)\},\\
H_{1,+}(\bx,\by) &:=\{(\bet,\alp)_+\in (\Hord(\by)\times\Hord(\bx))\cup
               (\Hsp(\by)\times\Hsp(\bx)) \mid G_\by(\bet)=G_\bx(\alp)\},\\
H_{1,\times} (\bx,\by)&:=\{(\nu,\mu)_\times\in \Hord(\by)\times\Hord(\bx)\mid
  G_\by(\nu)=G_\bx(\mu)\in\Qsp\},\\
H_{1,o}(\bx,\by) &:= \{(\nu,\eta)_\ominus\in \Hord(\by)\times \Hsp(\bx)\mid G_\by(\nu)=G_\bx(\eta)\}\\
             &\ \cup \{(\eta,\mu)_\oplus\in \Hsp(\by)\times \Hord(\bx)\mid G_\by(\eta)=G_\bx(\mu)\},\\
  H_1(\bx,\by) &:= H_{1,+}\cup H_{1,\times}\cup H_{1,o},
\end{align*}
with start- and endpoints of the arrows given by
\begin{alignat*}{2}
s(\bet,\alp)_+ &=(s(\bet),s(\alp)), &\qquad t(\bet,\alp)_+&=(t(\bet),t(\alp)),\\
s(\nu,\mu)_\times &=(t(\nu),s(\mu)), &\qquad t(\nu,\mu)_\times&=(s(\nu), t(\mu)),\\
s(\nu,\eta)_\ominus &=(t(\nu),s(\eta)), &\qquad t(\nu,\eta)_\ominus &=(s(\nu), t(\eta)),\\
    s(\eta,\mu)_\oplus &=(s(\eta),s(\mu)), &\qquad t(\eta,\mu)_\oplus &=(t(\eta),t(\mu)).
\end{alignat*}

Finally, we highlight certain vertices of $H(\bx,\by)$ by one of the following
colors: red, orange and purple. For the red vertices we use the notation from
Section~\ref{ssec:read}.
\begin{align*}
\redx{\Hred}  &:= \{(j,i)\in H_0(\bx,\by)\mid
           \widebarbar{\by}(j,\rho)>\widebarbar{\bx}(i,\rho) \text{ and }
           \Del(\widebarbar{\by}(j,\rho),\widebarbar{\bx}(i,\rho))=0\\
       & \hspace*{9cm} \text{ for some }\rho\in\{-1,1\}\},\\
\orax{\Hora}  &:= \{ t(\nu,\mu)_\times \mid (\nu,\mu)_\times\in H_{1,\times}(\bx,\by)\},\\
\purx{\Hpur}  &:= \{ t(\nu,\mu)_* \mid (\nu,\mu)_*\in H_{1,o}(\bx,\by)\}.\\
\end{align*}
Note that a red vertex can also receive an orange label or a purple label.
It can even receive a second red label if it is an isolated vertex. 
However, an orange and a purple label never occur at the same vertex.

For $(j,i)\in H_0(\bx,\by)$ we define $\val(j,i)$ to be the number of arrows
in $H(j,i)$ which are incident with $(j,i)$.   Moreover, we say that
$(j,i)$ is a \emph{boundary point} of $H(\bx,\by)$ is $j$ is a boundary point
of $H(\by)$ or $i$ is a boundary point of $H(\bx)$ in the sense of
the end of Section~\ref{ssec:Hw}. See Section~\ref{ssec:exHxy} for an example.

\begin{Rem}
We observe that for each
$(j,i)\in H_0(\bx,\by)$ we have $\val(j,i)\leq 2$.

It follows from the definitions, that for
$(j,i)\in H_0(\bx,\by)$ and $\rho\in\{-1,+1\}$
the right inextensible (or right periodic) words
$\widebarbar{\bx}(i,\rho)$ and $\widebarbar{\by}(j,\rho)$ are comparable.

Finally, we observe that $(j,i)\in\Hred(\bx,\by)$ implies $\val(j,i)\leq 1$.
\end{Rem} 

\subsection{h-lines and kisses} For $\bx,\by\in\Adm(\uQ)$ we look at
the connected components of the quiver $H_\uQ(\bx,\by)$ from the previous
section.

\begin{Def} \label{def:lines} Let $\bx,\by\in\Adm(\uQ)$.\\
We consider the following two subquivers 
\[
 H_\uQ^+(\bx,\by):=(H_0(\bx,\by), H_{1,+}(\bx,\by))\ \subset\    
 H^o_\uQ(\bx, \by):=(H_0(\bx, \by), H_{1,+}(\bx,\by)\cup H_{1,o}(\bx,\by)).
\]
of $H_\uQ(\bx, \by)$. With this at hand, we define:

(1) A connected component of $H^+_\uQ(\bx,\by)$, which
contains no red, orange or purple vertices is a \emph{real h-line} of $H_\uQ(\bx,\by)$. The set of real h-lines of 
$H_\uQ(\bx,\by)$ is denoted by $\cR_\uQ(\bx,\by)$.

(2) A connected component of $H^o_\uQ(\bx, \by)$,
which contains no red or orange vertices, is a \emph{$\tH$-line.}
A connected component of $H^+_\uQ(\bx,\by)$, which is contained in a $\tH$-line, is a \emph{h-line}.

(3)  A connected component of the quiver $H_\uQ(\bx,\by)$, which contains no red vertex, is a \emph{long h-line}. The set of
long h-lines of  $H_\uQ(\bx,\by)$ is denoted by $\cH_\uQ(\bx,\by)$.

(4) A real h-line $L\in\cR_\uQ(\bx,\by)$,  is a \emph{kiss} if none of its (at most 2)  endpoints is a boundary point of 
$H_\uQ(\bx,\by)$.

(5)  Suppose that for 
$(j,i)\in H_0(\by)\times H_0(\bx)$ we have 
\[
  \widebarbar{\by}(j,\rho)=\widebarbar{\bx}(i,\rho)\text{ for both }
  \rho\in\{-1,+1\},
\]
then it is easy to see that we have necessarily $\bx\sim\by$, and moreover $(j,i)$ lies
on a component $L$ of $H(\bx,\by)$, whose underlying quiver is isomorphic
to the one of $H(\bx)\cong H(\by)$.  We say that in this situation $L$ is a 
\emph{generalized diagonal}.   
Such a component is in particular a real h-line.
\end{Def}

\begin{Rem}
(1) Obviously, each vertex of a $\tH$-line is contained in precisely one $h$-line.
We will find in Section~\ref{ssec:chls} alternative characterization of the different types of h-lines, 
which allow us to uncover further relations between these lines.  
In particular, we will verify the non-trivial fact, that
each real h-line is indeed a h-line, as the name suggests. 
See~Proposition~\ref{prp:bij-hl} for more details.
We will also prove the crucial fact that each real h-line is (in general properly) contained in a 
unique long h-line, and that this inclusion induces a bijection 
$\cR_\uQ(\bx,\by)\ra\cH_\uQ(\bx,\by)$.
See also Example~\ref{ssec:exHxy} below for illustration.    

(2) If $L$ is a connected component of the quiver $H(\bx, \by)$, 
then the underlying graph of $L$ is one of the types
$\sfA_n, \sfD'_n, \sftA_n$ or $\sftD'_n$ for some $n$. If $L$ is of type
$\sftA_n$ or of type $\sftD'_n$, then $\bx\sim\by$ is a band, and $L$ is 
a generalized diagonal of $H(\bx,\by)$. This follows easily, since all
vertices of $(j,i)\in H_0(\bx,\by)$ have $\val(j,i)\leq 2$, as we already
observed. In particular, a vertex with a loop is incident to at most one other
arrow.

(3)  If $\Qsp=\emptyset$, i.e.~if $\uQ$ is a gentle, polarized quiver,
the three notions of h-lines coincide. In this case, they are in bijection
with the isomorphism classes of Krause's admissible
triples~\cite[p.~189]{Kr91}.  See also Section~\ref{ssec:triph}.

(4) If  $\bx$ and $\by$ both don't have punctured letters, the quiver
$H_\uQ(\bx,\by)$ has no loops, and thus each h-line is a real h-line.  

(5) If $\bx,\by\in\Adm_b(\uQ)$, i.e. if $H_\uQ(\bx,\by)$ has no boundary
vertices, then trivially \emph{all} real h-lines are kisses. 
On the other hand, if $\bx\in\AdmSt_s(\uQ)$, then the diagonal $\id_\bx\in\cH_\uQ(\bx,\bx)$ is \emph{not} a kiss, since it involves boundary points.
If $\uQ$ is gentle,  our notion  of a kiss is equivalent to the notion of a kiss in~\cite[Sec.~4]{BY20}, though we also allow kisses, which involve bands.  
\end{Rem}

\subsection{Elementary properties of h-lines}
Consider a connected component of $H_\uQ(\bx,\by)$ of type 
$\sfA$. Thus, the underlying graph of $L$ is 
\[ 
\abs{L} =\xymatrix{(j_0, i_0) \ar@{-}[r]^{\lam_1} &(j_1, i_1) \ar@{-}[r]^{\lam_2}&
    \cdots \ar@{-}[r]^{\lam_l} & (j_l, i_l)}
\]
and  we have a sequence $(\ome_1,\ome_2,\ldots, \ome_l)\in\{-,+\}^l$,
such that $\lam_k$ points to the left if and only if $\ome_k=+$. Moreover,
each arrow is of the form $\lam_k=(\bet_k,\alp_k)_{d(k)}$ with
$d(k)\in\{+, \times, \ominus, \oplus\}$, where $G_\by(\bet_k)=G_\bx(\alp_k)\in Q_1$.
Moreover, we set 
\[
\kap_k:=\begin{cases} t_1(G_\by(\bet_{k+1})) &\text{if } \ome_k=+,\\
    s_1(G_\by(\bet_{k+1})) &\text{if } \ome_k=-\end{cases}
\text{ for } 0\leq k<l, \text{ and }
\kap_l:=\begin{cases} -s_1(G_\by(\bet_{l})) &\text{if } \ome_l=+,\\
    -t_1(G_\by(\bet_{l})) &\text{if } \ome_l=-,\end{cases}
\]
where we recall that $s_1(\vph)$ and $t_1(\vph)$ are the polarizations
of an arrow $\vph$ of the skewed-gentle polarized quiver $\uQ$.
Now, consider for a vertex  $(j_k, i_k)\in L_0$  (i.e. $0\leq k\leq l$),  the possibly right infinite words from Section~\ref{ssec:read}:
\begin{align*}
    \widebarbar{\bx}(i_k,\kap_k) &= x_1 x_2\cdots, &
    \widebarbar{\by}(j_k,\kap_k) &= y_1 y_2\cdots,\\
    \widehat{\bx}^+(i_k,\kap_k) &= x_1' x_2'\cdots, &
    \widehat{\by}^+(j_k,\kap_k) &= y_1' y_2'\cdots,\\ 
  \widehat{\bx}^-(i_k,\kap_k) &= x_1'' x_2''\cdots, &
    \widehat{\by}^-(j_k,\kap_k) &= y_1'' y_2''\cdots.\\  
\end{align*}
In Table~\ref{tab:arr-let} we have worked out the precise  relations
between the arrows of $L$ and the first $l-k$  letters of the various words.
\begin{table}[hbt!]
\begin{tabular}{|c|c|c|c|c|}\hline
&\multicolumn{4}{|c|}{$\ome_{j+k}=+$}\\\hline
letter\textbackslash d(k+j) & + & $\times$ & $\oplus$ & $\ominus$\\\hline
$x_j=y_j$ & $\alp^{(*)}$ & $\alp^*$ &  $\alp^*$ & $\alp^*$ \\\hline
$x'_j$ & $\alp$ & $\alp$ &  $\alp$ & $\alp$ \\\hline
$x''_j$ & $\alp$ & $\alp$ &  $\alp$ & $\alp^{-1}$ \\\hline
$y'_j$ & $\alp$ & $\alp^{-1}$ &  $\alp$ & $\alp^{-1}$ \\\hline
$y''_j$ & $\alp$ & $\alp^{-1}$ &  $\alp^{-1}$ & $\alp^{-1}$ \\\hline
\end{tabular}
\quad
\begin{tabular}{|c|c|c|c|c|}\hline
&\multicolumn{4}{|c|}{$\ome_{j+k}=-$}\\\hline
letter\textbackslash d(k+j) & + & $\times$ & $\oplus$ & $\ominus$\\\hline
$x_j=y_j$ & $\alp^{(*)}$ & $\alp^*$ &  $\alp^*$ & $\alp^*$ \\\hline
$x'_j$ & $\alp^{-1}$ & $\alp^{-1}$ &  $\alp^{-1}$ & $\alp$ \\\hline
$x''_j$ & $\alp^{-1}$ & $\alp^{-1}$ &  $\alp^{-1}$ & $\alp^{-1}$ \\\hline
$y'_j$ & $\alp^{-1}$ & $\alp$ &  $\alp$ & $\alp$ \\\hline
$y''_j$ & $\alp^{-1}$ & $\alp$ &  $\alp^{-1}$ & $\alp$ \\\hline
\end{tabular}\\[2ex]
\caption{Relation between the arrow $\lam_{k+j}$ and letters $x_j=y_j$ and
$x_j^*, y^*_j$ for $*\in\{', ''\}$,\quad Convention: $G_\bx(\alp_{j+k})=G_\by(\bet_{j+k})=:\alp\in Q_1$ }
\label{tab:arr-let}
\end{table}

It is straightforward to produce in this situation  a table, similar to Table~\ref{tab:arr-let}, for the
first $k$ letters of each of the words $\widebarbar{\bx}(i_k, -\kap_k)$,
$\hbx^{\pm}(i_k,-\kap_k)$,  $\hby^\pm(j_k,-\kap_k)$ on one side and
the arrows $\lam_k, \lam_{k-1}, \ldots, \lam_1$ on the other side. 

The proof of the following lemma is a bit tedious. However, it illustrates
the interplay of our definitions.

 \begin{Lem} \label{lem:hl-basic}
   Let $\uQ$ be a skewed-gentle polarized quiver and
   $\bx, \by\in\Adm(\uQ)$.
   \begin{itemize}
\item[(a)]
A connected component $L$ of $H(\bx,\by)$ is a \textbf{long h-line} if 
and only if for some vertex $(j,i)$ on $L$ we have
\begin{equation} \label{eq:lhl}
\widebarbar{\by}(j,\rho)\leq \widebarbar{\bx}(i,\rho)
\text{ for both } \rho\in\{-1,+1\}.
\end{equation}
Moreover, condition~\eqref{eq:lhl} is independent of the choice of the vertex  $(j,i)\in L_0$.
\item[(b)]
A connected component $L$ of $H^o_\uQ$ is a
\textbf{$\tH$-line} if and only if for each vertex $(j,i)$ on $L$ we have
     \begin{equation} \label{eq:hl}
       \widehat{\by}^+(j,\rho)\preceq \widehat{\bx}^-(i,\rho)
       \text{ for any } \rho\in\{-1,+1\}.
     \end{equation}
Moreover, if $L'$ is any connected component of $H^+_\uQ(\bx,\by)$ and 
condition~\eqref{eq:hl} holds for some $(j_0,i_0)\in L'_0$, then it holds automatically for all
$(j,i)\in L'_0$. 
\item[(c)]
If a connected component $L$ of $H_\uQ^+$ is a
\textbf{real h-line}, for each vertex  $(j,i)$ on $L$ we have
\begin{equation} \label{eq:rhl}
\widehat{\by}^{\del(\rho)}(j,\rho)\preceq 
\widehat{\bx}^{\del(\rho)}(i,\rho) \text{ for  both } 
\rho\in\{-1,+1\} \text{ and certain } \del(\rho)\in\{+, -\}.
\end{equation}  
\end{itemize}
\end{Lem}

\begin{proof}
(a) Consider a connected component $L$ of $H_\uQ(\bx,\by)$.  If $L$
is a generalized diagonal, the claim is trivial.  Else, $L$ is of type
$\sfA$ or of type $\sfD'$.  In the first case, we adopt for $L$ the notations from the beginning of the section. 
It follows then from our definitions  that
\[
\Del(\widebarbar{\bx}(i_k,\kap_k), \widebarbar{\by}(j_k,\kap_k))=
\begin{cases}
   l-k &\text{if } \widebarbar{\bx}(i_k,\kap_k)\neq \widebarbar{\by}(j_k,\kap_k),\\
   l-k+1 &\text{if }\widebarbar{\bx}(i_k,\kap_k)= \widebarbar{\by}(j_k,\kap_k),
\end{cases}
\]
and in any case the letters $x_{l-k+1}, y_{l-k+1} \in (\cL(\uQ),<)$
are comparable. A similar discussion applies for the comparison of 
$\widebarbar{\bx}(i_k,-\kap_k)$ with 
$\widebarbar{\by}(j_k,-\kap_k)$.  However, in this case,
$\Del(\widebarbar{\bx}(i_k,-\kap_k), \widebarbar{\by}(j_k,-\kap_k)\in\{k, k+1\}$.   We conclude:
For any vertex $(j_k,i_k)\in L_0$ we have
$\widebarbar{\by}(j_k,\kap_k)\preceq\widebarbar{\bx}(i_k,\kap_k)$
if and only if $(j_l,i_l)$ is not a red vertex.  
Similarly, 
$\widebarbar{\by}(j_k,-\kap_k)\preceq\widebarbar{\bx}(i_k,-\kap_k)$
if and only if $(j_0,i_0)$ is not a red vertex.

If $L$  is of type $\sfD'$, we may use the same convention, and may
assume that $L$ has a special loop $\eta_0$ at $(j_0,i_0)$. It is 
straightforward to see that in this case the condition 
$\widebarbar{\by}(j_k,-\kap_k)\leq \widebarbar{\by}(i_k,-\kap_k)$
is equivalent to the condition 
$\widebarbar{\by}(j_k,\kap_k)\leq \widebarbar{\by}(i_k,\kap_k)$.
This latter condition is, with the same argument as in the previous case,
equivalent to the fact that $(j_l, i_l)$ is not a red vertex. 

(b) If $L$ is a generalized diagonal, we may assume $\bx=\by$, and our claim follows from~\eqref{eq:obspm}. So, we may assume  that $L$ is of type
$\sfA$ or is of type $\sfD'$. We work out the case when $L$ is of type
$\sfA$.  The case when $L$ is of type $\sfD'$ is similar. 
For convenience, we use for the connected component $L$ of $H^o_\uQ(\bx,\by)$ the same notation as at the beginning of this section,
with the only difference that the label $d(k)=\times$ will not occur. 

Now, suppose first, that $L$ contains no red or orange vertices, and consider
a vertex $(j_k,i_k)\in L_0$. It follows  that
$\hby^+(j_k,\kap_k)=y_0y_1\cdots y_{l-k}\cdots$ and 
$\hbx^-(i_k,\kap_k)=x_0x_1\cdots x_{l-k}\cdots$ have both  at least $l-k+1$
letters, and are possibly right infinite.  Since $(j_l,i_l)$ is not red
or orange, $y_{l-k}\preceq x_{l-k}$ follows. 
Moreover, for $0\leq r <l-k$, the letters $x_r$ and $y_r$  are determined
by $(\lam_{k+r},\ome_{k+r})$ and in fact $x_r=y_r$, 
except when $\ome_{k+r}=-$ and $d(k+r)\in\{\ominus, \oplus\}$.  
In this case $y_r=\eps$ and $x_r=\eps^{-1}$
for $\eps=G_\by(\bet_{k+r})=G_\bx(\alp_{k+r})\in\Qsp$,  by the construction of
the arrows in $H_{1,o}(\bx,\by)$. 
This implies obviously
$\hby^+(j_k,\kap_k)\preceq \hbx^-(i_k,\kap_k)$.  Similarly,
we find that both, $\hbx^-(i_k,-\kap_k)=x'_0x'_1\cdots x'_k\cdots$ and
$\hby^+(i_k,-\kap_k)=y'_0y'_1\cdots y'_k\cdots$ have at least $k+1$ letters
with $y'_k<x'_k$.  In this case, for $0\leq r< k$ the letters $x'_r$ and $y'_r$ are determined by $(\lam_k,\ome_k)$, and we have $x'_r=y'_r$,  except when
$\kap_r=+$ and $d(r)\in\{\ominus, \oplus\}$. In this case we have
$y'_r=\eps$ and $x'_r=\eps^{-1}$ for $\eps=G_\by(\bet_r)=G_\bx(\alp_r)\in\Qsp$.
This implies $\hby^+(j_k,-\kap_k)\preceq\hby^-(i_k,-\kap_k)$.  
The last claim about $L'$ follows by an easy induction argument.

(c) Again, if $L$ is a generalized diagonal, the  claim is trivial.
We use for the real h-line $L\subset H^+_\uQ(\bx,\by)$
of type $\sfA$ 
the labeling scheme from the beginning of the section, however now
$d(k)=+$ for all $k$. Consider $(j_k, i_k)\in L_0$. Then we have clearly
\[
D^\del_k:=\Del(\hby^\del(j_k,\kap_k), \hbx^\del(i_k,\kap_k))\in\{l-k, l-k+1\}. 
\]
Let us first suppose that at $(j_l, i_l)$ starts, in $H_Q(\bx,\by)$, an arrow 
\[
\lam=(\bet,\alp)_d\in H_{1,\times}(\bx,\by)\cup H_{1,o}(\bx,\by).
\]
We have then in view of Table~\ref{tab:arr-let}
\[
D^\del_k=l-k  \text{ and }
\hby^\del(j_k,\kap_k) \prec\hbx^\del(i_k,\kap_k) \quad\text{ for } 
\begin{cases}
    \del\in\{-,+\} & \text{if } d=\times,\\
    \del=+         & \text{if } d=\oplus,\\
    \del=-         & \text{if } d=\ominus.
\end{cases}
\]
Else, we have $x:=x'_{l-k+1}=x''_{l-k+1}$ and $y:=y'_{l-k+1}=y''_{l-k+1}$,
and $\II_{s,-\kap_l}\in\{x, y\}$ for $s=G_\bx(i_l)=G_\by(j_l)$. Since
in this situation $(j_l, i_l)$ is a boundary vertex which is not red,
we have either $x=\II_{s,-\kap_l}=y$ and thus 
$\hby^\del(j_k,\kap_k)=\hbx^\del(i_k,\kap_k)$ for any $\del\in\{-,+\}$,
or we have $y<x$ and thus 
$\hby^\del(j_k,\kap_k)\prec\hbx^\del(i_k,\kap_k)$ for any 
$\del\in\{-,+\}$.  The case when $L$ is of type $\sfD'$ is similar. 
\end{proof}

\begin{Rem}
We will find in Section~\ref{ssec:chls} characterizations of (real) h-lines
that are similar to the characterization of long h-lines in Lemma~\ref{lem:hl-basic}~(a). 
In particular, we will sharpen considerably Lemma~\ref{lem:hl-basic}~(c). 
\end{Rem}

\subsection{The combinatorics of h-lines} \label{ssec:chls}
Let us also restate the crucial result~\cite[Prp.~3~(d)]{Ge99}
(see also  Definition~(2) in~\cite[Sec.~3.3]{Ge99})
in a more symmetric form:

\begin{Lem} \label{lem:ged} 
  Let $\bx,\by\in\Adm(\uQ)$ and suppose that for
  $(j,i)\in H_0(\bx,\by)$ we have
  \[{\widehat{y}^+}(j,\rho)\preceq\widehat{x}^-(i,\rho) \text{ for both }
    \rho\in\{-1,+1\},
  \]  
  then we have also
  \[ \widebarbar{\by}(j,\rho)\leq \widebarbar{\bx}(i,\rho) \text{ for both }
    \rho\in\{-1,+1\}.
   \] 
In particular, in this situation $(j,i)$ belongs to a (unique) 
long h-line of $H_\uQ(\bx, \by)$.  
\end{Lem} 

\begin{proof}
For the convenience of the reader, we work out in detail  the most interesting case. The idea is exactly the same as in~\cite[Sec.~7.6]{Ge99}.  Let us assume that 
$\widehat{y}^+(j,\rho)\prec\widehat{x}^-(i,\rho)$  for both $\rho\in\{-1,+1\}$,
and that $(j,i)$ does not belong to a generalized diagonal of $H_\uQ(\bx,\by)$.
Let us assume furthermore that we have factorizations
\begin{align*}
    \hbx^-(i,+1) &= \bz_+\eps_p^{-1}\hbx_+, & \hbx^-(i,-1)&=\bz_-\eps_q^{-1}\hbx_-,\\
    \hby^+(j,+1) &= \bz_+\eps_p\hby_+, & \hby^+(j,-1)&=\bz_-\eps_q\hby_-,\\    
\end{align*}
for certain $\eps_p,\eps_q\in\Qsp$.  This means that the pairs 
$(\eps_p, \eps_p^{-1})$ and $(\eps_q,\eps_q^{-1})$ both correspond to an 
arrow in $H_{1,o}(\bx,\by)\cup H_{1,\times}(\bx,\by)$.
We have then also factorizations
\begin{align*}
    \bbbx(i,+1) &= \bbz_+\eps_p^{-1}\bbx_+, & \bbbx(i,-1)&=\bbz_-\eps_q^{*}\bbx_-,\\
    \bbby(j,+1) &= \bbz_+\eps_p^*\bby_+, & \bbby(j,-1)&=\bbz_-\eps_q^*\bby_-,\\    
\end{align*}
with $l(\bz_\pm)=l(\widebar{\bz}_\pm)$.
Thus, using the definitions, we obtain the following inequalities:
\begin{align*}
\bbz_+^{-1}\bbz_-\eps^*_q\bbx_- &\leq \bbx_+,&
\bbz_-^{-1}\bbz_+\eps^*_p\bbx_+ &\leq \bbx_-,\\
\bbz_+^{-1}\bbz_-\eps^*_q\bby_- &\geq \bby_+&
\bbz_-^{-1}\bbz_+\eps^*_p\bby_+ &\geq \bby_-.\\
\end{align*}
Now, if we suppose, for example $\bbbx(i,+1)<\bbby(j,+1)$, or equivalently
$\bbx_+<\bby_+$, we get the following 
chain of inequalities:
\begin{multline} \label{eq:chain}
(\bbz_+^{-1}\bbz_-\eps_q^*)\bbz_-^{-1}\bbz_+\eps_p^*\bbx_+\leq(\bbz_+^{-1}\bbz_-\eps_q^*)\bbx_-\leq\bbx_+<\\
<\bby_+\leq(\bbz_+^{-1}\bbz_-\eps_q^*)\bby_-\leq (\bbz_+^{-1}\bbz_-\eps_q^*)\bbz_-^{-1}\bbz_+\eps^*_p\bby_+.
\end{multline} 
In other words, with $\bz:= (\bbz_+^{-1}\bbz_-)\eps_q^*(\bbz_+^{-1}\bbz_-)^{-1}\eps_p^*$ we have
\[
\bbz\bbx_+\leq\bbx_+<\bby_+\leq\bbz\bby_+,
\]
which is absurd, since this implies obviously 
\[
\Del(\bbz\bbx_+,\bbz\bby_+)=l(\bz)+\Del(\bbx_+,\bby_+).
\]
This is impossible in the lexicographic order. 
Note,  that $\bbbx(i,+1)<\bbby(j,+1)$, via equation~\eqref{eq:chain},
implies $\bbbx(i,-1)<\bbby(j,-1) \Iff \bbx_-<\bby_-$.
\end{proof}

\noindent Note, that the last claim of the above Lemma follows from Lemma~\ref{lem:hl-basic}~(a).

For a connected component $L$ of $H_\uQ(\bx,\by)$ let
\[
\cC(L):= \{L'\subset H^+_\uQ(\bx,\by)\text{ connected component }
\mid L'_0\subset L_0\}.
\]
We consider on $\cC(L)$ the relation 
\[
L'> L'' \iff \text{ there exists } 
\lam\in H_{1,o}(\bx,\by) \cup H_{1,\times}(\bx,\by) \text{ with }
s(\lam)\in L' \text{ and } t(\lam)\in L''.
\]
By a slight abuse of notation, we write also $>$ for the 
transitive relation generated by $>$. It is easy to see, that
$\cC(L)=(\cC(L), >)$ is a poset, and moreover the Hasse diagram 
of $\cC(L)$ is of type $\sfA_{\abs{\cC(L)}}$. In fact, if 
$L$ is a generalized diagonal, we have $\cC(L)=\{L\}$.  Else, 
$L$ is of type $\sfA$ or of type $\sfD'$ and our claim is easy
to verify.  

\begin{Lem} \label{lem:1max}
If $L$ is a long h-line, the poset $\cC(L)$ has a unique maximal
element $L^+$.  Since the Hasse diagram of $\cC(L)$ is of type 
$\sfA$, this poset has at most two minimal elements. 
\end{Lem}

\begin{proof}
In view of the above discussion, we only have to show the first claim.  Suppose, that our poset  $\cC(L)$   has at least two maximal elements.  Since the Hasse diagram of $\cC(L)$ is of type
$\sfA$, this implies that there exists a minimal element $L^-\in\cC(L)$, which is the endpoint of two different arrows in
the Hasse diagram.  This means, in view of Table~\ref{tab:arr-let},
that for each vertex $(j,i)\in L^-_0$ we have
\[
\hby^-(j,\kap)\succ \hbx^+(i,\kap) \text{ for any } \kap\in\{-1,+1\}.
\]
This implies by Lemma~\ref{lem:ged}
\[
\widebarbar{\by}(j,\kap) \geq \widebarbar{\bx}(i,\kap) \text{ for any } \kap\in\{-1,+1\}.
\]
Since $L^-\subset L$, which is a long h-line, we conclude from
Lemma~\ref{lem:hl-basic}~(a) that $L$ is a generalized diagonal.
However, in this case $\abs{\cC(L)}=1$, a contradiction.
\end{proof}

This allows us to prove the following  result.

\begin{Prop} \label{prp:bij-hl}
  Let $\uQ$ be a skewed-gentle polarized quiver and $\bx,\by\in\Adm(\uQ)$.  Then
  the decorated quiver $H_\uQ(\bx,\by)$ has the following properties:
\begin{itemize}
\item[(a)]
The map $\cH_\uQ(\bx,\by)\ra\cR_\uQ(\bx,\by), L\mapsto L^+$
is bijective.  In particular, each real line is contained
in a unique long h-line.  Moreover, this bijection preserves
the type of the long resp. real h-lines. 
\item[(b)]   
A connected component $H$ of $H^+_\uQ(\bx,\by)$ is a real 
h-line if for some vertex $(j,i)\in H_0$ we have
\begin{equation} \label{eq:rhl2}
    \hby^\del(j,\kap)\prec\hbx^\del(i,\kap) \text{ for all }
(\del,\kap)\in\{-,+\}\times\{-1,+1\}.
\end{equation}
Moreover, the validity of condition~\eqref{eq:rhl2} is  
independent of the choice of the vertex $(j,i)\in H_0$. 
\item[(c)]
A connected component $H$ of $H^+_\uQ(\bx,\by)$ is a h-line,
if and only if for some vertex $(j,i)\in H_0$ we have
\begin{equation} \label{eq:hl2}
   \hby^+(j,\kap)\preceq\hbx^-(i,\kap) \text{ for all }
   \kap\in\{-1,+1\}.
\end{equation}
Moreover, the validity of condition~\eqref{eq:hl2} is  
independent of the choice of the vertex $(j,i)\in H_0$. 
\end{itemize}
\end{Prop}

\begin{proof}
 (a) If $L$ is a long h-line, by definition it has no red
 vertex.  Moreover, by construction, no arrow from 
 $H_{1,o}(\bx,\by)\cup H_{1,\times}(\bx,\by)$ ends in a vertex of 
 the maximal element $L^+$ of $\cC(L)$.  Thus, $L^+$ is indeed
 a real h-line.  Since $\cC(L)$ contains, by Lemma~\ref{lem:1max},
a unique maximal element $L^+$, we have indeed a well-defined
map $\cH_\uQ(\bx,\by)\ra\cR_\uQ(\bx,\by), L\mapsto L^+$.  
This map is obviously injective, since each connected component
of $H^+_\uQ(\bx,\by)$ is, by construction, contained in a unique
connected component of $H_\uQ(\bx,\by)$. If in turn $H$ is a
real h-line, it follows from Lemma~\ref{lem:hl-basic}~(c) and
equation~\eqref{eq:obspm} that for all $(j,i)\in H_0$ we have
$\hby^+(j,\kap)\succ \hbx^-(i,\kap)$ for all $\kap\in\{-1,+1\}$. 
By Lemma~\ref{lem:ged} follows that $(j,i)$ belongs to (unique) long h-line
by Lemma~\ref{lem:hl-basic}~(a). Thus, our map is also surjective.

(b)  If $H$ is a real h-line, we have $H=L^+$ for a unique 
long h-line $L$ by~(a). In view of the description of $\cC(L)$
it is now straightforward to check, with Table~\ref{tab:arr-let}
and the definitions, that condition~\eqref{eq:rhl2} holds
for all vertices $(j,i)\in H_0$. Conversely, if condition~\eqref{eq:rhl2} holds for some $(j,i)\in H_0$, in view
of Table~\ref{tab:arr-let}, no arrow from 
$H_{1,o}(\bx,\by)\cup H_{1,\times}(\bx,\by)$ can end in an extremal
vertex of $H$. Thus, $H$ has no orange or purple vertex.
Since condition~\eqref{eq:rhl2} implies moreover,
by  equation~\eqref{eq:obspm} and Lemma~\ref{lem:ged}, that
$H$ is contained in a long h-line. Thus, $H$ has no red vertex, either.  

(c) Let $H\subset H^+_\uQ(\bx,\by)$ be a connected component. 
If $H$ is a  h-line, it is by definition contained in a $\tH$-line.
Thus,  by Lemma~\ref{lem:hl-basic}~(b),
condition~\eqref{eq:hl2} is fulfilled for each $(j,i)\in H_0$.
Conversely, if condition~\eqref{eq:hl2} is fulfilled for some
$(j,i)\in H_0$, we consider the unique connected component 
$\widetilde{H}$ of $H^o_\uQ(\bx,\by)$ which contains $H$. 
By Lemma~\ref{lem:ged}, $\widetilde{H}$ is contained in a long
h-line $L$. Thus, we can consider  the poset 
$\cC(\widetilde{H})$, which consists of the connected components
$H'$ of $H^+_\uQ(\bx,\by)$, which are contained in $\widetilde{H}$.
Note that $\cC(\widetilde{H})\subset\cC(L)$.  Thus, by 
Lemma~\ref{lem:1max}, $\cC(\widetilde{H})$ contains a unique 
maximal element $H^+> H$. In fact, $H^+=L^+$, since otherwise a boundary vertex of $H^+$ would be the endpoint of an arrow from
$H_{1,\times}(\bx,\by)$, and thus in view of Table~\ref{tab:arr-let}
we would have $\hby^+(j,\kap)\succ\hbx^-(i,\kap)$ for some
$\kap\in\{-1,+1\}$ and $(i,i)\in H_0$, a contradiction.
Now, since $L^+\subset\widetilde{H}\subset L$ it is clear that
$\widetilde{H}$ is a $\tH$-line, and thus $H$ is a h-line.
\end{proof}

\begin{Rem}
(1) Note, that by Proposition~\ref{prp:bij-hl}~(c), our notion of h-lines
is equivalent to the original Definition in~\cite[Sec.~3.3]{Ge99}.  Moreover,  this result implies a characterization of $\tH$-lines, which
is slightly sharper than Lemma~\ref{lem:hl-basic}~(b): It is in fact
sufficient to verify condition~\eqref{eq:hl} for just one vertex
$(j,i)\in L_0$.

(2) By comparing parts~(b) and~(c) of Proposition~\ref{prp:bij-hl},
we see, with the help of Equation~\eqref{eq:obspm}, that real h-lines
are in particular h-lines. 

(3) In Lemma~\ref{lem:1max} we may replace the long h-line $L$ by a $\tH$-line $\tilde{L}$ and arrive 
with essentially the same argument to the same conclusion. In particular, each $\tH$-line $\tilde{L}$
contains a unique real $h$-line. Thus, in view of (2), we have also a bijection (via inclusion)
between the set of real $h$-lines and the set of $\tH$-lines.  However, we do not need this fact here. 
\end{Rem}

\subsection{Dual h-lines} \label{ssec:dualh}
Given $\bx,\by\in\Adm(\uQ)$, for the calculation of E-invariant, it will be
convenient to read off from $H_\uQ(\bx,\by)$ the different types of
h-lines in  $H_\uQ(\by,\bx)$.
To this purpose we need to introduce a further set of color labels for
the vertices of $H_\uQ(\bx,\by)$, namely by the colors blue, cyan and teal.
\begin{align*}
\blux{\Hblu}  &:= \{(j,i)\in H_0(\bx,\by)\mid
           \widebarbar{\by}(j,\rho)<\widebarbar{\bx}(i,\rho) \text{ and }
           \Del(\widebarbar{\by}(j,\rho),\widebarbar{\bx}(i,\rho))=0\\
       & \hspace*{9cm} \text{ for some }\rho\in\{-1,1\}\},\\
\cyax{\Hcya} &:= \{ s(\nu,\mu)_\times \mid (\mu,\nu)_\times\in H_{1,\times}(\bx,\by)\},\\
\teax{\Htea}  &:= \{ s(\nu,\mu)_* \mid (\mu,\nu)_*\in H_{1,o}(\bx,\by)\}.\\
\end{align*}
Note that a blue vertex can also receive a cyan label, or a teal label.
However, a given vertex can only receive one of the labels orange, purple, cyan and teal. 
We note also, that in $H_\uQ^+(\bx,\by):=(H_0(\bx,\by), H_{1,+}(\bx))$
all vertices $(j,i)\in H_0(\bx,\by)$ of valency 1, where at least one of $j$
and $i$ is not a boundary vertex, have now assigned at least one color label.
Only the isolated vertices of $H_\uQ(\bx,\by)$  can have assigned 2
(possibly  the same) color labels from the colors blue and red.
We agree that the vertices, which have received no color label in the above sense, are displayed as black vertices.  

A connected component of the quiver  $H_\uQ^+(\bx,\by)$,
which contains no blue, cyan or teal  vertex, is a \emph{dual real h-line}. 
A connected component of the quiver  $H_\uQ^+(\bx,\by)$, such
that its connected component in
$(H_0(\bx,\by), H_{1,+}(\bx,\by)\cup H_{1,o}(\bx,\by))$ contains no blue or cyan
vertex, is a \emph{dual h-line}.
A connected component of the quiver $H(\bx,\by)$, which contains no blue
vertex, is a \emph{dual long h-line}.

A dual real h-line $L\in\cH_\uQ(\bx,\by)$, 
is a \emph{dual kiss} if none of its (at most 2) endpoints is a boundary point
of $H_\uQ(\bx,\by)$ (see end of Section~\ref{ssec:Hvw}).

\begin{Rem} \label{rem:dual}
(1) It is easy to see that $(j,i)\in H_0(\bx,\by)$ lies on a dual (real, long)
h-line, resp.~on a dual kiss if and only if $(i,j)\in H_0(\by,\bx)$ lies on a
(real, long) h-line resp. a kiss.

(2) Each loop of $H_Q(\bx,\by)$ belongs to a real h-line, or to a dual real
h-line.  Indeed, if the connected component of a loop is not a generalized
diagonal, then its unique endpoint must have at least one color label.
In particular, a real h-line of type $\sfD'$ is a kiss except if it is a diagonal,
or an anti-diagonal.  

(3)   From the definitions it is  clear,
that the kisses are precisely the components, where each endpoint has one
of the labels blue, cyan or teal, and where moreover vertices with blue label
are not boundary vertices.
By replacing blue by red, cyan by orange, and teal by purple, we obtain a
corresponding characterization of the dual kisses.  
\end{Rem}

\subsection{Example} \label{ssec:exHxy}
Consider the polarized quiver
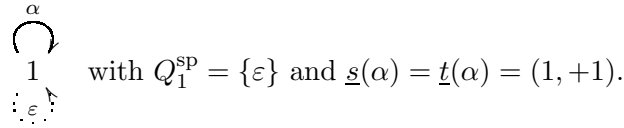
\[
\xymatrix{ 1 \ar@(lu,ru)[]^{\alp}\ar@{.>}@(ld,rd)[]^{\eps}}
\quad\text{ with } \Qsp=\{\eps\} \text{ and } \us(\alp)=\ut(\alp)=(1,+1). \]
This is ``almost'' a skewed-gentle polarized quiver, except for the boundedness.
It makes perfect sense to consider for this quiver admissible strings and bands,
and we can present, already for relatively short words, an interesting quiver
$H_\uQ(\bx,\by)$. Consider
$\bx=\II_{1,1}^{-1}\eps\alp\eps\alp^{-1}\II_{1,-1}$, which is admissible of
type~(u,p), and
$\by=\II_{1,-1}^{-1}\alp\eps\alp\eps\alp^{-1}\II_{1,-1}$, which is admissible
of type (p,p). We display the quivers $H(\bx), H(\by)$ and the decorated quiver
$H_\uQ(\bx,\by)$ with all the colored labels.  Observe, that this is essentially
the same situation as the one, which was considered in~\cite[Sec.~3.4]{Ge99}.
Concretely, in the diagram below,
the two arrows of the form $(a, b)_{\ominus}$ are dashed and depicted vertically, the four of the form $(a, b)_\oplus$ are dashed and depicted horizontally, and these six arrows go from teal-decorated vertices to purple-decorated vertices. 
The four arrows of the form $(a, b)_\times$ are dashed and are depicted diagonally, and
by definition go from cyan-decorated vertices to orange-decorated vertices.
The remaining arrows, none of which are dashed, are of the form $(a, b)_+$ and traverse vertices $(j, i)$ that are either black, red or blue determined by comparing the arrows, of a given signature, incident at $i$ in $H(\bx)$ with those incident at $j$ in $H(\by)$.

\[\def\objectstyle{\scriptstyle}\xymatrix{
                   &           &1  &\ar[l]_{\eps}           2&\ar[l]_{\alp}3 &%
\ar[l]_{\eps}\ar[r]^{\alp}         4&\ar@{.>}@(ur,dr)[]5\\
\ar@{.>}@(ul,ur)[] 1&           &\pbPT &\ar@{.>}[l]        \tPT&        \pbPT &%
\ar@{.>}[l]                \tbPT&\PT\ar@(u,r)[]\\
\ar[u]^{\alp}       2&           &\rPT&\ar@{.>}[ld]\crPT&\ar[ul]  \PT&%
\ar@{.>}[ld]\ar[ur]\cPT&\trPT\ar@{.>}[d]\\
\ar[u]^{\eps}       3&           &\obPT  &\ar[ul]             \PT&      \obPT &%
\ar[ul]                     \bPT& \pPT\\
\ar[u]^{\alp}       4&           &\rPT &\ar@{.>}[dl]\crPT&\ar[ul]  \PT&%
\ar@{.>}[dl]\ar[ur]\cPT& \trPT\ar@{.>}[d]\\
\ar[u]^{\eps}\ar[d]_{\alp}5&       &\orPT&\ar[ul]        \rPT&\ar[dl]\oPT&%
\ar[ul]\ar[dr]              \PT&\prPT\\
\ar@{.>}@(dr,dl)[] 6&         &\pbPT&\ar@{.>}[l]         \tPT&       \pbPT &%
\ar@{.>}[l]                 \tbPT&\PT\ar@(r,d)[]\\
\phantom{X}\\
}\]
The sets $L^{(1)}_0=\{(1,4)\}, L^{(2)}_0=\{(3,4), (2,3), (1,2)\}$ and $L^{(3)}_0=\{(6,4)\}$ 
are the underlying sets of 3 different real h-lines of type $\sfA$.  Note, that $L^{(2)}_0\cup\{(1,1)\}$
is the underlying set of a $\tH$-line, which is at the same time a long h-line. Thus, $\{(1,1)\}$ is
a h-line. On the other hand,  $L^{(4)}_0=\{(2,4), (1,5)\}$ is the underlying set of  real h-line of type $\sfD'$,
and $L^{(4)}_0\cup\{(3,3)\}$ is the underlying set of a long h-line which is \emph{not} a $\tH$-line.

We leave it as an exercise to identify the different dual real h-lines. 

\subsection{Triples and h-lines} \label{ssec:triph}
In~\cite{Kr91} Krause introduced admissible triples as a framework for
the combinatorial description of homomorphisms between representations
of gentle algebras. We adapt his construction here to the setting of
skewed-gentle polarized quivers, and relate it to our construction in
terms of the quiver $H_\uQ(\bx,\by)$.

A skewed-gentle polarized quiver $\uH$ is called \emph{hereditary} if the
following two conditions are fulfilled:
\begin{itemize}
\item
  The underlying graph of $\uH$ is one of the four types in
  Table~\ref{tab:AD-graphs} and all loops are special.  
\item
  If $\alp,\bet\in H_1$ are two different arrows with $s(\alp)=t(\bet)$,
  then $s_1(\alp)=-t_1(\bet)$.
\end{itemize}
In Section~\ref{ssec:skga} we will assign to each skewed-gentle polarized
quiver $\uQ$ a skewed-gentle $\Ka$-algebra $\Ka\uQ$. This algebra is hereditary
precisely, if $\uQ$ is hereditary in the above sense.

For a fixed, skewed-gentle polarized quiver $\uQ$, we are interested in the
category of homomorphisms $G\df\uH\ra\uQ$ where $\uH$ is hereditary.
By definition, a \emph{morphism} $\phi\df G\ra G'$ between homomorphisms
$G\df\uH\ra\uQ$ and $G'\df\uH'\ra\uQ$ is given by a strict homomorphism
$\phi\df\uH\ra\uH'$ of skewed-gentle  polarized quivers, such that
$G'\phi=G$.  Recall, that a homomorphism between polarized quivers is
\emph{strict}, if it sends ordinary arrows to ordinary arrows, see the
definitions in Section~\ref{ssec:polar}.

Note, that if in this situation $\uH$ is of type $\sftA$
or of type  $\sftD'$, then $\phi$ is surjective, and  $\uH'$ is of the
same type as $\uH$.

A \emph{winding} of the skewed-gentle quiver $\uQ$ is a homomorphism
$G\df\uH\ra\uQ$ of skewed-gentle polarized quivers, where $\uH$ is hereditary.
If $\uH$ is of type $\sftA$ or of type $\sftD'$, we request furthermore
that any morphism of windings, which starts in $G$, is an isomorphism.
In view of Remark~\ref{rem:polarized},  our definition of windings is
compatible  with the one from~\cite{Kr91}.

In particular, the morphisms $G_\bx\df \uH(\bx)\ra\uQ$ for $\bx\in\Adm(\uQ)$
from Section~\ref{ssec:Hw} are windings.  Observe that a winding may send an
ordinary arrow to a special loop, whilst the morphisms between windings
send ordinary arrows to ordinary arrows.

We consider the following properties (which may or may not hold) of a morphism  $\phi\df G\ra G'$ of windings $G\df\uH\ra\uQ$ and
$G'\df\uH'\ra\uQ$.  
\begin{itemize}
\item[(q)] If $j\in H_0$ is a boundary vertex 
  and $\alp'\in (H'_1)^\ord$ with $t(\alp')=\phi(j)$, then there exists an
  arrow $\alp$ in $\uH$ with $\phi(\alp)=\alp'$.
\item[(s)] If $j\in H_0$ is a boundary vertex 
  and $\bet'\in (H'_1)^\ord$ with $s(\bet')=\phi(j)$, then there exists an
  arrow $\bet$ in $\uH$ with $\phi(\bet)=\bet'$.
\item[(K)]
  If $j\in H_0$ is a boundary vertex of $\uH$, then $\phi(j)\in H'_0$
  is \emph{not} a boundary vertex of $\uH'$.
\end{itemize}

\begin{Rem} \label{rem:q+s}
Let $\phi\df\uH\ra\uH'$ be a strict morphism between hereditary, skewed-gentle polarized quivers, and let $h\in H_0$ be a vertex, \emph{which is not a boundary}.  
Then it is an easy exercise to see, that $\phi$ induces  bijections
\begin{align*}
\{\alp\in H_1\mid s(\alp)=h\} &\ra\{\alp'\in H'_1\mid s'(\alp')=\phi(h)\}\text{ and}\\
\{\bet\in H_1\mid t(\bet)=h\} &\ra\{\bet'\in H'_1\mid t'(\bet')=\phi(h)\}.
\end{align*}
Thus, the above conditions (q) and (s) are equivalent to the following, apparently stronger conditions
\begin{itemize}
\item[(q')] If $j\in H_0$ is \emph{any} vertex 
  and $\alp'\in (H'_1)^\ord$ with $t(\alp')=\phi(j)$, then there exists an
  arrow $\alp$ in $\uH$ with $\phi(\alp)=\alp'$.
\item[(s')] If $j\in H_0$ is \emph{any} vertex 
  and $\bet'\in (H'_1)^\ord$ with $s(\bet')=\phi(j)$, then there exists an
  arrow $\bet$ in $\uH$ with $\phi(\bet)=\bet'$.
\end{itemize}
It follows, that the composition of two morphisms with property (q) has again property
(q), and the same holds if we replace (q) by (s).  
\end{Rem}

\begin{Def}\label{def:triples}
For $\bx,\by\in\Adm(\uQ)$ a \emph{H-triple} $(G,\phi_q,\phi_s)$ for $(\bx,\by)$
consists of a winding $G\df\uH(G)\ra\uQ$, together with a  morphism
$\phi_q\df G\ra G_\bx$ with property (q), and a  morphism 
$\phi_s\df G\ra G_\by$ with property (s). 
Moreover, we request for $j\in H_0(G)$ with $\phi_q(j)\in H_0^\spe(\bx)$
\emph{and} $\phi_s(j)\in H_0^\spe(\by)$, that $j\in H_0^\spe(G)$.  
Recall, that by definition $\phi_q$ and $\phi_s$ are strict homomorphisms of polarized
quivers.  
An H-triple $(G,\phi_q,\phi_s)$ is a \emph{K-triple}, if $\phi_q$
and $\phi_s$ have moreover property (K).

Two H-triples $(G,\phi_q,\phi_s)$ and $(G',\phi'_q,\phi'_s)$ are
\emph{equivalent}, if there is an isomorphism of windings $\psi\df G\ra G'$
such that $\phi_q=\phi'_q\psi$ and $\phi_s=\phi'_s\psi$.  In other words,
in the following diagram all triangles must be commutative. 
\[
  \xymatrix{&\uH(G)\ar_{\phi_q}[ld]\ar^G[d]\ar^{\phi_s}[rd]\ar@/_1.2pc/_(.62){\psi}[dd]\\
    \uH(\bx)\ar^{G_\bx}[r] &\uQ &\ar_{G_\by}[l] \uH(\by)\\
   &\uH(G')\ar^{\phi'_q}[lu]\ar_{G'}[u]\ar_{\phi'_s}[ru] }
\]
We denote the set of equivalence classes $[(G,\phi_q,\phi_s)]$ of H-triples
for $(\bx,\by)$ by $\cT_\uQ(\bx,\by)$. The subset of classes of K-triples is
denoted by $\cK_\uQ(\bx,\by)$.

The \emph{type} of  $(G,\phi_q,\phi_s)\in\cT(\bx,\by)$ 
is the \emph{type} of its underlying graph
of $\uH(G)$. Thus, the type of $(G,\phi_q,\phi_s)$ belongs
to the list from Table~\ref{tab:AD-graphs}.
\end{Def}

On the other hand, we define
$L(G,\phi_q,\phi_s)$ to be the following subquiver of
$H_\uQ^+(\bx,\by)$: It consists of the vertices
$(\phi_s(h),\phi_q(h))$ for $h\in H_0(G)$ and has
an arrow
\[
  (\phi_s(\alp),\phi_q(\alp))\df (\phi_s(s(\alp)),\phi_q(s(\alp)))\ra
  (\phi_s(t(\alp)),\phi_q(t(\alp)))
\]  
for each arrow $\alp$ of $\uH$. In particular, as a quiver, 
$L(G,\phi_q,\phi_s)$ is isomorphic to $\uH$.

Note that the connected subquiver $L:=L(G,\phi_q,\phi_s)\cong H(G)$ of 
$H_\uQ(\bx,\by)$ contains no arrows from $H_{1,o}(\bx,\by)$ since the morphisms 
$\phi_q$ and $\phi_s$ are strict. 
Next, $L$ contains no arrow from 
$H_{1,\times}(\bx,\by)$ in view of the “crossed" definition of those arrows, 
and we can view $\phi_s$ as a (projection) morphism $L\ra\uH(\by)$. 
Thus, $L$ is indeed a connected subquiver of $H^+_\uQ(\bx,\by)$. 

\begin{Prop} \label{prp:trip-line}
  The map $L$ induces a bijection
\[
  \cT_\uQ(\bx,\by)\ra\cR_\uQ(\bx,\by),\quad
  [(G,\phi_q,\phi_s)]\mapsto L(G,\phi_q,\phi_s),
\]  
from the equivalence classes of H-triples for $(\bx,\by)$ to the real 
h-lines of
$H_\uQ(\bx,\by)$. 
Under this correspondence, K-triples correspond
precisely to kisses.
\end{Prop}

\begin{proof}
  Clearly, $L(G,\phi_q,\phi_s)=L(G',\phi'_q,\phi'_s)$ 
  if $(G',\phi'_q,\phi'_s)$ is equivalent to $(G,\phi_q,\phi_s)$. 
  Moreover, for
  $(G,\phi_q,\phi_s)\in\cT_\uQ(\bx,\by)$ the connected subquiver 
  $L:=L(G,\phi_q,\phi_s)\subset H_\uQ^+(\bx,\by)$ is a connected component. 
  Otherwise, for a boundary vertex $b$ of $H(G)$, there exists an arrow 
  $(\bet',\alp')_+\in H_{1,+}\setminus L_1$ which starts or stops in
  $(\phi_s(b), \phi_q(b))\in L_0$.  In the first case, this means that for the arrow
  $\bet'\in H_1(\by)$ we have $t(\bet')=\phi_s(b)$, but $\bet'$ is not in the image
  of $\phi_s$. This contradicts the property (s) of $\phi_s$.  
  In case $(\bet',\alp')$ terminates in $(\phi_s(b), \phi_q(b))$, we find similarly a contradiction to the property $(q)$ of $\phi_q$.

Next, in this case $L$ contains no red, orange or purple vertex, 
due to the properties (q) resp.~(s) of $\phi_q$ resp.~$\phi_s$.
In fact, suppose, for example,  that $L$ contains an orange vertex 
$l:=(\phi_s(b),\phi_q(b))$.  
By the above discussion, we see that $b$ is a boundary vertex of $\uH(G)$, and
$l$ is the terminal point $(s(\nu), t(\mu))$ of an arrow 
$(\nu,\mu)_\times\in H_{1,\times}(\bx,\by)$, which in particular does not belong to $L$.
This means, that the arrow $\nu\in H_1(\bx)$ is not
in the image of $\phi_q$, but $t(\nu)=\phi_q(b)$, contradicting the property (q)
of $\phi_q$. In fact, we obtain in this case also a contraction to the property
(s) of $\phi_s$.  Similarly, if $L$ contains a purple vertex, this vertex must be the endpoint of an arrow from $H_{1,o}(\bx,\by)$. In case this arrow is of the form
$(\eta,\mu)_\oplus$, we obtain as above a contradiction to the property (q) of $\phi_q$.
If the arrow is of the form $(\nu,\eta)_\ominus$, we obtain a contradiction to the
property (s) of $\phi_s$. We leave it as an exercise, that $L$ doesn't contain red vertices either.   Thus, our map $\cT_\uQ(\bx,\by)\ra\cR_\uQ(\bx,\by)$ is well-defined.

  It is easy to see that the following construction yields an inverse to $L$.
  If $L$ is a real h-line, we can consider it as a quiver, together with
  projections $\pi_s\df L\ra\uH(\by)$ and $\pi_q\df L\ra\uH(\bx)$.
  For example, $\pi_s(j,i)=j$ and $\pi_s(\nu,\mu)=\nu$.  By the construction
  of $H_\uQ(\bx,\by)$ we have then $G_\bx\pi_q=G_\by\pi_s=:G$, and
  $(G,\pi_q,\pi_s)$ becomes an H-triple with $G\df L\ra\uQ$ a winding, where
  we can equip $L$ with the canonical polarization compatible with $G$.
\end{proof}  

\begin{Expl}  \label{expl:TripS}
  Let $\bs_i, \bx\in\Adm(\uQ)$, where $\bs_i:=\II_{i,+}^{-1}\II_{i,-}$
  for some $i\in Q_0$. If $i\in\Qordv$, all (real, possibly dual)
  h-lines in $H_\uQ(\bs_i,\bx)$ are of type $\sfA_1$.
  In this situation, the equivalence classes H-triples for 
  $(\bx, \bs_i)$ are (via projection) in bijection with
  morphisms $G_{\bs_i}\ra G_\bx$, which have property (q). The reason is,
  that  $H(\bs_i)=H$ consists of a single vertex $1$, and thus the identity
  is the unique morphism $\phi_s$ with property (s).  
  For the same reason we don't have to deal with equivalence classes.
  The just mentioned morphisms with property (q)  are obviously given by the
  sources $j$ of $\uH(\bx)$ with
  $G_\bx(j)=i$. 

  For $i\in\Qspv$ the quiver $\uH(\bs_i)$ is of type $\sfD'_1$, i.e.~it has
  only one vertex $1$ with a single (special) loop $\eta$. In this case, 
  the real h-lines in $H_\uQ(\bx,\bs_i)$ are of type $\sfA_1$ or of type 
  $\sfD'_1$.
  All real h-lines of type $\sfA_1$ are contained in a long h-line of type
  $\sfA_2$.  
  Similar to the previous case, the real h-lines of type
  $\sfA_1$ are in bijection with morphisms 
  $\phi_q\df G\ra G_\bx$ with property $(q)$, 
  where $G\df \{1\}\ra\uQ$ is given by $1\mapsto i$. In this situation,
  the map $\phi_s\df \{1\}\ra\uH(\bs_i)$ is trivially given by the inclusion.
  The just mentioned morphisms $\phi_q$ are also in bijection with the sources
  $j$ of $\uH(\bx)$ with $G_\bx(j)=i$.

  On the other hand, the h-lines of type $\sfD'_1$ are in bijection with
  the morphisms of windings 
  $\phi_q\df\uH(\bs_i)\ra\uH(\bx)$ which have property (q).
  Those are also in bijection with the \emph{special} vertices $j$ of $\uH(\bx)$
  where no ordinary arrow starts, and moreover $G_\bx(j)=i$.  Thus, there are
  at most two h-lines of type $\sfD'$.
\end{Expl}

\subsection{Fringing} \label{ssec:fring}
Let $\uQ$ be a skewed-gentle polarized quiver.
We extend and slightly generalize the notion of fringing
from~\cite[Sec.~3]{BY20} to our setting of skewed-gentle polarized quivers.

\begin{Def}
Let $\uQ\subset\uQ^f$ be an inclusion  of  skewed-gentle polarized
quivers. We say that $\uQ^f$ is
a \emph{fringing} of $\uQ$ if the following two conditions hold:
\begin{itemize}
\item[(1)]
For each interior vertex  $i\in\uQ_0\subset\uQ^f_0$ we have
\[
  \abs{\{\alp\in\uQ^f_1\mid s(\alp)=i\}}+\abs{\{\bet\in\uQ^f_1\mid t(\bet)=i\}}= 4.\]  
\item[(2)]
 For each \emph{fringe arrow} $\alp\in\uQ^f_1\setminus\uQ_1$ we have:
 $s(\alp)\in \uQ_0$ implies $\ut(\alp)\in (\uQ_0^f\setminus \uQ_0)\times\{+1\}$,
 and $s(\alp)\in\uQ^f_0\setminus\uQ_0$ implies $s_1(\alp)=+1$ and
  $t(\alp)\in Q_0$.  
\end{itemize}  
\end{Def}
In particular, each fringe arrow connects a \emph{fringe vertex}
$j\in Q^f_0\setminus Q_0$  with an interior vertex. 
Moreover, each fringe vertex has valency at most 2, since at a fringe vertex
the polarization is, by definition, 
always $+1$.  It is clear that each skewed-gentle polarized quiver admits a
fringing, however our fringings are  not unique. For example,
in the following display, two non-isomorphic fringings of the same
skewed-gentle polarized quiver $\uQ$ with $Q_0=\{1,2,3\}$ and $\Qsp=\{\eps\}$,
from Section~\ref{ssec:ExSt}, are exhibited.
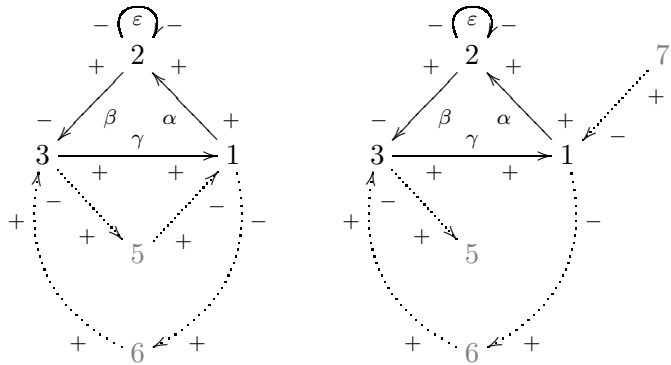
\[
\xymatrix{&\ar[ld]^(.45)\bet_(.8)-_(.25)+ 2\ar@(ul,ur)[]_{\eps}^(.2)-^(.8)- \\
    3\ar[rr]^\gam_(.3)+_(.7)+ \ar@{.>}[rd]_(.3)-_(.65)+
    && 1\ar[lu]^\alp_(.2)+_(.75)+\ar@{.>}@/^1.5pc/[ldd]^(.25)-^(.8)+\\
    & \grex{5}\ar@{.>}[ru]_(.3)+_(.65)-\\
    & \grex{6}\ar@{.>}@/^1.5pc/[luu]^(.75)+^(.2)+
  }\qquad
  \xymatrix{&\ar[ld]^(.45)\bet_(.8)-_(.25)+ 2\ar@(ul,ur)[]_{\eps}^(.2)-^(.8)-
    &&\ar@{.>}[ld]^(.25)+^(.65)- \grex{7}\\
    3\ar[rr]^\gam_(.3)+_(.7)+ \ar@{.>}[rd]_(.3)-_(.65)+
    && 1\ar[lu]^\alp_(.2)+_(.75)+\ar@{.>}@/^1.5pc/[ldd]^(.25)-^(.8)+\\
    & \grex{5}\\
    & \grex{6}\ar@{.>}@/^1.5pc/[luu]^(.75)+^(.2)+
  }
\]    
Note, that the operator $\tau_{\uQ^f}$ from Section~\ref{ssec:AR-adm} acts
on $\bx\in\Adm_s(\uQ)$ by replacing each unpunctured (trivial) letter by
the corresponding cohook, which ends then in a fringe vertex.
As observed in~\cite{BY20}, this regular behavior has several advantages.

\subsection{Kisses, K-triples and fringing} \label{ssec:kissH}
Let $\uQ\subset\uQ^f$ be a skewed-gentle polarized quiver with fringing.
Recall from Section~\ref{ssec:AR-adm} that we have an operator
$\tau=\tau_{\uQ}$ on $\Adm(\uQ)$.  Clearly $\Adm(\uQ)\subset\Adm(\uQ^f)$,
and we abbreviate $\tau_f=\tau_{\uQ^f}$ for the corresponding operator
on $\Adm(\uQ^f)$. For $\bx\in\Adm(\uQ)$ we will consider frequently
$\iota G_\bx\df\uH(\bx)\ra\uQ^f$ as a winding over $\uQ^f$,
where $\iota$ is  inclusion of $\uQ$ into $\uQ^f$.
See also Section~\ref{ssec:Hw}.

\begin{Rem} \label{rem:qs}
Let $\bx, \by\in\Adm(\uQ)$ with $\by$ not projective.
With the notation from Section~\ref{ssec:triph} we have 
injective morphisms of windings
  \[
    \pi_\bx\df G_{\bx}\ra G_{\tau_f\bx}\quad\text{and}\quad
    \sig_\by\df G_{\tau\by}\ra G_{\tau_f\by},
  \]
where $\pi_x$ has properties (q) and (K), and $\sig_\by$ has properties
(s) and (K). This follows easily from Remark~\ref{rem:Admtau}. 

In fact, the morphism  $\pi_\bx$, close to a boundary vertex $j$ of $\uH(\bx)$, may be visualized as follows:
\[\xymatrix{
\uH(\bx):&\ar@{.}[r]&\ar@{-}[rr]&\phantom{j}\ar@{^{(}->}[d]^{\pi_{\bx}}&j\\
\uH(\tau_f\bx):&\ar@{.}[r]&\ar@{-}[rr]&\phantom{A}&j\ar[r]^{\alp}&\cdot&\ar[l]_{\bet_1}&\ar@{.}[l]&\ar[l]_{\bet_f} f
}\]
In view of the only new arrow $\alp$ which starts in $j$, it is clear that the inclusion $\pi_\bx$
has properties (q) and (K). Note, that in case $\alp$ is a fringe arrow, the arrows $\bet_1,\ldots\bet_f$ do not occur. If $\alp$ is not an fringe arrow,  $\bet_f$ is a fringe arrow.

For the morphism $\sig_\by$ we have to distinguish at each boundary vertex $j$ two cases.\\  
(1) If $\uH(\tau\by)$ is obtained (locally) from $\uH(\by)$ by adding at $j$ a co-hook, we are in the following situation:
\[\xymatrix{
\uH(\tau\by):&\ar@{.}[r]&\ar@{-}[rr]&\phantom{j}\ar@{^{(}->}[d]^{\sig_{\by}}
&j\ar[r]^{\alp}&\cdot  &\ar[l]_{\bet_1}&\ar@{.}[l]&\ar[l]_{\bet_{f-1}} k\\
\uH(\tau_f\by):&\ar@{.}[r]&\ar@{-}[rr]&\phantom{i}&j\ar[r]^{\alp}&\cdot&\ar[l]_{\bet_1}&\ar@{.}[l]&\ar[l]_{\bet_{f-1}} k&\ar[l]_{\bet_f}f
}\]
In view of the unique new (fringe) arrow $\bet_f$ which terminates at $k$, the inclusion $\sig_\by$ has in this case property (s).\\
(2) If $\uH(\tau\by)$ is obtained (locally) from $\uH(\by)$ by removing at $j$ a hook, we are in the following situation:
\[\xymatrix{
\uH(\by):    &\ar@{.}[r]&\ar@{-}[r]&i&\ar[l]_{\bet'}\cdot\ar[r]^{\alp'_1}&\ar@{.}[r]&\ar[r]^{\alp'_h}&j\\
\uH(\tau\by):&\ar@{.}[r]&\ar@{^{(}->}[d]^{\sig_{\by}}\phantom{j}\!\!\ar@{-}[r]&i\\
\uH(\tau_f\by): &\ar@{.}[r]&\ar@{-}[r]\phantom{i}\!\!&i&\ar[l]_{\bet'}\cdot\ar[r]^{\alp'_1}&\ar@{.}[r]&\ar[r]^{\alp'_h}&j
\ar[r]^{\alp}&f\\
}\]
In view of the unique new  arrow $\bet'$ which terminates at $i$, the inclusion $\sig_\by$ has also in this case property (s). Note, that in this case $\alp$ is a fringe arrow. 
\end{Rem}

The following result is our adaption of~\cite[Lem.~4.2]{BY20}
to the situation of skewed-gentle polarized quivers.
We include the proof for the convenience of the reader,
though the idea is very similar to the original version.

\begin{Lem} \label{lem:cce}
  With the notation, which was  introduced above, we have the following:
\begin{itemize}
  \item[(a)] For each $\bx\in\Adm(\uQ)$,  composition with $\pi_\bx$ yields
    a bijection between the equivalence classes of morphisms of windings
    $\phi\df G\ra G_\bx$ with property (q) and morphisms of windings
    $\phi'\df G'\ra G_\bx$ with property (q) and (K). 
  \item[(b)] For each $\bx\in\Adm(\uQ)$ which is not projective,
    composition with $\sig_\bx$ yields
    a bijection between the equivalence classes of morphisms of windings
    $\phi\df G\ra G_{\tau\bx}$ with property (s) and morphisms of windings
    $\phi'\df G\ra G_{\tau_f\bx}$ with property (s) and (K). 
\end{itemize}
\end{Lem}

\begin{proof}
  The composition of morphisms of windings with property (q)
  resp.~with property (s) has again this property, see Remark~\ref{rem:q+s}.  Thus, the maps described in (a) and (b) are injective by Remark~\ref{rem:qs}. Thus,
  it remains to prove surjectivity in both situations. We may assume
  that $\bx\in\Adm_s(\uQ)$, because otherwise we have $\tau_f\bx=\bx=\tau\bx$,
  and our claim becomes trivial. For the same reason we may even assume that
  $\bx=\II^{-1}_{i,\mu}\bx'\II_{j,\nu}$ with $\II_{j,\nu}$ unpunctured. Thus
\[\tau_f\bx=\begin{cases}
    \II_{i',\mu'}^{-1}\alp_m^{-1}\cdots\alp_2^{-1}\alp_1\bx'\bet_1^{-1}\bet_2\cdots\bet_n\II_{j',\nu'} &\text{ if } \II_{i,\mu} \text{ is unpunctured,}\\
    \II_{i,\nu}^{-1}\bx'\bet_1^{-1}\bet_2\cdots\bet_n\II_{j',\sig'} &\text{ if }
    \II_{i,\nu} \text{ is punctured,}
  \end{cases}
\]  
for certain arrows $\alp_1,\ldots,\alp_m,\bet_1,\ldots,\bet_n$ with
$m, n\geq 1$,  $\alp_1,\bet_1$ ordinary arrows, and
$\alp_m,\bet_n\in\uQ_1^f\setminus Q_1$.

(a) Suppose for example in the second case that we have a morphism of windings
$\psi\df G\ra G_{\tau_f\bx}$ with property (q), which does not factor
of $\pi_\bx$.  Then $G\df\uH\ra\huQ$ must be given by
$H=H(\II^{-1}_{i'',\mu''}\bx''\bet_1^{-1}\bet_2\cdots\bet_n\II_{j',\nu'})$, where $\bx'=\by\bx''$.  
However, in this case $\psi$ does not have property (K).
A contradiction.  The first case is similar.

(b) Assume that $\II_{i,\nu}$ is punctured, and thus
$\tau_f\bx=\II_{i,\nu}^{-1}\bx'\bet_1^{-1}\bet_2\cdots\bet_{n-1}\bet_n\II_{j',\sig'}$.
The other possibility is very similar.
If $n\geq 2$, we have
$\tau\bx=\II_{i,\nu}^{-1}\bx'\bet_1^{-1}\bet_2\cdots\bet_{n-1}\II_{j'',\sig''}$.
If in this case we have $\psi\df G\ra G_{\tau_f\bx}$ which does not factor over
$\sig_\bx$, then clearly $\psi$ lacks property (k).

If $n=1$, the string $\tau\bx\in\Adm(\uQ)$ is obtained from $\bx$ by removing
a hook.  Thus, we have a factorization
$\bx'=(\tau\bx)'\gam_1\gam_2^{-1}\cdots\gam_l^{-1}$ for certain arrows
$\gam_1,\ldots,\gam_l$ with $l\geq 1$. Thus, if $\psi\df G\ra G_{\tau_f\bx}$
has properties (s) and (k), the winding $G\df H\ra\huQ$ with $H=H(\by)$ must
correspond actually to a substring $\by'$ of $(\tau\bx)'$.
\end{proof}  

The next result  is a generalization of~\cite[Thm.~4.3]{BY20}
to  our setting of skewed-gentle polarized quivers, and to our
relaxed definition of fringing. Note, that we allow in our correspondence
also bands.  This makes it necessary to use a setup which is inspired by
Krause's admissible triples from~\cite{Kr91}, and which was prepared
in the previous Lemma.

\begin{Prop} \label{prp:kisses}
  Let $\uQ\subset\uQ^f$ be a skewed-gentle polarized quiver
  with a fringing, and $\bx,\by\in\Adm(\uQ)$.
 
\begin{itemize}  
\item[(a)] If $\by$ is not projective, the map
\[
    \cT_\uQ(\bx,\tau\by)\ra\cK_{\uQ^f}(\tau_f\bx,\tau_f\by),\quad
     (G,\phi_q,\phi_s)\mapsto (\iota G, \pi_\bx\phi_q, \sig_\by\phi_s)
\]
is a bijection from the h-triples for $(\bx,\tau\by)$ to the k-triples for
$(\tau_f\bx,\tau_f\by)$, which respects the types.
If $\by$ is projective, $\cK_{\uQ^f}(\tau_f\bx, \tau_f\by)=\emptyset$. 
\item[(b)]
  There is a natural bijection
  $\cR_\uQ(\bx,\tau\by)\ra\cK_{\uQ^f}(\tau_f\bx,\tau_f\by)$, between the set of
  real h-lines from $\bx$ to $\tau\by$ and the set of kisses from $\tau_f\bx$
  to $\tau_f\by$.
\end{itemize}  
\end{Prop}

\begin{proof}
  (a) is a direct consequence of Lemma~\ref{lem:cce} and the definitions.

  (b) Follows from (a) and Proposition~\ref{prp:trip-line}.
\end{proof}

\begin{Rem} \label{rem:no-kiss}
Let $\bx,\by\in\Adm(\uQ)$, and abbreviate $\tbx:=\tau_{\uQ^f}(\bx)$,
$\tby:=\tau_{\uQ^f}(\by)$. By the construction of $\tau_{\uQ^f}$ and of $\uQ^f$,
a vertex $i\in H_0(\tbx)$ is a boundary vertex if and only if
$G_{\tbx}(i)$ is a fringe vertex. The same statement holds, if we replace
$\tbx$ by $\tby$.  Thus, a vertex $(j,i)$ of $H_{\uQ^f}(\tbx,\tby)$ is a
boundary vertex if and only if $G_\tby(j)=G_\tbx(i)$ is a fringe vertex and
in particular, \emph{both} $j$ and $i$ are boundary vertices.  Let us call
such vertices of $H_{\uQ^f}(\tbx,\tby)$ \emph{corner vertices}.  Clearly,
$H_{\uQ^f}(\tbx,\tby)$ has at most 4 corner vertices.  The connected
component in  $H^+_{\uQ^f}(\tbx,\tby)$ of a corner vertex is clearly a real
h-line, or a dual real h-line, and these are precisely the (possibly dual)
real h-lines, which are not (possibly dual) kisses.
\end{Rem}

\section{The Module Category of a Skewed-Gentle Algebra}
\label{sec:mod}
In this section $\Ka$ will be a  field with $\kar(\Ka)\neq 2$.

\subsection{Skewed-gentle algebras} \label{ssec:skga}
If $\uQ$ is a polarized  quiver, we write $\Ka\uQ:=\Ka Q/\ebrace{\cR(\uQ)}$
for the path algebra of the underlying quiver $Q$ modulo the relations
\[
  \cR(\uQ)=\{\alp\bet\mid \alp,\bet\in\Qord\text{ and } \us(\alp)=\ut(\bet)\}
  \cup\{\eps'\eps-e_{s(\eps)}\mid \eps\in\Qsp\}.
\]
Here, $e_i$ denotes the trivial path concentrated at the vertex $i$.
If $\uQ$ is skewed-gentle, the second set of relations can be rewritten
as $\{\eps_i^2-e_i\mid i\in\Qspv\}$. In this case $\Ka\uQ$ is a skewed-gentle
algebra in the sense of~\cite[Sec.~4.1]{GP99}.
In particular, it is a clannish algebra in the sense of
Crawley-Boevey~\cite{CB89b}, because the polynomial $X^2-1$ has two different
zeros, $-1$ and $1$, in our $\Ka$.
If $\Qsp=\emptyset$ clearly $\Ka\uQ$ is
a gentle algebra.

$\Ka\uQ$ admits an automorphism
$\chi$ of order $2$ which acts on the arrows as follows:
\[
  \chi(\alp)=\begin{cases} \phantom{-}\alp &\text{if } \alp\in\Qord,\\
    -\alp &\text{if } \alp\in\Qsp.
  \end{cases}
\]    
For later use it will be convenient to recall how $\Ka\uQ$ can be described
in terms of a quiver $\tQ$ and admissible relations $\tcR$. For this purpose
we define for each vertex $i\in Q_0$
\begin{align*}
S(i) & :=\begin{cases} \{+,-\} &\text{if } i\in\Qspv\\
\{o\} & \text{if } i\in\Qordv,
\end{cases}\\
\tQ_0 & := \{(i,\rho) \mid i\in Q_0 \text{ and } \rho\in S(i)\},\\
\tQ_1 & := \{{_\tau\alp_\sig} \mid \alp\in\Qord, \tau\in S(t(\alp)), \sig\in S(s(\alp))\},\\
s({_\tau\alp_\sig})&:=(s(\alp), \sig),\\
t({_\tau\alp_\sig})&:=(t(\alp), \tau),\\
\tcR &:= \bigcup_{\alp\bet\in\cR^\ord} \{
\sum_{\sig\in S(t(\bet))} {_\tau\alp_\sig}\,{_\sig\bet_\rho}\mid \tau\in S(t(\alp))
  \text{ and } \rho\in S(s(\bet))\}.
\end{align*}
Then it is not difficult to see that
\begin{align*}
\vph_\uQ\df\Ka\uQ &\ra\Ka\tQ/\ebrace{\tcR}, \text{ defined by}\\
  e_i    &\mapsto\sum_{\sig\in S(i)} e_{(i,\sig)} \text{ for all } i\in Q_0,\\
  \alp     &\mapsto \sum_{\tau\in S(t(\alp)), \sig\in S(s(\alp))} {_\tau\alp_\sig}
  \text{ for all } \alp\in\Qord,\\
  \eps_i  &\mapsto e_{(i,+)}- e_{(i,-)} \text{ for all } i\in\Qspv,
\end{align*}
is an $\Ka$-algebra isomorphism.  Note that
$\vph_\uQ^{-1}(e_{(i,\pm)})=1/2(e_i\pm\eps_i)$ for all $i\in\Qspv$.  Moreover,
for a representation $M= ((M_i)_{i\in Q_0}, (M(\alp))_{\alp\in Q_1})$ of $\Ka\uQ$
and $i\in\Qspv$,  the decomposition 
$M_i=\vph_\uQ^{-1}(e_{(i,+)})M_i \oplus \vph_\uQ^{-1}(e_{(i,-)})M_i$ corresponds
precisely to the two eigenspaces of the involutive endomorphism $M(\eps_i)$
of $M_i$.

\subsection{Parametrization of the indecomposable modules}
Let from now on $\uQ$ be a skewed-gentle polarized quiver.
Recall that by Proposition~\ref{prp:Adm}, the completion operation
$\widebar{?}$, defined in Equation~\eqref{eq:cpl}, 
induces a bijection between the equivalence classes of admissible
words $\Adm(\uQ)$ and the equivalence classes of strings and bands
$\St(\uQ)\cup [\pBa(\uQ)]$.
In Section~\ref{ssec:Hw} we constructed for each $\bx\in\Adm(\uQ)$
a homomorphism of polarized quivers $G_\bx\df\uH(\bx)\ra\uQ$.  These maps are
the combinatorial skeleton of Crawley-Boevey's description of the
indecomposable representations of a clannish algebra in terms of
arbitrary defining quadratic polynomials $q_i(X)=(X-a_i)(X-b_i)$ with
$a_i,b_i\in\Ka^*$ and $a_i\neq b_i$ for the special loops $\eps_i$.
We define now a simplified
quiver $\uH^o(\bx)$ together with a homomorphism of polarized quivers
\[
  G^o_\bx\df \uH^o(\bx)\ra\uQ,
\]
where the orientation of the arrows in $(G^o(\bx))^{-1}(\Qsp)$ does not depend
on the relation $>$ for words.
Following Crawley-Boevey~\cite{CB88}, this simplified construction is
  sufficient to describe the indecomposable
representations of the skewed-gentle algebras (a special case of clannish
algebras) in our situation, where we consider  
$\Ka\uQ$ with $\kar(\Ka)\neq 2$ and defining relations $X^2-1$ for all special
loops.

The quiver $\uH^o(\bx)$ is obtained from $\uH(\bx)$ by adding to
each arrow
$\xymatrix{a\ar[r]^{\nu_i}&b}$  with  $a\neq b$ and
$G_\bx(\nu_i)\in\Qsp$ an additional arrow $\xymatrix{a &\ar[l]_{\nu_i'} b}$ with
$G^o_\bx(\nu'_i)=G_\bx(\nu_i)$. 
In this case we declare the arrows $\nu_i$ and
$\nu'_i$ to be special. 
Note, that the polarized quiver $\uH^o(\bx)$ in general is not skewed-gentle, since it may contain special arrows which are not loops.
The restriction of $G^o_\bx$ to $\uH(\bx)$ is $G_\bx$.

In this situation we define now a $\Ka$-algebra $A_\bx$, and assign to
each arrow $\alp$ of $H^o(\bx)$ a unit $U_\bx(\alp)$ of $A_\bx$.  In 
Table~\ref{tab:AwUw} we list the algebras $A_\bx$ and the units which are not
of the form $U_\bx(\alp)=1_{A_\bx}$, according to the possible types of
admissible words.

It is  known~\cite[Lemma~2.3~(3)]{BTCB24} that for $\bx$ of type $(u,p)$ or of type $(p,u)$, the $\Ka$-algebra
$A_\bx$ from Table~\ref{tab:AwUw} can be identified with 
the group algebra of the cyclic group $\ebrace{\sig}$ of order $2$. It is
thus isomorphic to the semisimple algebra $\Ka\times\Ka$,
since $\kar(\Ka)\neq 2$. Moreover
$A_\bx\otimes\Ka\cong X^+\oplus X^-$, where $X^\pm$ is the one-dimensional
module on which $T$ acts $\pm 1$.

In Appendix~\ref{ssec:infdih} we will discuss the algebras $A_\bx$ for 
admissible words which contain two punctured letters, in more detail.
\begin{table}[h]
\begin{tabular}{|c|c|c|l|l|}\hline
  type of $\bx$   & $A_\bx$           & non triv. values of $U_\bx$ & autom. $\chi$ & autom. $\iota$ \\[.1cm]\hline
(u,u) & $\Ka$            & - & $\id$ & $\id$ \\[.1cm] \hline
(u,p)  & $\Ka[T]/(T^2-1)$ & $U_\bx(\eta_1)=T$ & $\chi(T)=-T$ &$\id$ \\[.1cm] \hline
(p,u)  & $\Ka[T]/(T^2-1)$ & $U_\bx(\eta_0)=T$ & $\chi(T)=-T$ & $\id$ \\[.1cm] \hline  
(p,p)     & $\displaystyle\frac{\Ka\ebrace{S,T}}{\ebrace{S^2\!-\!1, T^2\!-\!1}}$ 
          & $U_\bx(\eta_0)=S, U_\bx(\eta_1)=T$ 
          &  \parbox[c]{2.0cm}{$\chi(S)=-S$,\\ $\chi(T)=-T$} 
          &  \parbox[c]{2.0cm}{$\iota(S)=T$,\\ $\iota(T)=S$}\\[.1cm] \hline
 (b) & $\Ka[T,T^{-1}]$ 
 & $U_\bx(\nu_0)= {\begin{cases} T &\!\!\!\text{if } x_0\text{ dir.}\\
     T^{-1} & \!\!\!\text{if } x_0 \text{ inv.} \end{cases}}$ & $\id_{\Ka[T,T^{-1}]}$ & $\iota(T)=T^{-1}$  \\[.1cm] \hline
\end{tabular}
\vspace*{1ex}     
\caption{Algebras $A_\bx$, values of $U_\bx$, automorphism $\chi$ and $\iota$}
\label{tab:AwUw}
\end{table}

If we set furthermore $U_\bx(i)=A_\bx$ for each vertex $i\in\uH^o_o(\bx)$,
we see that $U_\bx$ defines a $\Ka\uH^o(\bx)\text{-}A_\bx$-bimodule, which we
denote, by a slight abuse of notation, also by $U_\bx$.
It is free as an $A_\bx$-right module, see for example~\cite[\S 3]{BTCB24}.
On the other hand, since the morphism of polarized quivers
$G^o_\bx\df \uH^o(\bx)\ra\uQ$ sends ordinary arrows to ordinary arrows,
it induces an exact push-forward functor
\[
  G^o_{\bx,\lam}\df\Ka\uH^o\lmd\ra\Ka\uQ\lmd,
\]
which is defined  on a representation $M$ of $\Ka\uH^o(\bx)$ as
\[
(G_{\bx,\lam} M)(a) := \bigoplus_{l\in G^{-1}_\bx(a)} M(l) 
\quad\text{for }  a \text{ either an arrow or a vertex of } \uQ.
\]
All together we obtain for each $\bx\in\Adm(\uQ)$ a functor
\[
  M^o_\bx\df A_\bx\lmd\ra\Ka\uQ\lmd, X\mapsto G^o_{\bx,\lam}(U_\bx\otimes_{A_\bx} X).
\]
\begin{Rem} \label{rem:isos}
  Between  the representations of $\Ka\uQ$ of the form $M^o_\bx(X)$ there are
  a few obvious isomorphisms, which come from isomorphisms between the
  underlying
  quivers.
  \begin{itemize}
    \item
      If $\bx\in \Adm_s(\uQ)$, then $M^o_\bx(X)$ and $M^o_{\bx^{-1}}(X)$
      are naturally isomorphic. 
      If $\bx\in\Adm(\uQ)$ is of type $(p,p)$, then $M^o_\bx(X)$ and
      $M^o_{\bx^{-1}}(X')$ are naturally isomorphic.
      Here, $X'$ is obtained from $X$ by twisting with
     the automorphism of $A_\bx$, which swaps $S$ and $T$.
     See Definition~\ref{rem:Adm} for the definition of $\Adm_s(\uQ)$.
     \item
     If $\bx\in\Adm(\uQ)$ and $\bx'\in [\bx]$, then
    $M^o_\bx(X)$ and $M^o_{\bx'}(X)$ are naturally isomorphic.
    Moreover, in this case $M^o_{\bx^{-1}}(X)$ is naturally
       isomorphic to $M^o_\bx(X^\iota)$. Here, $X^\iota$ is obtained from $X$ by
       twisting with the automorphism $\iota$ of~$A_\bw$, which
       is defined in Table~\ref{tab:AwUw}
     \item
       We have $M^o_\bx(X)^\chi\cong M^o_\bx(X^{\chi'})$, where $\chi'$ is the
       automorphism $\chi$ of $A_\bx$ from Table~\ref{tab:AwUw}, except when
       $\bx$ is
       a primitive band, which contains an odd number of special letters. In that
       case $\chi'(T)=-T$.  On the other hand, $\chi$ is the canonical
       automorphism of $\Ka\uQ$, defined in Section~\ref{ssec:skga},
       which changes the sign of all special loops.
     \end{itemize}
   \end{Rem}   
This motivates the following:
\begin{Def}
Let $\Ka$ be a field with $\kar{\Ka}\neq 2$ and $\uQ$ a skewed-gentle
polarized quiver.  Then, with the notation from Definition~\ref{rem:Adm}, we
set:
\begin{align}
 \MAdm(\uQ)  &:=\{(\bx,X)\mid \bx\in\Adm(\uQ)\text{ and } X\in A_\bx\ind\},\\
[\MAdm(\uQ)]&:=\{([\bx],X)\mid[\bx]\in[\Adm(\uQ)]\text{ and }X\in A_\bx\ind\}.
\end{align} 
Moreover, we consider the equivalence relation $\simeq$ on $[\MAdm(\uQ)]$,
which is defined by
\[
 ([\bx], X)\simeq ([\bx^{-1}], X^\iota)
\]
\end{Def}

We are now ready to state the main result from~\cite{CB89b} (see Section 3.8
   in loc.\!~cit.)  for the special case of a field $\Ka$ with $\kar(\Ka)\neq 2$
   and the choice of $q_i(X)=X^2-1$ for all special loops $\eps_i$, adapted
   to our notation. Recall, that the module category of each skewed-gentle algebra (a special class of clannish algebras) can be embedded into the
   representations of a clan by adding kernels and
   cokernels~\cite[Sec.~2.5]{CB89b}.  
   See also the main result in~\cite{BTCB24} for a more recent reference,
   where clannish algebras are directly addressed.  
   See also~\cite[Rem~2.6~(ii)]{BTCB24} for more context.
   Moreover, recall, that in view of
   Proposition~\ref{prp:Adm} we can replace the corresponding
   strings and bands by admissible words.
\begin{Thm}[Crawley-Boevey \& Bennett-Tennenhaus~\cite{CB89b} \&~\cite{BTCB24}] \label{Thm:classif}
  Let $\uQ$ be a skewed-gentle polarized quiver and $\Ka$ a
  field with $\kar(\Ka)\neq 2$.  Then 
  $\uQ$ the map
  \[
  \MAdm(\uQ)\ \ra\ \Ka\uQ\lmd,  (\bx, X)\mapsto M^o_\bx(X)
  \]
  induces a natural bijection between $[\MAdm(\uQ)]/\simeq$ and the
  isomorphism classes of indecomposable representations of the skewed-gentle
  algebra $\Ka\uQ$.  
\end{Thm}

\begin{Rem} \label{rem:classif}
In case $\Qsp=\emptyset$, Theorem~\ref{Thm:classif} is just the
well-known classification of indecomposable representations  of gentle
algebras~\cite{AS87}, which are in particular special biserial, see~\cite{BR87}, \cite{WW85}. 
In this case no symmetric strings or symmetric (primitive) bands occur.
As observed in~\cite{GP99},  in our situation ($\kar(\Ka)\neq 2$)
the skewed-gentle algebra $\Ka\uQ$ is Morita equivalent to a skew group
algebra $\Ka\tQ\ebrace{\sig}$ (in the sense of~\cite{RR85}) for a gentle polarized quiver
$\tuQ$ with an involution $\sig$.  As hinted in~\cite{CB88},  one can
deduce, with some extra effort, the above theorem from this observation,
at least for $\Ka$ algebraically closed. See for example~\cite{AP17}, though
there, no effort was made to describe the indecomposable $\Ka\uQ$-modules
explicitly.
\end{Rem}

\subsection{Description of the Homomorphisms}
Let $\bx, \by\in\Adm(\uQ)$ and $X\in A_\bx\lmd$, $Y\in A_\by\lmd$.
Our aim is to describe   $\Hom_{\Ka\uQ}(M^o_\bx(X), M^o_\by(Y))$
in terms of the quiver $H_\uQ(\bx,\by)$.  We  consider
\[
f=(f_i)_{i\in Q_0}\in \Hom_{\Ka\uQ}(M^o_\bx(X), M^o_\by(Y))
\text{ with }
f_i\in \Hom_\Ka((M^o_\bx(X)(i), (M^o_\by(Y))(i)),
\]
such that  we have
\begin{equation} \label{eq:hom1}
  f_{t(\alp)}\circ (M^o_\bv(V))(\alp)=(M^o_\bw(W))(\alp)\circ f_{s(\alp)}
\end{equation}
for each arrow  $\alp\in Q_1$.
In view of the definition of $M^o_\bx(X)$ and $M^o_\by(Y)$ we consider
a block decomposition
\[
  f_i = (f_{m,l})_{(m,l)\in D(i)}\text{ for }
  D(i):= (G^o_\bx)^{-1}(i)\times (G^o_\by)^{-1}(i),
\]  
where all  $f_{m,l}\in\Hom_\Ka(X,Y)$.  We observe that we have by definition
\[
  H_0(\bv,\bw)=\cup_{i\in Q_0} D(i).
\]
For each long h-line
$L\in\cH_\uQ(\bv,\bw)$ we define
\begin{multline}
  \Hom_{\Ka\uQ}(M^o_\bx(X), M^o_\by(Y))_L:=\\
  \{ (f_{i,i})_{(j,i)\in H_0(\bx,\by)}\in\Hom_{\Ka\uQ}(M^o_\bx(X), M^o_\by(Y)) \mid f_{j,i}=0 \text{ if } (j,i)\not\in L_0\}.
\end{multline}
We recall that the underlying graph of
$L$ is one of the type $\sfA_n, \sfD'_n, \sftA_n$ or $\sftD'_n$ for some
$n$.  Accordingly, we define $A_L$ to be one of the algebras from
Table~\ref{tab:AwUw}, corresponding to an asymmetric string, a symmetric string,
an asymmetric band, or a symmetric band, respectively.
We observe, that in this situation we have two natural injective
algebra homomorphisms
\[
f_\bx\df A_L\ra A_\bx\text{ and } f_\by\df A_L\ra A_\by. 
\]  
In fact, if the underlying graph of $L$ is of type $\sftA_n$ or $\sftD'_n$,
both morphisms are the identity. If $L$ is of type $\sfA_n$, both maps are
just the usual structure homomorphisms of $A_\bx$ resp.~$A_\by$.  If $L$ is
of type $\sfD'_n$, with loop $(\eta_2,\eta_1)$.
We  define then naturally $f_\bx(T)=U_\bx(\eta_1)$ and $f_\by(T)=U_\by(\eta_2)$.
In particular, $A_\bx$- resp.~$A_\by$-modules have a natural $A_L$-module
structure.

\begin{Thm} \label{thm:hom}
  Let $\uQ$ be a skewed-gentle polarized quiver. For 
  $\bx,\by\in\Adm(\uQ)$, $X\in A_\bx\lmd$, $Y\in A_\by\lmd$ then
\begin{equation} \label{eq:Hom}
    \Hom_{\Ka\uQ}(M^o_\bx(X), M^o_\by(Y))=
   \!\!\bigoplus_{L\in\cH_\uQ(\bx,\by)}\!\Hom_{\Ka\uQ}(M^o_\bx(X), M^o_\by(Y))_L \cong
   \!\!\bigoplus_{L'\in\cR_\uQ(\bx,\by)} \!\Hom_{A_{L'}}(X,Y).
  \end{equation}
Here, for each $L\in\cH_\uQ(\bx,\by)$ and $(j,i)\in L_0$ the projection
\begin{equation} \label{eq:Hom2}
\Hom_{\Ka\uQ}(M^o_\bx(X), M^o_\by(Y))_L\ra \Hom_{A_L}(X,Y), f\mapsto f_{j,i}
\end{equation}
is a well-defined, $\Ka$-linear isomorphism.
\end{Thm}
We  will show this result in Section~\ref{ssec:syseq} below.

\begin{Rem}
  Recall from Proposition~\ref{prp:bij-hl} that we have a natural
  bijection $\iota\df\cR_\uQ(\bv,\bw)\ra\cH_\uQ(\bv,\bw)$ from the set of real
  h-lines to the set of long h-lines. Thus, the isomorphism
  in~\eqref{eq:Hom} follows from~\eqref{eq:Hom2}.
\end{Rem}

\subsection{Proof of Theorem~\ref{thm:hom}} \label{ssec:syseq}
Since $G_\bx^o\df\uH^o(\bx)\ra\uQ$ and $G_\by^o\df\uH^o(\by)\ra\uQ$ are
homomorphisms of polarized quivers, which send ordinary arrows to ordinary
arrows, we can argue as in~\cite{Kr91} that equation~\eqref{eq:hom1} splits
into the following system of equations:
\begin{subequations} \label{eq:hom2-sys}
\begin{alignat}{2} 
  f_{t(\nu),t(\mu)}\circ X(U_\bx(\mu)) &= Y(U_\by(\nu))\circ f_{s(\nu),s(\mu)}
  &\qquad (\nu,\mu)_+ &\in H_{1,+}(\bx,\by),\label{eq:hom2pl} \\
  f_{s(\nu),t(\mu)}     &= f_{t(\nu),s(\mu)}
  &\qquad (\nu,\mu)_\times&\in H_{1,\times}(\bx,\by), \label{eq:hom2x}\\
  f_{t(\eta),t(\mu)}     &= Y(U_\by(\eta))\circ f_{s(\eta),s(\mu)}
  & (\eta,\mu)_\oplus &\in H_{1,o}(\bx,\by), \label{eq:hom2op}\\
  f_{s(\nu),t(\eta)}\circ X(U_\bx(\eta)) &= f_{t(\nu),s(\eta)}
  & (\nu,\eta)_\ominus &\in H_{1,o}(\bx,\by),\label{eq:hom2om}\\
  f_{j,t(\mu)}\circ X(U_\bx(\mu))      &=0  &\qquad (j,\mu)&\in H^l(\bx,\by),
  \label{eq:hom2zl}\\
  0 &= Y(U_\by(\nu))\circ f_{s(\nu),i}& \qquad (\nu,i)&\in H^r(\bx,\by),
  \label{eq:hom2zr}
\end{alignat}
\end{subequations}
with $H^l(\bx,\by):=\cup_{\mu\in H_1^{\ord}(\bx)} H^l_\mu(\bx,\by)$ and
$H^r(\bx,\by):=\cup_{\nu\in H_1^{\ord}(\by)} H^r_\nu(\bx,\by)$, where
\begin{align*}
  H^l_{\mu}(\bx,\by) &=\{(j,\mu)\!\in H_0(\by)\!\times\! H_1(\bx)\mid (j,t(\mu))\in
    H_0(\bx,\by)\setminus \{t(\nu,\mu)\mid (\nu,\mu)\in H_1(\bx,\by)\}\},\\
  H^r_{\nu}(\bx,\by) &=\{(\nu,i)\in H_1(\by)\!\times\!H_0(\bx)\mid (s(\nu),i)\in
     H_0(\bx,\by)\setminus \{s(\nu,\mu)\mid (\nu,\mu)\in H_1(\bx,\by)\}\}.
\end{align*}

\begin{Rem}
(1) We used in equations~\eqref{eq:hom2x}, \eqref{eq:hom2op}
and~\eqref{eq:hom2om}  that $U_\bx(\nu)=1=U_\bx(\nu')$ if $\nu\neq \nu'$ is
a pair of special arrows in $\uH^o(\bx)$ with
$s(\nu)=t(\nu')\neq s(\nu')=t(\nu)$.

(2) We observe, that for each arrow in $H_{1,o}$ we obtain
from equation~\eqref{eq:hom1} in fact a pair of
equivalent equations, and we skipped in the corresponding equations~\eqref{eq:hom2op}
and~\eqref{eq:hom2om} one from each such pair of equations.

More precisely, if $(\eta,\mu)_\oplus\in H_{1,o}(\bx,\by)$ with 
$t(\eta,\mu)_\oplus=(t(\eta),t(\mu))$ and $s(\eta,\mu)_\oplus=(s(\eta),s(\mu))$, 
where $s(\mu)\neq t(\mu)$ and $s(\eta)=t(\eta)$.
Then there is also an arrow $\mu'\in \uH^o_1(\bx)$ with $s(\mu')=t(\mu)$ and
$t(\mu')=s(\mu)$.  The corresponding equations become~\eqref{eq:hom2op} and
\[
\hspace*{-7cm}\eqref{eq:hom2op}'\qquad\qquad  f_{t(\eta),s(\mu)}=Y(U_\by(\eta))\circ f_{s(\eta),t(\mu)}
\]
since $U_\bx(\mu)=U_\bx(\mu')=1_{A_\bx}$. These two equations are equivalent
since $U_\by(\eta)^2=1_{A_\by}$, and for this reason we drop in this case 
equation~\eqref{eq:hom2op}' in our system of equations~\eqref{eq:hom2-sys}.\\
Similarly,  for each arrow $(\nu,\eta)_\ominus\in H_{1,o}(\bx,\by)$ we obtain additionally to 
equation~\eqref{eq:hom2om} an equation
\[
\hspace{-7cm}\eqref{eq:hom2om}'\qquad\qquad f_{t(\nu),t(\eta)}\circ X(U_\bx(\eta)) = f_{s(\nu),s(\eta)}.
\]
Again, these two equations are equivalent, and for this reason we drop in this case 
equation~\eqref{eq:hom2om}' from our system of equations~\eqref{eq:hom2-sys}.

(3)  It is easy to see that $H^l_{\mu}(\bx,\by)=\emptyset$ if $G_\bx(\mu)\in\Qsp$
and $H^r_\nu(\bx,\by)=\emptyset$ if $G_\by(\nu)\in\Qsp$.
\end{Rem}

Since the $U_\bx(\mu)\in A_\bx$ and
$U_\by(\nu)\in A_\by$ are by construction  units, $X(U_\bx(\mu))\in\GL(X)$
and $Y(U_\by(\nu))\in\GL(Y)$. Thus, we conclude from the system of
equations~\eqref{eq:hom2-sys} that for
$f= (f_{j,i})_{(j,i)\in H_0(\bx,\by)}\in \Hom_{\Ka\uQ}(M^o_\bx(X),M^o_\by(Y))$ and
$(j,i), (j',i')\in H_0(\bx,\by)$ lying on the same connected component of
$H_\uQ(\bx,\by)$ there exist $A\in\GL(X)$ and $B\in\GL(Y)$ such that
$f_{j',i'}= Bf_{j,i}A^{-1}$.  In particular, if $(j',i')$ does \emph{not} lie
on a long h-line, $f_{j',i'}=0$ by equations~\eqref{eq:hom2zl}
and~\eqref{eq:hom2zr} and the definition of the red vertices in
Section~\ref{ssec:Hvw}. In fact, 
\[
\Hred(\bx,\by)=\{(j,t(\mu))\mid (j,\mu)\in H^l(\bx,\by)\}\cup
\{(s(\nu),i)\mid (\nu,i)\in H^r(\bx,\by)\},
\]
and  thus $f_{j,i}=0$ for all $(j,i)\in\Hred(\bx,\by)$ since the maps $X(U_\bx(\mu))$
resp. $Y(U_\by(\nu))$ are invertible. This shows the first claim of Theorem~\ref{thm:hom}, and we conclude that the projection
\[
\Phi_L\df\Hom_{\Ka\uQ}(M^o_\bx(X), M^o_\by(Y))_L\ra \Hom_\Ka(X,Y), f\mapsto f_{j,i}
\]
is injective for any  long h-line $L$ and $(j,i)\in L_0$. 
It remains to show that
$\operatorname{Im}(\Phi_L)=\Hom_{A_L}(X,Y)\subset\Hom_\Ka(X,Y)$.
To this end we go over the $4$ possible types of $L$.
\begin{itemize}
\item[($\sfA_n$)] In this case, $A_L=\Ka$ and there are no restrictions on
  $f_{j,i}$ since $L$ contains no red points.
\item[($\sfD'_n$)] In this case we may assume that $(j,i)\in L_0$ is the
  unique vertex of $L$ with  loop $(\eta',\eta)\in H_{1,+}(\bx,\by)$.
  In this case equation~\ref{eq:hom2pl} for $(\eta',\eta)$ means
  \[
    f_{j,i}\circ X(U_\bx(\eta))=Y(U_\by(\eta'))\circ f_{j,i}.
  \]  
  Thus, in fact $f_{j,i}\in\Hom_{A_L}(X,Y)$.
\item[($\sftA_n$)]  In this case we may assume that $\bx=\by$ is an asymmetric
  band and $(j,i)=(0,0)$. Then the equations~\eqref{eq:hom2pl} for
  $(\nu_i,\nu_i)_+$ with $i=1,2,\ldots, n$ force $f_{0,0}=f_{1,1}$ since
  $U_\by(\nu_i)=1$ for $i=1,2,\ldots, n$. Thus, equation~\eqref{eq:hom2pl}
  for $(\nu_0,\nu_0)_+$ implies
  \[
    f_{0,0}\circ X(U_\bx(\nu_0))=Y(U_\by(\nu_0))\circ f_{0,0}.
  \]  
  Thus, in fact $f_{0,0}\in\Hom_{A_L}(X,Y)$.
\item[($\sftD'_n$)]   In this case we may assume that $\bx=\by$ is an asymmetric
  band and $(j,i)=(1,1)$. Then the equations~\eqref{eq:hom2pl} for
  $(\nu_i,\nu_i)_+$ with $i=1,2,\ldots, n-1$ force $f_{1,1}=f_{n,n}$ since
  $U_\by(\nu_i)=1$ for $i=1,2,\ldots, n-1$. Thus, equation~\eqref{eq:hom2pl}
  for $(\eta_1,\eta_1)_+$ and $(\eta_2,\eta_2)_+$ imply
  \[
    f_{1,1}\circ X(U_\bx(\eta_i))=Y(U_\by(\eta_i))\circ f_{1,1}\text{ for } i=1,2.
  \]  
  Thus, in fact $f_{1,1}\in\Hom_{A_L}(X,Y)$.
\end{itemize}

Using the same arguments as in~\cite[Sec.~5]{Ge99}, we can deduce from 
Theorem~\ref{thm:hom} the following combinatorial description of the
Auslander-Reiten translate in $\Ka\uQ\lmd$:

\begin{Cor} \label{cor:DTr}
Let $\uQ$ be a skewed-gentle polarized quiver and $\Ka$ a field with
$\kar(\Ka)\neq 2$.
Suppose  $\bx\in\Adm(\uQ)$ is not a projective string, and $X\in A_\bx\lmd$
indecomposable, then the Auslander-Reiten translation of $M^o(X)$ is
    $\tau_{\Ka\uQ} M^o_\bx(X)\cong M^o_{\tau_{\uQ}(\bx)}(X^\chi)$.
\end{Cor}

\subsection{E-invariant} \label{ssec:Einv}
For a finite dimensional $\Ka$-algebra $A$ and $M, N\in A\lmd$ the
\emph{E-invariant} is by definition
\[
  E_A(M,N):=\dim_\Ka\Hom_A(M,\tau_A(N)) +
  \dim_\Ka\Hom_A(N,\tau_A(M))
\]

Let  $\uQ$ be skewed-gentle polarized quiver with fringing
$\uQ^f$, $\Ka$ a field with $\kar(\Ka)\neq 2$ and
$(\bx,X), (\by,Y)\in\MAdm(\uQ)$. Our aim is to find a formula for
$E_{\Ka\uQ}(M^o_\bx(X), M^o_\by(Y))$ in terms of the K-triples
$\cK_{\uQ^f}(\tau_f\bx, \tau_f\by)$ and $\cK_{\uQ^f}(\tau_f\by, \tau_f\bx)$. 

\begin{Def} \label{def:E-invar}
  Let $\uQ$ be a skewed-gentle polarized quiver with fringing $\uQ^f$, and
  $\bx,\by\in\Adm(\uQ)$.  We abbreviate $\tbx:=\tau_{\uQ^f}(\bx)$ and
  $\tby:=\tau_{\uQ^f(\by)}$.

    For $\sfT\in\{\sfA, \sfD', \sftA, \sftD'\}$ we define the
    set of \emph{$K$-triples of type $\sfT$} as follows:
\begin{align*}
 \cK_{\uQ^f}(\tbx,\tby)_{\sfT} &:= \{[(G,\phi_q,\phi_s)]
      \in\cK_{\uQ^f}(\tbx,\tby)\mid \uH(G)\text{ is of type } \sfT_n
      \text{ for some } n\},\\
 \sfT_{\uQ^f}(\bx,\by)&:= \cK_{\uQ^f}(\tbx,\tby)_\sfT\coprod\cK_{\uQ^f}(\tby,\tbx)_{\sfT}.\\
\intertext{The \emph{pairs of punctured letters of type $\sfD'$} are given by}\\
\sfP_\uQ(\bx,\by) &:=\{(j,i)\in\{0,1\}\times\{0,1\}\mid
(\eta_j,\eta_i)\in\Hsp(\by)\times\Hsp(\bx)\text{ and }\\
&\qquad G_\by(\eta_j)=G_\bx(\eta_i)\} \setminus \sfP'_\uQ(\bx,\by),\text{ with}\\
\sfP'_\uQ(\bx,\by)&:=\begin{cases}
\{(0,0),(1,1)\}&\text{ if } \bx=\by^{\phantom{-1}}\in\AdmSt(\uQ)\\
\{(1,0),(0,1)\}&\text{ if } \bx=\by^{-1}\in\AdmSt(\uQ)\\
   \qquad\emptyset       &\text{ else.}
  \end{cases}    
\intertext{The \emph{band orientation} is}
  \Diag_b(\bx,\by)&:=\begin{cases}
\phantom{-}  1 &\text{if } [\bx]=[\by]\quad\in\Adm_b(\uQ),\\
 -1 &\text{if } [\bx]=[\by^{-1}]\in\Adm_b(\uQ),\\
\phantom{-} 0  &\text{else.}
  \end{cases}
\end{align*}
\end{Def}

\begin{Expl}
Note that $\sfP(\bx,\by)=\emptyset$ unless there are (special) loops 
$\eta\in H_1^\spe(\bx)$ and $\eta'\in H_1^\spe(\by)$ with $G_\bx(\eta)=G_\by(\eta')$.\\
If $\bx$ is a symmetric band and $\{\eta_0,\eta_1\}=H_1^\spe(\bx)$ is the set of (special) loops of $\uH(\bx)$, then we have $\sfP_\uQ(\bx,\bx)=\{(0,1), (1,0)\}$ if $G_\bx(\eta_0)=G_\bx(\eta_1)$, and $\sfP_\uQ (\bx,\bx)=\emptyset$ else.\\
In the situation of Example~\ref{ssec:exHxy} we have $\sfP_\uQ(\bx,\by)=\{(0,1), (1,1)\}$
\end{Expl} 

\begin{Rem}
  (0) Note that for each admissible triple $(G,\phi_q,\phi_s)$ with
  $[(G,\phi_q,\phi_s)]\in\sfD'_{\uQ^f}(\bx,\by)$  we have
  $\uH(G)_1^{\spe}=\{\eta_G\}$ for the unique loop $\eta_G$ of $\uH(G)$.  

  (1)  The sets $\sfT_{\uQ^f}(\bx,\by)$ for
  $\sfT\in\{\sfA, \sfD', \sftA, \sftD'\}$
  depend a priori on the choice of the  fringing  $\uQ^f$.  
  However, by Proposition~\ref{prp:kisses} this dependence is only up to a
  natural bijection.
  
  (2)  If the type  $[(G,\phi_q,\phi_s)]\in\cK_{\uQ^f}(\tbx,\tby)$ belongs
  to $\{\sftA, \sftD'\}$, both $\phi_q$ and $\phi_s$ are isomorphisms.
  In particular, $\tbx=\bx$ and $\tby=\by$ as well as $[\bx]\simeq [\by]$.
  We conclude that
\begin{align*}  
\abs{\sftA_{\uQ^f}(\bx,\by)}\cdot\abs{\sftD'_{\uQ^f}(\bx,\by)} &=0\quad\text{and}\\
\abs{\sftA_{\uQ^f}(\bx,\by)}+\abs{\sftD'_{\uQ^f}(\bx,\by)} &=
  2\abs{\Diag_b(\bx,\by)}\in\{0, 2\}.
\end{align*}

  (3) We will see in the next Lemma that
  $\abs{\sfD_{\uQ^f}(\bx,\by)}=\abs{\sfP_\uQ(\bx,\by)}$,
  and thus
  \[
  \abs{\sfA_{\uQ^f}(\bx,\by)}+\abs{\sfP_\uQ(\bx,\by)}+2\abs{\Diag_b(\bx,\by)}
    =\abs{\cK_{\uQ^f}(\tbx,\tby)\cup \cK_{\uQ^f}(\tby,\tbx)}.
  \]
\end{Rem}

\begin{Lem}\label{Lem:P=D}
We have a natural bijection
\begin{multline*}
  P\df\sfD'_{\uQ^f}(\bx,\by)\ra\sfP_{\uQ}(\bx,\by),\\
  [(G,\phi_q,\phi_s)] \mapsto \begin{cases}
    (j,i) &\text{if } (\phi_q(\eta_G), \phi_s(\eta_G))=(\eta_i,\eta_j)
    \text{ and } \operatorname{Im}(\phi_q)\subset\uH(\bx)\\
      (i,j) &\text{if } (\phi_q(\eta_G), \phi_s(\eta_G))=(\eta_i,\eta_j)
      \text{ and } \operatorname{Im}(\phi_q)\subset\uH(\by),
    \end{cases}   
\end{multline*}
where $\eta_G$ is the unique loop  of $\uH(G)$ for each
$[(G,\phi_q,\phi_s)]\in \sfD'_{\uQ^f}(\bx,\by)$.  
\end{Lem}

\begin{proof}
By construction, the elements of $\sfP_{\uQ}(\bx,\by)$ can be identified with
the loops of $H_\uQ(\bx,\by)$, which do not lie on a diagonal or an
anti-diagonal. 
Next, we  have clearly a natural bijection between the loops of 
$H_\uQ(\bx,\by)$ and the loops of 
$H_f(\bx,\by):=H_{\uQ^f}(\tau_f\bx,\tau_f\by)$, 
see  the remark at the end of Section~\ref{ssec:fring}.
Moreover, each loop $(\eta_j,\eta_i)$ of
$H_f(\bx,\by)$ with $(i,j)\in \sfP_\uQ(\bx,\by)$ lies  on a kiss of type
$\sfD'$.  In fact,  each loop  $(\eta_j,\eta_i)$ in $H_f(\bx,\by)$ lies on a
(possibly dual) real h-line  by Remark~\ref{rem:dual}~(2). This real
(possibly dual) h-line is indeed a (possibly dual) kiss of type $\sfD'$
if $(j,i)\in \sfP_\uQ(\bx,\by)$, by the discussion in Remark~\ref{rem:dual}~(2).
Our claim follows now from the bijection between kisses and K-triples from
Proposition~\ref{prp:trip-line}. 
\end{proof}  

With the notation from Definition~\ref{def:E-invar} we have the
following formula:
\begin{Prop} \label{E-inv}
  Let $\uQ$ be a skewed-gentle polarized quiver with fringing $\uQ^f$
  and $\Ka$ a field with
$\kar(\Ka)\neq 2$.
For  $\bx,\by\in\Adm(\uQ)$, and $X\in A_\bx\lmd$, $Y\in A_\by\lmd$ we have
\begin{multline} \label{eqn:E-inv}
  E_{\Ka\uQ}(M^o_\bx(X), M^o_\by(Y)) = 
    \abs{\sfA_{\uQ^f}(\bx,\by)}\times (\dim_\Ka X)\times (\dim_\Ka Y)  
 \ +\hspace{-1.3em} \sum_{(j,i)\in \sfP_{\uQ}(\bx,\by)}\hspace{-1.2em} \dim_\Ka \Hom_{A(j,i)}(X, Y^\chi)\\ 
   \hspace*{5em}+\abs{\Diag_b(\bx,\by)}\times
    (\dim_\Ka \Hom_{A_\bx}(X', Y^\chi) + \dim_\Ka \Hom_{A_\bx}(Y',X^\chi)),\\
\end{multline}
where $Z':=Z$ if $\Diag_b(\bx,\by)\in \{0, 1\}$ and $Z':=Z^\iota$ if $\Diag_b(\bx,\by)=-1$
for any $Z\in A_\bx\lmd$.  
Moreover,
$A(j,i):=\Ka[T]/(T^2-1)$ comes with injective $\Ka$-algebra homomorphisms
$A(j,i)\ra A_\bx, T\mapsto U_\bx(i)$ and $A(j,i)\ra A_\by, T\mapsto U_\by(j)$,
which give $X$ and $Y^\chi$ an $A(j,i)$-module structure via restriction
of scalars.
\end{Prop}

\begin{proof}
  For typographical reasons we will use the following abbreviations:
  $\tau:=\tau_\uQ$ and $\tbx:=\tau_{\uQ^f}(\bx), \tby:=\tau_{\uQ^f(\by)}$ as
  well as $\cK_f:=\cK_{\uQ^f}$.
  
By combining Corollary~\ref{cor:DTr}, Theorem~\ref{thm:hom} and
Proposition~\ref{prp:kisses} we calculate
\begin{multline*}
  \dim_\Ka\Hom_{\Ka\uQ}(M^o_\bx(X), \tau_{\Ka\uQ}(M^o_\by(Y))) \\
\begin{aligned}
  \qquad  &=\dim_\Ka\Hom_{\Ka\uQ}(M^o_\bx(X), (M^o_{\tau(\by)}(Y^\chi)))
  =\sum_{L'\in\cR_\uQ(\bx,\tau(\by))}\dim_\Ka\Hom_{A_{L'}}(X, Y^\chi)\\[1ex]      
  &=\sum_{[(G,\phi_q,\phi_s)]\in\cK_f(\tbx,\tby)}
    \dim_\Ka\Hom_{A(G,\phi_q,\phi_s)}(X, Y^\chi)\\[1ex]
  & =\abs{\cK_f(\tbx, \tby)_{\sfA}}\times\dim_\Ka \Hom_\Ka (X, Y^\chi)\\
  &\qquad +\hspace*{-2em} \sum_{[(G,\phi_q,\phi_s)]\in\cK_f(\tbx,\tby)_{\sfD'}}
    \hspace*{-3em}\dim\Hom_{A(G,\phi_q,\phi_s)}(X, Y^\chi)
    +\abs{\Diag_b(\bx,\by)}\cdot (\dim_\Ka\Hom_{A_\bx}(X', Y^\chi))\\[1ex]
  & =\abs{\cK_f(\tbx, \tby)_{\sfA}}\times (\dim_\Ka X)\cdot (\dim_\Ka Y)\\
  &\qquad +\sum_{(j,i)\in\sfP^+_\uQ(\bx,\by)}\dim_\Ka\Hom_{A(j,i)}(X, Y^\chi)
    +\abs{\Diag_b(\bx,\by)}\cdot (\dim_\Ka\Hom_{A_\bx}(X', Y^\chi)),\\  
  \end{aligned}
\end{multline*}
where we used in the last step Lemma~\ref{Lem:P=D} together with 
the notation $\sfP^+_\uQ(\bx,\by):= P(\cK_f(\tbx,\tby))$. Similarly, we obtain
\begin{multline*}
    \dim_\Ka\Hom_{\Ka\uQ}(M^o_\by(Y), \tau_{\Ka\uQ}(M^o_\bx(X))) \\
\begin{aligned} \qquad
  & =\abs{\cK_f(\tby, \tbx)_{\sfA}}\times (\dim_\Ka Y)\cdot (\dim_\Ka X)\\
  &\qquad +\sum_{(j,i)\in\sfP^+_\uQ(\by,\bx)}\dim_\Ka\Hom_{A(j,i)}(Y, X^\chi)
    +\abs{\Diag_b(\by,\bx)}\cdot (\dim_\Ka\Hom_{A_\by}(Y', X^\chi)),\\  
  & =\abs{\cK_f(\tby, \tbx)_{\sfA}}\times (\dim_\Ka X)\cdot (\dim_\Ka Y)\\
  &\qquad +\sum_{(j,i)\in\sfP^-_\uQ(\bx,\by)}\dim_\Ka\Hom_{A(j,i)}(X, Y^\chi)  
    +\abs{\Diag_b(\bx,\by)}\cdot (\dim_\Ka\Hom_{A_\bx}(Y', X^\chi)),\\
  \end{aligned}  
\end{multline*}
where we used for the second equation that $A(j,i)\cong \Ka[T]/(T^2-1)$ is
a commutative semisimple algebra, and thus
$\Hom_{A(i,j)}(Y, X^\chi)\cong\Hom_{A(i,j)}(X^\chi,Y)$, together with the
fact that $\chi$ is an automorphism of order 2.  Moreover, we abbreviated
$\cP^-_\uQ(\bx,\by):= P(\cK_f(\by,\bx))$.  

Summing up the two equations we obtain, in view of the definitions, the
desired formula.
\end{proof}

\subsection{g-vectors} \label{ssec:gvec}
Let $A$ be a finite dimensional $\Ka$-algebra over a field $\Ka$, and
$S_1, S_2, \cdots, S_n$ pairwise non-isomorphic representatives of the
simple $A$-modules.  In view of the formula in~\cite[Prp.~5.3]{AIR14},
(see also~\cite[Sec.~3.2]{CLS15}), we may define
the \emph{g-vector} of a finite-dimensional $A$-module $M$ as
\begin{equation} \label{eq:g-vec}
  \bg_A(M):= (\dim\Hom_A(S_i,\tau_A(M)) -\dim\Hom_A(M, S_i))_{i=1,2,\ldots, n}.
\end{equation}
The g-vector contains relevant information about the minimal projective
presentation of $M$, see~\cite[Sec.~5.1]{AIR14} for more details and
background. 
Note however, that we use the opposite sign convention for our g-vectors.

In our context, $A=\Ka\uQ$ for a skewed-gentle polarized quiver, and $\Ka$
is a field with $\kar(\Ka)\neq 2$.  In particular, $A$ is a skewed-gentle
algebra.
We recall the definition of the Gabriel quiver $\tQ$ of $\Ka\uQ$ from
Section~\ref{ssec:skga}  and set
$\bs_i=\II_{i,+}^{-1}\II_{i,-}$ for all $i\in Q_0$.  
The modules
\[
  S_{(i,\rho)}:=M^o(\bs_i, V_\rho)\quad\text{with}\quad
  (i,\rho)\in \tQ_0 
\]
represent the  isoclasses of the simple $\Ka\uQ$-modules,
and they are pairwise non-isomorphic.
Here, we write for the sake of a uniform notation $V_o:=\Ka$, and 
$V_\pm$ denotes the 1-dimensional $\Ka[T]/(T^2-1)$-module on which $T$ acts by
$\pm 1$. 

Let now $\uQ\subset\uQ^f$ be a fringing of $\uQ$, as in
Section~\ref{ssec:fring}, and write for $\bx\in\Adm(\uQ)$
$\tau_f\bx:=\tau_{\uQ^f}(\bx)$, see Section~\ref{ssec:AR-adm}.

\begin{Def} With the above notation, and $i\in Q_0$, we set 
  \begin{alignat*}{2}
A_i^+(\bx) & :=\{j\in H_0(\tau_f\bx) &\mid G_{\tau_f\bx}(j)&=i
            \text{ and two ordinary arrows start at } j\},\\
A_i^-(\bx) & :=\{j\in H_0(\tau_f\bx)&\mid G_{\tau_f\bx}(j)&=i
            \text{ and two ordinary arrows end at } j\},\\
D_i^+(\bx) & :=\{\eta_j\in H_1^{\mathrm{sp}}(\tau_f\bx)&\ \mid G_{\tau_f\bx}(\eta_j)&=\eps_i
\text{ and } \{\alp\in H_1^{\mathrm{ord}}(\tau_f\bx) \mid
             t(\alp)=t(\eta_j)\}=\emptyset \},\\
D_i^-(\bx) & :=\{\eta_j\in H_1^{\mathrm{sp}}(\tau_f\bx)&\ \mid G_{\tau_f\bx}(\eta_j)&=\eps_i
\text{ and } \{\bet\in H_1^{\mathrm{ord}}(\tau_f\bx) \mid
             s(\bet)=s(\eta_j)\}=\emptyset \}.    
\end{alignat*}
We call the elements of $A_i^\pm(\bx)$ the \emph{proper sources resp.~sinks of
type} $i$ of $H(\tau_f\bx)$.  Whilst the elements of
$D_i^\pm(\bx)\subset H_1^{\spe}(\tau_f\bx)\subset\{\eta_0,\eta_1\}$ are
\emph{the special sources resp.~sinks} of $H(\tau_f\bx)$.
We agree that $D^\pm_i(\bx)=\emptyset$ if $i\in\Qordv$.
\end{Def}

Note, that $A_i^\pm(\bx)$ may be non-empty for any $i\in Q_0$. On the other hand,
$D_i^\pm(\bx)$ can be only non-empty if $i\in\Qspv$. Moreover,
$\sum_{\del\in\{+,-\}}\sum_{i\in Q_0}\abs{D_i^\del(\bx)}$ gives the number of
punctured letters of $\bx$. In particular, this number is at most 2. 
Even if at  $i\in\Qspv$ no ordinary arrow starts, the unique special vertex
of $H(\tau_f\bs_i)$ is a special source of type $i$, but \emph{not} a special
sink, due to the construction of $\tau_f\bs_i$. 

Recall moreover the following conclusions from Example~\ref{expl:TripS}:
\begin{itemize}
\item  For any $i\in Q_0$, all real h-lines of type $\sfA$ in $H_{\uQ}(\bx,\bs_i)$ 
are in fact of type $\sfA_1$. They are in bijection with the sources of $j$ of
$\uH(\bx)$ with $G_\bx(j)=i$.  
\item For $i\in\Qspv$, the real h-lines of type $\sfD'$ 
in $H_\uQ(\bx,\bs_i)$ are in fact all of
type $\sfD'_1$.  They are in bijection with the special vertices $j\in H_0^\spe(\bx)$
with $G_\bx(j)=i$ and where no ordinary arrow terminates. 
\end{itemize}
We will use these observations in the proof of Lemma~\ref{lem:gvec}~(a) and~(c)
below.

\begin{Lem} \label{lem:gvec} With the above notations we have the following:
\begin{itemize}
\item[(a)] $\abs{A_i^+(\bx)}$ is the number of real  h-lines of type $\sfA$
    in $H_\uQ(\bx,\bs_i)$.
\item[(b)] $\abs{A_i^-(\bx)}$ is the number of real  h-lines of type $\sfA$
  in $H_\uQ(\bs_i,\tau_\uQ\bx)$ if $\bx$ is not projective,
  else $A_i^-(\bx)=\emptyset$
  \item[(c)] $\abs{D_i^+(\bx)}$ is the number of real  h-lines of type $\sfD'$
    in $H_\uQ(\bx,\bs_i)$.
\item[(d)] $\abs{D_i^-(\bx)}$ is the number of real  h-lines of type $\sfD'$
  in $H_\uQ(\bs_i, \tau_\uQ\bx)$ if $\bx$ is not projective,
  else $D_i^-(\bx)=\emptyset$. 
  \end{itemize}
\end{Lem}

\begin{proof}
  (a)   By the construction of $\tau_f$ and of $\uH(\tau_f\bx)$,
  the sources $j$ of $\uH(\bx)$ with $G_\bx(j)=i$ 
  correspond precisely to $A_i^+(\bx)$.  
  So our claim follows from the discussion in Example~\ref{expl:TripS}.
  
  (b) By Proposition~\ref{prp:kisses} the real h-lines of type $\sfA_1$ in
  $H_\uQ(\bs_i,\tau\bx)$ are in bijection with the kisses of type $\sfA_1$  in
  $H_{\uQ^f}(\tau_f\bs_i,\tau_f\bx)$.  For $i\in\Qordv$ the quiver
  $H(\tau_f\bs_i)$ has a unique proper source, this source is of type $i$
  (possibly there are two other sources, which are of fringe type). Thus,
  there is a unique morphism of windings from $H(\bs_i)$ to $H(\tau_f\bs_i)$,
  which has properties (q) and (k) from Section~\ref{ssec:triph}. On the
  other hand, $A_i^-(\bx)$ can be identified with the morphisms of winding
  from $H(\bs_i)$ to $H(\tau_f\bx)$ which have properties (s) and (k). Now,
  our claim follows from Proposition~\ref{prp:trip-line}.

  (c)  It is easy to see that the real h-lines of type $\sfD'$ in
  $H(\tau_f\bx,\bs_i)$ can be identified with the real h-lines of type
  $\sfD'$ in
  $H(\bx,\bs_i)$.  Thus, our claim follows from the discussion at the end of
  Example~\ref{expl:TripS}. 

  (d) By Proposition~\ref{prp:kisses} the real h-lines of type $\sfA_1$ in
  $H_\uQ(\bs_i,\tau\bx)$ are in bijection with the kisses of type $\sfA_1$  in
  $H_{\uQ^f}(\tau_f\bs_i,\tau_f\bx)$.  For $i\in\Qspv$ the quiver
  $H(\tau_f\bs_i)$ has a unique (special) loop, say at $1$ and a
  unique ordinary arrow starts at $1$. In view of the discussion at the end of
  Example~\ref{expl:TripS}, we can now conclude as in~(b).
\end{proof} 

\begin{Prop} \label{prp:gvec} 
Let $\uQ$ be a skewed-gentle, polarized quiver with fringing
$\uQ^f$ and $\Ka\uQ$ the corresponding skewed-gentle algebra over a field
$\Ka$ with $\kar(\Ka)\neq 2$.  For $\bx\in\Adm(\uQ)$ and $X\in A_\bx\lmd$
we have
\begin{align} \label{eqn:g-vect}
  \bg_{\Ka\uQ}(M^o(\bx,X))_{(i,\rho)}= 
&(\abs{A_i^-(\bx)}-\abs{A_i^+(\bx)})\times \dim X\\      
& \sum_{\eta_j\in D^-_i}\dim\Hom_{A(\eta_j)}(V_\rho, X^\chi) 
-\sum_{\eta_j\in D^+_i}\dim\Hom_{A(\eta_j)}(X, V_\rho),
\end{align}
where $A(\eta_j)= \Ka[T]/(T^2-1)$, and we have
an algebra homomorphism $A(\eta_j)\ra A_{\tau_f\bx}, T\mapsto U_{\tau_f\bx}(\eta_j)$.
This gives $X\in A_{\tau_f\bx}\lmd$, via restriction, an $A(\eta)$-module
structure.
\end{Prop}

\begin{proof} 
In view of Formula~\eqref{eq:g-vec} and Theorem~\ref{thm:hom}, this follows
indeed from  Lemma~\ref{lem:gvec}. 
\end{proof}

\begin{Expl} Consider the skewed-gentle polarized quiver $\uQ$ from
  Example~\ref{ssec:ExSt} with the first fringing $\uQ^f$ from
  Section~\ref{ssec:fring}. We calculate the g-vectors
  \[
    \bg(M)= (\bg_{1,o}, \bg_{2,-},\bg_{2,+},\bg_{3,o})\in\ZZ^4
  \]  
  of some indecomposable
  finite dimensional $\Ka\uQ$-modules $M$ with $\bx\in\Adm(\uQ)$.
  \begin{itemize}
  \item  
    $\bv=\II_{1,-}^{-1}\gam\bet\eps^*\bet^{-1}\II_{3,+}\in\St(\uQ)$  In
  this case $\bx:=A(\bw)=\II_{1,-}^{-1}\gam\bet\eps\bet^{-1}\II_{3,+}\in\Adm(\uQ)$
  and
$\tau_f\bx=\II_{6,-}^{-1}\phi_{6,1}\gam\bet\eps\bet^{-1}\gam^{-1}\phi_{1,5}\II_{5,-}$.
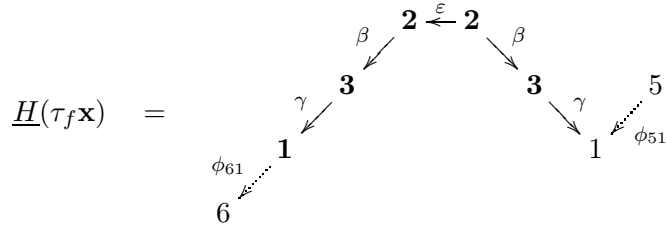
\[\uH(\tau_f\bx)\quad =\quad\vcenter{\xymatrix @-1.1pc{
    &&&\ar_{\bet}[ld] \mathbf{2}&\ar_{\eps}[l] \mathbf{2}\ar^\bet[rd]\\
    &&\ar_{\gam}[ld]\textbf{3} &&& \textbf{3}\ar^\gam[rd]&&\ar@{.>}^{\phi_{51}}[ld] 5\\
      &\ar@{.>}_{\phi_{61}}[ld]\textbf{1}&&&&&1\\
      6}}
\]
In the above display we label vertices and arrows by their image under
$G_{\tau_f\bx}$.  The \textbf{bold} vertices support the subquiver $\uH(\bx)$.
Since $\bx$ is of type $(u,u)$ (and correspondingly $\bw$ an asymmetric string)
$A_\bx=\Ka$ and thus $X=\Ka$ is the only choice. 
We see that we have one proper source of type~2 and one proper sink of  type~1.
Thus  $\bg(M^o(\bx,\Ka))=(-1,1,1,0)$, according to Proposition~\ref{prp:gvec}.
\item
  $\bw= \II_{1,-}^{-1}\alp^{-1}\eps^*\alp\II_{1,-}\in\St(\uQ)$ is a symmetric
  string and correspondingly $\by:=A(\bw)=\II_{2,-}^{-1}\alp\II_{1,-}\in\Adm(\uQ)$
  is of type $(p,u)$.  Consequently, $\tau_f\by=\alp\phi_{6,1}^{-1}\II_{6,-}$ and
  \[
    \uH(\tau_f\by)\quad = \quad\vcenter{\xymatrix{
        &\ar_{\alp}[ld]\textbf{1}\ar@{.>}^{\phi_{6,1}}[rd]\\
        \ar@(ul,dl)^\eps[] \mathbf{2}&& 6}}.
  \]
  In this case we have a proper source of type 1 and a special sink of type 2.
  Thus, we get from our Corollary
  \[
    \bg(M^o(\by, X^+))= (1,-1,0,0)\quad\text{and}\quad
    \bg(M^o(\by, X^-))= (1,0,-1,0),
\]
since  $A_\by=\Ka[T]/(T^2-1)$ and $X^\pm\cong V^\pm$ are the two isomorphism
classes of indecomposable (simple) $A_\by$-modules.
\end{itemize}
\end{Expl}

\section{Indecomposable  generically
  \texorpdfstring{$\tau$}{τ}-regular irreducible
  components} \label{sec:gentred}
\subsection{Generically \texorpdfstring{$\tau$}{τ}-regular
  irreducible components for tame algebras} \label{ssec:tau-reg}
Let $\Ka$ be an algebraically closed field, $Q$ a finite quiver and $I<\Ka Q$
an admissible ideal. Then $A:=\Ka Q$ is a finite dimensional
basic $\Ka$-algebra with Grothendieck group $K_0(A)\cong\ZZ^{Q_0}$.  In this
context, we consider for each dimension vector $\bd\in\NN^{Q_0}$ the affine
$\Ka$-scheme $\rep_A^\bd$ of representations with dimension vector $\bd$.
The algebraic group $G_\bd:=\times_{i\in Q_0}\GL_{\bd(i)}(\Ka)$ acts on the
$\Ka$-rational points $\rep_A^\bd(\Ka)$ by conjugation, and the orbits are
in one to one correspondence with the isoclasses of representations of
$A$ with dimension vector $\bd$.  For $M\in\rep_A^\bd(\Ka)$ we write
$\cO_M\subset\rep_A^\bd $ for the $G_\bd$-orbit $G_\bd\dot M$.   

We denote  by $\Irr(A,\bd)$ the set of irreducible components of $\rep_A^\bd$
and
\[
\Irr(A):=\cup_{\bd\in\NN^{Q_0}}\Irr(A,\bd).  
\]
A component $Z\in\Irr(A)$ is
called \emph{indecomposable} if it contains a dense subset of indecomposable
representations of $A$.  It is well known, that on each $Z\in\Irr(A)$ the
function
$Z\ra\NN, M \mapsto\codim_Z(\cO_M)$
is upper semicontinuous.  Thus,
\[
  c(Z):=\min\{\codim_Z(\cO_M)\mid M\in Z\}
\]
is the generic value of the codimension of the orbits in $Z$.

For $X, Y\in\Irr(A)$ the function
\[
  E_A\df X\times Y\ra\NN, (M,N)\mapsto\dim\Hom_A(M,\tau N)+\dim\Hom_A(N,\tau M)
\]  
is upper semicontinuous by~\cite[Cor.~1.4]{GLFS23}, and thus
\[
  e_A(X, Y):= \min\{E_A(M,N)\mid (M,N)\in X\times Y)\}
\]  
is the generic value of $E_A$ on $X\times Y$.
It will be convenient to consider also the generic value
\[
  e_A(Z):=\min\{\frac{1}{2} E_A(M,M)\mid M\in Z\}\in\NN
\]
We write
\[
  \tIrr(A):=\{Z\in\Irr(A)\mid c(Z)=e_A(Z)\}
\]
for the set of \emph{generically} $\tau$-\emph{regular} irreducible components,
and $\tIrrind(A)$ is the subset of generically $\tau$-regular
irreducible components which are indecomposable. These components were called
$\tau$-reduced or strongly reduced in some older publications.  In a recent preprint~\cite{BobS25}, Bobiński and Schröer make a very good point, why these components should be better called $\tau$-\emph{regular} components.

With the help of the generic
E-invariant $e_A$ it is possible to write each $Z\in\tIrr(A)$ in an
essentially unique way as a direct sum of \emph{indecomposable}, generically
$\tau$-regular irreducible components. See~\cite[Sec.~5.2]{CLS15}
for more details.  The indecomposable, generically $\tau$-regular irreducible
components of a \emph{tame} algebra have a convenient characterization,
which is kind of folklore.  
In fact, part~(a) of the following Proposition is~\cite[Corr.~1.7]{GLFS23}, 
see also the comment in \emph{loc.~\!cit.} after Corollary~1.7.  
Part~(b) of the Proposition follows easily from the definitions. 
Part~(c) and~(d) follow from the discussion in~\cite[Sec.~3]{GLFS22}.

\begin{Prop} \label{prop:E-tame}
  Suppose that $A=\Ka Q/I$ as above is tame.  Then the following  holds:
  \begin{itemize}
    \item[(a)] For each $Z\in\tIrr(A)$ we have $e_A(Z,Z)=0$.
    \item[(b)] If $M\in A\lmd$ is indecomposable with $E_A(M,M)=0$, then
      the (Zariski) orbit closure $Z:=\overline{\cO_M}$ belongs to
      $\tIrrind(A)$, and $c(Z)=0$. 
    \item[(c)] If $C\subset\rep_A^\bd(\Ka)$ is an irreducible curve which
      intersects infinitely many $G_\bd$-orbits of bricks, then
      \[ Z:=\overline{\cup_{X\in C} \cO_X}\in\tIrrind(A)\]
      and $c(Z)=1$.
    \item[(d)]  Let $Z\in\tIrrind(A)$, then we have $e_A(Z)\in\{0, 1\}$.
      If $e_A(Z)=0$ then $Z$ is of the form described in $(b)$, else $Z$ is
      of the form described in (c).
    \end{itemize}   
\end{Prop}

Recall, that over an algebraically closed field $\Ka$, bricks are precisely
the $A$-modules $M$ with $\End_A(M)=\Ka$

\subsection{Combinatorial description}
In this section we fix a skewed-gentle polarized quiver $\uQ$ with fringing
$\uQ^f$  and $\Ka$, an algebraically closed field with
$\kar{\Ka}\neq 2$. We identify the skewed-gentle algebra $\Ka\uQ$ along the
isomorphism $\phi_Q$ with $\Ka\tQ/\tcR$ as in Section~\ref{ssec:skga}.
In particular, $\tQ$ is the Gabriel quiver of $\Ka\uQ$, and we identify the dimension vectors of representations of $\Ka\uQ$
with vectors in $\NN^{\tQ_0}$.

For each $\bx\in\Adm(\uQ)$ we define a subset of
$S:=(\{-1,1\}\times\{-1,1\})\cup\{\ast, (\ast,\ast)\}$: 
\[
S(\bx):=\begin{cases}
\{(1,1)\}                        &\text{if } \type(\bx)=(u,u),\\
\{(1,-1), (1,1)\}                &\text{if } \type(\bx)=(u,p),\\
\{(-1,1), (1,1)\}                &\text{if } \type(\bx)=(p,u),\\
\{1,-1\}^2\cup\{(\ast,\ast)\}    &\text{if } \type(\bx)=(p,p),\\
\{\ast\}                &\text{if } \type(\bx)=(b).
\end{cases}
\]
The set $S(\bx)$ symbolically parametrizes in a convenient way the simple
$A_\bx$-modules.
For later use we introduce several operations  on the elements of $S$, namely
we set $(s_0,s_1) := s$  for $s\in S\setminus\{*\}$ and  
\begin{alignat*}{2}
  s^\iota &:=\begin{cases} (s_1,s_0)&\text{if } s=(s_0,s_1)\in\{-1,1\}^2,\\ s &\text{else,}
  \end{cases} &
   \qquad\wt(s) &:=\begin{cases} 2 &\text{if } s=(*,*),\\ 1 &\text{if } s\neq (*,*),\end{cases}\\   \\
   s^\chi &:=\begin{cases} (-s_0,-s_1)&\text{if } s=(s_0,s_1)\in\{-1,1\}^2,\\ \ s&\text{else.}
   \end{cases} &
 \end{alignat*}
With this at hand we introduce, with the notation from
Definition~\ref{rem:Adm}, the (equivalence classes of) \emph{tagged
  admissible words}
\begin{align*}
  \Adm^*(\uQ)   &:=\{(\bx,s)\mid \bx\in\Adm(\uQ)\text{ and } s\in S(\bx)\},
  \text{ and}\\
  [\Adm^*(\uQ)] &:=\{(\bx,s)\in\Adm(\uQ)\mid\bx\in\AdmSt(\uQ)\} \cup
                  [\AdmBa(\uQ)]\times\{*\}.
\end{align*}
We introduce on $[\Adm^*(\uQ)]$ an equivalence relation $\simeq$ by
declaring $(\bx^{-1},s^\iota)\simeq (\bx,s)$.
Next we introduce for each $(\bx,s)\in\Adm^*(\uQ)$ a set of simple
$A_\bx$-modules as follows:
\[
C(\bx,s):=\begin{cases}
\{V_o\}    &\text{if } \type(\bx)=(u,u),\\
\{V_\del\}  &\text{if } \type(\bx)=(u,p) \text{ and } s=(1,\del 1),\\
\{V_\del\}  &\text{if } \type(\bx)=(p,u) \text{ and } s=(\del 1,1),\\
\{W_{(1;i\cdot j)}^{\chi^{(1-i)/2}}\} &\text{if } \type(\bx)=(p,p)\text{ and } s=(i,j)\in\{1,-1\}^2,\\
\{\widetilde{V}_{(1,t)}\mid t\in\Ka^*\setminus\{-1,1\}\}&\text{if }\type(\bx)=(p,p)
  \text{ and } s=(\ast,\ast),\\
  \{V_{(1,t)}\mid t\in\Ka^*\} &\text{if } \type(\bx)=(b),
\end{cases}
\]
where we used the notations from the beginning of Section~\ref{ssec:gvec}
in case $\bx$ has type $(u,u)$, $(u,p)$ or $(p,u)$, and we use the notation introduced in Appendix~\ref{ssec:infdih} when $\bx$ has type $(p,p)$ or $(b)$. In particular,  for
$R:=\Ka\ebrace{S,T}/(S^2-1, T^2-1)$ and the $1$-dimensional $R$-module
$W_{(1;i\cdot j)}^{\chi^{(1-i)/2}}$ we have
$W_{(1;i\cdot j)}^{\chi^{(1-i)/2}}(S)=i$ and
$W_{(1;i\cdot j)}^{\chi^{(1-i)/2}}(T)=j$.

Note that for all $\bx\in\Adm(\uQ)$ the set
$\cup_{s\in S(\bx)} C(\bx,s)$ covers all isomorphism classes of simple
  $A_\bx$-modules. 
Note moreover that we may see $C(\bx,s)$ as an affine curve
when $s\in\{*, (*,*)\}$.

We observe that for $(\bx,s)\in\Adm_*(\uQ)$ there exists 
$\bd(\bx,s)\in\ZZ_{\geq 0}^{\tQ_0}$ such that
\[
\bd(\bx,s)=\dimv M^o_\bx(X) \text{ for all } X\in C(\bx,s).
\]
\begin{Def} \label{def:comb-inv}
  With the notations  from Section~\ref{ssec:Einv} and
  from Section~\ref{ssec:gvec} we define for
  $(\bx,s), (\by,t)\in\Adm^*(\uQ)$ and $(i,\rho)\in\tQ_0$   the
  \emph{combinatorial E-invariant} and \emph{combinatorial g-vector}:
\begin{align}
  \begin{split} \label{eq:comb-e}
  e_\uQ((\bx,s), (\by,t)) &= \abs{\sfA_{\uQ^f}(\bx,\by)}\times \wt(s)\times\wt(t)\; \\
&\qquad +  \sum_{(j,i)\in \sfP_\uQ(\bx,\by)} \!\!\!d_2(s_i,(t^\chi)_j) 
\;+\; 2\abs{\Diag_b(\bx,\by)} \times d_3(s',t^\chi)\quad\in\NN,
\end{split}\\[1ex]
\begin{split} \label{eq:comb-g} 
  \bg_\uQ(\bx,s)_{(i,\rho)}&= (\abs{A_i^-(\bx)}-\abs{A_i^+(\bx)})\times \wt(s)\\
  &\qquad
  +\sum_{\eta_j\in D^-_i(\bx)} d_2(\rho\cdot 1, -s_j)
  -\sum_{\eta_j\in D^+_i(\bx)} d_2(s_j,\rho\cdot 1)\ \in\ZZ,
\end{split}
\end{align}    
where we use additionally the symmetric functions
$d_2\df\{-1,1,*\}^2\ra\{0,1,2\}$, $d_3\df S^2\ra\{0,1\}$, which are
defined as follows:
\[
\begin{array}{r|r|r|r|}
  d_2 & -1 & 1 & *\\ \hline
   -1 &  1 & 0 & 1\\ \hline
    1 &  0 & 1 & 1\\ \hline
    * &  1 & 1 & 2\\ \hline
\end{array}\qquad
d_3(s,t):=\begin{cases}
1 &\text{if } s=t\in\{-1,1\}^2,\\ 0 &\text{else.}
\end{cases}
\]
Thus, for example, $d_2(*,-1)=1=d_2(-1,*)$.

Moreover, we agree that $s':=(s_1,s_0)$ if $\Diag_b(\bx,\by)=-1$ and 
$s\in\{-1,1\}^2$. In all other cases  $s'=s$. If $s=(*)$ we agree that $s_0=*=s_1$.  In this case 
$\bx$ is of type $(b)$ and thus $\sfP_\uQ(\bx,\by)=\emptyset$. 
Note, that in this case the second summand  in Equation~\eqref{eq:comb-e}, and the second summand in Equation~\eqref{eq:comb-g} are both the empty sum.  
The same applies for $t=(*)$. We also agree that, $\pm\cdot 1:=\pm 1$ and $o\cdot 1:=*$, but note that  $\rho=*$ implies $i\in\Qordv$ and thus $D^\pm_i(\bx)=\emptyset$. 

\medskip

We will need also the crucial subset
\begin{equation} \label{deqn:admts}
\Adm_\tau^*(\uQ):=\{(\bx,s)\in\Adm^*(\uQ)\mid e_\uQ((\bx,s),(\bx,s))=0\}
\end{equation}
of $\tau$-\emph{generic} pairs in $\Adm^*(\uQ)$.
\medskip

We also consider for $(\bx,s)\in\Adm^*(\uQ)$ the Zariski closure
\begin{equation} \label{deqn:Mclos}
\overline{M}^o_{(\bx,s)}:= \overline{\cup_{X\in C(\bx,s)} \cO_{M^o_\bx(X)}}
\subset\rep_{\Ka\uQ}^{\bd(\bx,s)}(\Ka).
\end{equation}
\end{Def}

\begin{Rem} \label{rem:e-comb}
(1) For each $(\bx,s)\in\Adm^*(\uQ)$ the closed set
  $\widebar{M}^o_{(\bx,s)}\subset\rep_{\Ka\uQ}^{\bd(\bx,s)}(\Ka)$ contains,
  by Crawley-Boevey's Theorem~\ref{Thm:classif} the union of orbits 
  \[
  \cup_{V\in C(\bx,s)} \cO_{M^o_\bx(V)}
  \]
  as an irreducible, dense subset of \emph{indecomposable} representations.
  It consists of infinitely many orbits if and only if $s\in\{*,(*,*)\}$.
  In particular, each $\widebar{M}^o_{(\bx,s)}$ is irreducible.
  Moreover, we have by the same theorem
  \[
  \widebar{M}^o_{(\bx,s)}=\widebar{M}^o_{(\by,t)} \Iff [(\bx,s)]=[(\by,t)]. 
  \]
  
(2)  The functions $e_\uQ$ resp.~$\bg_\uQ$ descend obviously to
 $([\Adm^*(\uQ)]/\simeq) \times ([\Adm^*(\uQ)]/\simeq)$ 
  resp.~to
  $[\Adm^*(\uQ)]/\simeq$.  In particular,
  the subset $[\Adm^*_\tau(\uQ)]/\simeq$ of $[\Adm^*(\uQ)]/\simeq$
  is well-defined.

(3) Recall from Definition~\ref{def:E-invar}, that $\abs{\sfA_{\uQ^f}(\bx,\by)}$
is the number of equivalence classes of K-triples of type $\sfA$ from $\tau_f\bx$ to $\tau_f\by$ 
plus the number of equivalence classes of K-triples of type $\sfA$ from $\tau_f\by$ to $\tau_f\bx$.  
\end{Rem}

\begin{Thm} \label{thm:tirrind}
Let $\uQ$ be a skewed-gentle polarized quiver and $\Ka$ an
algebraically closed field with $\kar(\Ka)\neq 2$.
\begin{itemize}
\item[(a)]  We have the following equalities:
  \begin{align*}
  e_{\Ka\uQ}(\widebar{M}^o_{(\bx,s)}, \widebar{M}_{(\by,t)})&=e_\uQ((\bx,s),(\by,t)) 
\text{ for all } (\bx,s), (\by,t)\in\Adm^*(\uQ),\\    
  e_{\Ka\uQ}(\widebar{M}^o_{(\bx,s)})&=\begin{cases}
  \frac{1}{2} e_\uQ((\bx,s), (\bx,s))   &\text{if } s\in\{-1,1\}^2\\
 \frac{1}{2} e_\uQ((\bx,s), (\bx,s))+1 &\text{if } s\in\{*, (*,*),\}  
  \end{cases}  \quad\text{and}\\
  \bg_{\Ka\uQ}(\widebar{M}^o_{(\bx,s)}) &=\bg_\uQ(\bx,s)
  \quad\text{ for each } (\bx,s)\in\Adm^*(\uQ).
\end{align*}
In other words, the combinatorial E-invariant and combinatorial g-vector
calculate the generic values of the (algebraic) E-invariant and 
g-vector on the closed subsets of the form $\widebar{M}^o_{(?)}$.  

\item[(b)] The set $\Adm^*_\tau(\uQ)$ has the following simplified description:
\begin{align*}  
\Adm^*_\tau(\uQ) &=\{(\bx,s)\in\Adm^*(\uQ)\mid \bx\in\Adm_\tau(\uQ)\text{ and } s\in S'(\bx)\},
\quad\text{where}\\
\Adm_\tau(\uQ) &:=\{\bx\in\Adm(\uQ)\mid\sfA_{\uQ^f}(\bx,\bx)=\emptyset\},\\ 
S'(\bx) &:= \begin{cases}  \{(-1,-1), (1,1)\} &\text{if }
  \bx=\II^{-1}_{p,-}\bx'\II_{p,-} \text{ for some } p\in\Qspv,\\
  S(\bx) &\text{else.}
\end{cases}
\end{align*}
\item[(c)]
  The construction $\widebar{M}^o_{?}$ induces a bijection
  \[
  ([\Adm^*_\tau(\uQ)]/\simeq)\ \ra\ \tIrrind(\Ka\uQ),\quad
    [(\bx,s)]\mapsto \widebar{M}^o_{(\bx,s)}.
  \]
\end{itemize}
\end{Thm}
  
\begin{proof}
(a)  Tracking back the definitions, and by comparing with  the formulae
  for the E-invariant~\eqref{eqn:E-inv} and the  g-vector~\eqref{eqn:g-vect}
  these equalities follow easily.
\medskip
  
(b) When we calculate $e_\uQ((\bx,s), (\bx,s))$according to its definition
in equation~\eqref{eq:comb-e},  we see that it is a sum of non-negative integers.  So $e_\uQ((\bx,s), (\bx,s))=0$ if and only if each of its summands
is $0$. Since $\wt(s)^2\in\{1,4\}$, the first summand is $0$ if and only if
$\abs{\sfA_{\uQ^f}(\bx,\bx)}=0$, or in other words $\bx\in\Adm_\tau(\uQ)$.

Next, we observe that $\abs{\Diag_b(\bx,\bx)}\cdot d_3(s',s^\chi)=0$ for all $(\bx,s)\in\Adm^*(\uQ)$.  In fact, in this situation $s'=s$, and, 
according to Definition~\ref{def:E-invar},
$\abs{\Diag_b(\bx,\bx)}\neq 0$ implies that $\bx$ is of type $(b)$ or of type
$(b,b)$. Now, $d_3((*), (*)^\chi)=0$ and also $d_3((s_0,s_1), (s_0, s_1)^\chi)=
d_3((s_0, s_1), (-s_0, -s_1))=0$ for all $(s_0,s_1)\in\{-1,1\}$. 

From  Definition~\ref{def:E-invar} we see that
$\sfP_\uQ(\bx,\bx)=\{(0,1), (1,0)\}$ if
$\bx=\II^{-1}_{p,-}\bx'\II^{-1}_{p,1}$ for some $p\in\Qspv$.  
Else $\sfP_\uQ(\bx,\bx)=\emptyset$. Thus, we see from the definition of $d_2$, that in the first case it is necessary and sufficient to have
$s\in\{(1,1),(-1,-1)\}$ whilst otherwise there is no restriction on $s$
in order to have 
\[
    \sum_{(j,i)\in \sfP_\uQ(\bx,\bx)} \!\!\!d_2(s_i,(s^\chi)_j)=0. 
\]
\medskip
    
 (c) (i)  We claim that for  $(\bx,s)\in\Adm^*(\uQ)$ we have
    $\widebar{M}^o_{(\bx,s)}\in\tIrrind(\uQ)$ if and only if
    $(\bx,s)\in\Adm^*_\tau(\uQ)$.

 In fact, suppose first $\widebar{M}^o_{(\bx,s)}\in\tIrrind(\Ka\uQ)$.  Since
 $\Ka\uQ$ is tame, we must have by Proposition~\ref{prop:E-tame}~(a) and
 Part~(a) of our theorem
 \[
0=e_{\Ka\uQ}(\widebar{M}^o_{(\bx,s)},\widebar{M}^o_{(\bx,s)})=e_\uQ((\bx,s),(\bx,s)),
\]
and thus by definition $(\bx,s)\in\Adm^*_\tau(\uQ)$.

Suppose next that $(\bx,s)\in\Adm^*_\tau(\uQ)$.  If $s\not\in\{\ast, (*,*)\}$,
we have $C(\bx,s)=\{X\}$ for a one-dimensional $A_\bx$-module $X$, 
and thus by definition
$\overline{M}^o_{(\bx,s)}=\overline{\cO_{M^o_\bx(X)}}$.  
By (a) we have in this case
$E_{\Ka\uQ}(M^o_\bx(X))=\frac{1}{2}e_\uQ((\bx,s),(\bx,s))=0$ for the
indecomposable $\Ka\uQ$-module $M_\bx^o(X)$. 
Thus, $\overline{M}^0_{(\bx,s)}$ is a generically $\tau$-regular indecomposable
irreducible component with a dense orbit by Proposition~\ref{prop:E-tame}~(b).
On the other hand, for $s\in\{*, (*,*)\}$, we have $\bx\in\Adm_b(\uQ)$ and  the set $C(\bx,s)$ consists of infinitely many isomorphism classes of simple $A_\bx$-modules $X$ with $X^\chi\cong X$.  
Moreover, we have $\sfP_\uQ(\bx,\bx)=\emptyset$ by the discussion in the proof of (b).
On the other hand, it follows then from Corollary~\ref{cor:DTr} that 
$\tau_{\Ka\uQ} M^o_\bx(X)\cong M^o_\bx(X)$ is an indecomposable $\Ka\uQ$-module for all $X\in C(\bx,s)$. 
Thus, we have
\[
  \dim\End_{\Ka\uQ}(M^o_\bx(X))=\dim\Hom_{\Ka\uQ}(M^o_\bx(X), \tau_{\Ka\uQ}M^o_\bx(X))
  =\dim\Hom_{A_\bx}(X,X)=1,  
\]
by the formula for the E-invariant in Proposition~\ref{E-inv}.
Thus, for $(\bx,s)\in\Adm^*_\tau(\uQ)$ with $s\in\{*, (*,*)\}$ the curve
$\cup_{X\in C(\bx,s)} M^o_\bx(X)$ intersects infinitely many orbits of bricks and
$\widebar{M}^o_{(\bx,s)}\in\tIrrind(\uQ)$ by Proposition~\ref{prop:E-tame}~(c).
\medskip 
  
(ii) In view of Remark~\ref{rem:e-comb}~(1), we conclude that the restriction  
$\overline{M}^o_?\df ([\Adm_*^\tau(\uQ)]/\simeq ) \ra\ \tIrrind(\Ka\uQ)$
is well-defined and injective. 
\medskip

(iii) For the surjectivity we observe first, that for each $\bx\in\Adm(\uQ)$
the set $\cup_{s\in S(\bx)} C(\bx,s)$ covers precisely all isomorphism classes of  \emph{simple} $A_\bx$-module.  
 If $X\in A_\bx\ind$ is not simple, we must have $\bx\in\Adm_b(\uQ)$.   
 It is not difficult to deduce, with the help of Appendix~\ref{ssec:infdih},    
 that for such $X$ one of the following two possibilities occurs:
    \begin{itemize}
    \item[(i)] $X\not\cong X^\chi$ and $\dim\Hom_{A_\bx}(X,X^\chi)\geq 1$.
    \item[(ii)] $X\cong X^\chi$ and $\dim\Hom_{A_\bx}(X,X)\geq 2$.
    \end{itemize}
Thus, by Proposition~\ref{E-inv} we have in both cases
$E_{\Ka\uQ}(M^0_\bx(X), M^o_\bx(X))\neq 0$.  
In view of Proposition~\ref{prop:E-tame}~(d) this implies,  that
    $\overline{\cO_{M^o_\bx(X)}}$ can't be of the form $Z\in\tIrrind(\Ka\uQ)$ with
    $c(Z)=0$.   On the other hand, we have in case (i)
    $\tau_{\Ka\uQ}M^0_\bx(X)\not\cong M^0_\bx(X)$ by Corollary~\ref{cor:DTr}, and
    $E(M^o_\bx(X), M^o_\bx(X))=\dim\End_{\Ka\uQ}(M^o_\bx(X))\geq 2$ in case (ii).
    Thus,
    $M_\bx^o(X)$ does not belong to any of the curves which appear in the
    characterization of the $Z\in\tIrrind(\Ka\uQ)$ with $c(Z)=1$ in
    Proposition~\ref{prop:E-tame}. We conclude, by the same proposition and
    the classification result~Theorem~\ref{Thm:classif},  that each 
    $Z\in\tIrrind(\Ka\uQ)$    is of the form
    $\overline{M}^o_\bx(\bx,s)$ for some $(\bx,s)\in\Adm^*(\uQ)$.
\end{proof}

\appendix

\section{The algebra \texorpdfstring{$A_\bw$}{Aw} for symmetric bands}
\label{ssec:infdih}
For a symmetric band $\bw$ let us denote the algebra $A_\bw$ by
$R:=\Ka\ebrace{S,T}/\ebrace{S^2-1, T^2-1}$. It is
isomorphic to the group algebra of the infinite dihedral group.
$R$ is a tame hereditary order, which is closely related to the
classification of
the regular modules over  hereditary algebras of type $\sftD$. Recall, that
we denote by $\chi$ the automorphism of $R$, which is defined by
$\chi(S)=-S$ and $\chi(T)=-T$, and $\iota$ is the automorphism which swaps
$S$ and $T$. Clearly, $\iota\circ\chi=\chi\circ\iota$.

The  indecomposable representations of $R$ have been classified independently by several authors, see for example~\cite{BB81} and~\cite{D86}. 
See also~\cite[Rem.~3.9]{BTCB24} for a larger class of algebras.
We present this classification  with a notation which is compatible with our setup.

The somehow more complicated situation for $\kar(\Ka)=2$ has been studied
previously by Bondarenko~\cite{Bo75} and Ringel~\cite{Ri75}. 

Since the infinite dihedral group is the semidirect product of the infinite
cyclic group and a cyclic group of order two, $R$ is isomorphic to
the skew group algebra  $\Ka[X,X^{-1}]\ebrace{\sig}$ of the cyclic group $\ebrace{\sig}$ of order 2 over  $\Ka[X,X^{-1}]$,
where $\sig$ acts on $\Ka[X,X^{-1}]$ via the involution 
that swaps $X$ and $X^{-1}$. This isomorphism sends
$X\otimes 1$ to $ST$ and $1\otimes\sig$ to $T$ (and $X\otimes\sig$ to $S$).

We denote by $V_{m,s}$ the $m$-dimensional representation of $\Ka[X,X^{-1}]$
on which $X$ acts by the $m\times m$ Jordan block $J_m(s)$ for some
$s\in\Ka\setminus\{0\}$.
Then we have $V_{m,s}^\sig\cong V_{m,s^{-1}}$ and thus $V_{m,s}^\sig\cong V_{m,s}$
if and only if $s\in\{-1,1\}$. 
Similarly, we denote  by $\widetilde{V}_{m,s}$ the $2m$-dimensional
representation of
$R$ with
\[
    \widetilde{V}_{m,s}(S)=\begin{pmatrix} 0 &J_m(s)\\ J_m(s)^{-1}& 0\end{pmatrix}
    \qquad\text{and}\qquad
    \widetilde{V}_{m,s}(T)=\begin{pmatrix} 0 & E_m \\ E_m & 0\end{pmatrix}.
\]
On the other hand, with
\[
  (\tilde{S}_{m,\pm 1})_{i,j} := \begin{cases}
    (-1)^{i+1} (\pm 1)^{i+j}  \binom{i}{j} & \text{if } 1\leq j\leq i\leq m\\
    0                                 & \text{if } 1\leq i< j\leq m,
    \end{cases}
\]
we observe that
\[
  \tilde{S}_{m,\pm 1}^2=E_m\quad\text{and}\quad
  \tilde{S}_{m,\pm 1}\cdot J_m(\pm 1)=J_m(\pm 1)^{-1} \cdot \tilde{S}_{m,\pm 1}.
\]
We conclude that with
\[
  W_{m,\pm 1}(S):=\tilde{S}_{m,\pm 1}\cdot J_m(\pm 1)\quad{and}\quad
  W_{m,\pm 1}(T):=\tilde{S}_{m,\pm 1}
\]
we get for each $m\in\NN_0$ two indecomposable, $m$-dimensional representations
of $R$, which are not isomorphic. In particular, we have
$W_{1,\pm 1}(S)=\pm 1$ and $W_{1,\pm 1}(T)=1$ whilst $W_{1,\pm 1}^\chi(S)=\mp 1$ and
$W_{1,\pm 1}^\chi(T)=-1$.

From the above considerations we obtain the
following classification result,  which seems to be folklore.

\begin{Prop} \label{lem:repR}
  Let $\Ka$ be an algebraically closed field with $\kar(\Ka)\neq 2$. 
  With the above notations, for each $s\in\Ka\setminus\{0\}$ and
  $m\in\ZZ_{>0}$ holds
  \[
    R\otimes_{\Ka[X,X^{-1}]} V_{m,s}\cong
    \begin{cases} \widetilde{V}_{m,s} &\text{if } s\not\in\{-1,1\},\\
      W_{m,s}\oplus W_{m,s}^\chi &\text{if } s\in\{1,-1\}.
    \end{cases} \]
  The modules $\widetilde{V}_{m,s}$ with $s\in\Ka\setminus\{-1,0,1\}$
  and $W_{m,\pm 1}, W_{m,\pm 1}^\chi$ present a complete list of the indecomposable
  representations of $R$, with $\widetilde{V}_{m,s}\cong\widetilde{V}_{m,s^{-1}}$
  the only isomorphism between them.  Moreover, we have
  $\widetilde{V}_{m,s}^\chi\cong\widetilde{V}_{m,s}\cong\widetilde{V}_{m,s}^\iota$ for  all $s\in\Ka\setminus\{-1,0,1\}$, and  $W_{m,1}^\iota\cong W_{m,1}$,
  $W_{m,-1}^\iota\cong W_{m,-1}^\chi$ for all $m$.  
\end{Prop}

\section{Proof of Proposition~\ref{prp:Adm}}
\label{sec:PfAdm}
  Let $\bw=w_0w_1\cdots w_l\in\St(\uQ)\cup\pBa'(\uQ)$. We have to distinguish
  several cases, with each one requiring slightly different arguments.

 (1)  $\bw$ is an asymmetric string.  Following
 Section~\ref{ssec:huQA}, the word $A(\bw)=\tw_0\tw_1\cdots\tw_l$
  is  a string of type $(u,u)$ for $\huQ$, and
  $\overline{A(\bw)}=\bw$. 

  First, suppose that $A(\bw)$ is not admissible. Then we may assume without
  restriction that for some $j$ we have $F(\tilde{w_j})=\eps^*$ is special and
  $A(\bw)_{[j-1]}^{-1}\prec A(\bw)_{[j+1]}$ whilst $\tilde{w}_j=\eps$ is a
  direct letter. By the definition of $A(\bw)$ this means that
  $\bw_{[j-1]}^{-1}>\bw_{[j+1]}$.  These two inequalities can hold only at the
  same time if
\begin{equation} \label{eq:facAw}
    A(\bw)=\tby\eta^{-1}\tbx^{-1}\eps\tbx\eta^{-1}\tbz
\end{equation}
with certain words $\tbx,\tby,\tbz$ for $\huQ$. Moreover, with $i=j-l(\tbx)-1$
and $k=j+l(\tbx)$ the letters $F(\tw_i)=\eta^*=F(\tw_k)$
are necessarily special.  

In fact, we have then
\begin{align*}
A(\bw)_{[j-1]}^{-1}=\tbx\eta\tby^{-1}&\prec A(\bw)_{[j+1]}=\tbx\eta\tbz\text{ and}\\
               \bw_{[j-1]}^{-1}&>\bw^{[j+1]} \text{ implies}\\
  \by^{-1}=F(\tby^{-1}) &>F(\tbz)=\bz.
\end{align*}
On the other hand, in the
following chain of inequalities
  \[
   \by^{-1}<\bx^{-1}\eps^*\bx\eta^*\bz<\bx^{-1}\eps^*\bx\eta^*\by^{-1}<\bx
  \]
 the first and the third one hold because of the construction of $A(\bw)$,
 which is codified in the orientation of the letters $\tw_i=\eta^{-1}$ and
 $\tw_k=\eta^{-1}$.
 The second inequality follows trivially from $\by^{-1}>\bz$. A contradiction.

 Next, suppose  $\tbw'=\tw'_0\tw'_1\cdots\tw'_l\neq A(\bw)$ is an admissible
 string with $\overline{\tbw'}=\bw$. Then there exists $j\in\{1,2,\ldots,l-1\}$
 with $\tw'_j\neq\tw_j$.  In particular, $F(\tw'_j)=F(\tw_j)$ is a special
 letter $\eps^*$, and we may assume that $\Del_\bw(j)\leq\Del_\bw(j')$ for
 all $j'$ with $\tw'_{j'}\neq\tw_{j'}$. See Section~\ref{ssec:OrdStBa} for
 the definition of $\Del_\bw(j)$.
 We may assume further, without loss of generality, that
  \[
    (\tbw'_{[j-1]})^{-1}\prec(\tbw')^{[j+1]}\text{ and }
    \bw_{[j-1]}^{-1}>\bw^{[j+1]}.
  \]r
Thus, we have a factorization as in~\eqref{eq:facAw}, which implies  
$\Del_\bw(j)>\Del_{\tbw'}(j)$. With $d':=l(\bx)=\Del_{\tbw'}(j)$
and $i:=j-d'-1$ and $k:=j+d'+1$ we have here $\tw'_i=\eta^{-1}=\tw'_k$.
We have also the corresponding factorization
  \[
    \bw=\by' a \bu^{-1}\eta^*\bx^{-1}\eps^*\bx\eta^*\bu b\bz',
  \]
where, with $d:=l(\bx)+1+l(\bu)=\Del_\bw(j)$,   
the letters $w_{j-d-1}=a$ and $w_{j+d+1}=b$ are ordinary direct ones.

If $\Del_\bw(i)<l(\bu)$ and $\Del_\bw(k)<l(\bu)<d$, then by our hypothesis $\tw'_i=\tw_i$
and $\tw'_k=\tw_k$.   However, by the above factorization of $\bw$, this means
that $\tw_i^{-1}=\tw_k$, in contradiction to $\tbw'_i=\tbw'_k=\eta^{-1}$.

Thus, due to the position of the direct letters $w_{j-d-1}=a$ and 
$w_{j+d+1}=b$, we see that either $\Del_\bw(i)=l(u)$ or $\Del_\bw(k)=l(u)<d$. This forces in the first case $\tw'_{i}=\eta$ and in the second case 
$\tw'_k=\eta$. A contradiction.
 
(2) $\bw=\eps_p^*\bv\eps_q^*\bv^{-1}$ is a symmetric band with
$\bv=w_1w_2\cdots w_{n-1}$ and $2n=l+1$.  In this case
$A(\bw)=\tw_0\tw_1\cdots\tw_n\in\St(\huQ)$ is of type $(p,p)$ with
$\tw_0=\II_{p,-1}^{-1}$ and $\tw_n=\II_{q,-1}$. By construction,
$F(A(\bw))F(\tw_{n-1}^{-1}\cdots\tw_2^{-1}\tw_1^{-1})=\hat{F}(A(\bw))=\bw\in\pBa$,
so that the
second condition for $A(\bw)$ to be an admissible string is fulfilled.

Let us now suppose that $A(\bw)$ is not admissible. Without loss of generality
we have then for some $j\in\{1,2,\ldots,n-1\}$ the following:
\begin{itemize}
\item
  $\tw_j=\eps$ with $F(\eps)=\eps^*$ a special letter and thus
  $\bw_{[j-1]}^{-1}(\bw^{[j+1]})>\bw_{[j+1]}\bw^{[j-1]}$,
\item  
  $A(\bw)_{[j-1]}^{-1}<A(\bw)^{[j+1]}$.
\end{itemize}  
Similarly, in the case~(1), this means that we have a factorization as
in~\eqref{eq:facAw}. In particular, we have again $F(\tw_i)=F(\tw_j)=\eta^*$
special letters.

Note however, that $\by^{-1}=F(\tby)^{-1}$ and
$\bz=F(\tby)$ are in this situation \emph{not} right inextensible, and
thus the two words might be incomparable.  However, we have the following
inequalities and factorizations from the construction of $\uH(\bw)$ for
a symmetric band, if we take into account the orientation of the letters
$\tbw_i, \tbw_j$ and $\tbw_k$:
\begin{alignat}{4}
  \by^{-1}(\bw^{[i+1]})^{-1} &=\bw_{[i-1]}^{-1}(\bw^{[i+1]})^{-1}  &&<\ &
  \bw^{[i+1]}\bw_{[i-1]}     &=\bx^{-1}\eps^*\bx\eta^*\bz\bv^{-1}\bw_{[i-1]}
  \label{eq:Ad2}\\
  \bx\eta^*\bz\bv^{-1}\bw_{[j-1]} &=\bw^{[j+1]}\bw_{[j-1]}&&<\ &
  \bw_{[j-1]}^{-1}(\bw^{[j+1]})^{-1}&=\bx\eta\by^{-1}(\bw^{[j+1]})^{-1} \label{eq:Ad3}\\
  \bx^{-1}\eps^*\bx\eta^*\by^{-1}(\bw^{[k+1]})^{-1}&=\bw_{[k-1]}^{-1}(\bw^{[k+1]})^{-1}
  &&<\ &
  \bw^{[k+1]}\bw_{[k-1]}&=\bz\bv^{-1}\bw_{[k-1]}.\label{eq:Ad4}
\end{alignat}  
Since $i<j<k$ we conclude from~\eqref{eq:Ad3} that
\begin{equation} \label{eq:Ad5}
  \bz\bv^{-1}\bw_{[k-1]}<\by^{-1}(\bw^{[i+1]})^{-1},
\end{equation}  
see end of Section~\ref{ssec:lex}.
From this we obtain with~\eqref{eq:Ad2}
and~\ref{eq:Ad4} the following chain of inequalities
\[
  \bx^{-1}\eps^*\bx\eta^*\by^{-1}(\bw^{[k+1]})^{-1}<\bz\bv^{-1}\bw_{[k-1]}<
  \by^{-1}(\bw^{[i+1]})^{-1}<\bx^{-1}\eps^*\bx\eta^*\bz\bv^{-1}\bw_{[i-1]}.
\]  
Again, since $i<j<k$, this implies
$\by^{-1}(\bw^{[i+1]})^{-1}<\bz\bv^{-1}\bw_{[k-1]}$, contradicting~\eqref{eq:Ad5}.

Next, suppose $\tbw'=\tw'_0\tw'_1\cdots\tw'_n\neq A(\bw)$ is an admissible
string of type $(p,p)$ with $\hat{F}_A(\tbw')=\bw$. Then there exists
$j\in\{1,2,\ldots,n-1\}$ with $\tw'_j\neq\tw_j$. In particular
$F(\tw'_j)=F(\tw_j)$ is a special letter $\eps^*$, and we may assume that
$\Del_\bw(j)\leq\Del_\bw(j')$ for all $j'$ with $\tw'_{j'}\neq\tw_{j'}$.
See Section~\ref{ssec:OrdStBa} for the definition of $\Del_\bw(j)$.
We may assume further without loss of
generality that
  \[
    (\tbw'_{[j-1]})^{-1}\prec(\tbw')^{[j+1]}\text{ and }
    \bw_{[j-1]}^{-1}(\bw^{[j-1]})^{-1}>\bw^{[j+1]}\bw_{[j-1]}.
  \]
Thus we have a factorization as in~\eqref{eq:facAw}, which implies  
$\Del_\bw(j)>\Del_{\tbw'}(j)$. With $d':=l(\tbx)=\Del_{\tbw'}(j)$
and $i:=j-d'-1$ and $k:=j+d'+1$
we have here $\tw'_i=\eta^{-1}=\tw'_k$ again.
We have also the corresponding factorization
\[
 w_{-l+j-1}\cdots w_{-1} w_0\cdots w_lw_{l+1}\cdots w_{l+j-1}= \bw^{[j+1]}\bw\bw_{[j-1]}=\by'\,a\,\bu^{-1}\eta^*\bx^{-1}\eps^*\bx\eta^*\bu\,b\,\bz',
\]
where, with $d:=l(\bu)+1+l(\bx)=\Del_\bw(j)$, the letters
$w_{j-d-1}=a$ and $w_{j+d+1}=b$ ordinary direct ones.

If $\Del_\bw(i)<l(\bu)$ and $\Del_\bw(k)<l(\bu)<d$, then by our hypothesis $\tw'_i=\tw_i$
and $\tw'_k=\tw_k$.   However, by the above factorization of $\bw^{[j+1]}\bw\bw_{[j-1]}$ this means that $\tw_i^{-1}=\tw_k$, in contradiction to $\tbw'_i=\tbw'_k=\eta^{-1}$.

Thus, due to the position of the direct letters $w_{j-d-1}=a$ and $w_{j+d+1}=b$
we see that either $\Del_\bw(i)=l(u)$ or $\Del_\bw(k)=l(u)<d$. This forces
in the first case $\tw'_{i}=\eta$ and in the second case $\tw'_k=\eta$.
A contradiction.

(3)  It should now be obvious how to modify the argument for the remaining
two cases when $\bw$ is either a symmetric string or an asymmetric band.

\subsection*{Acknowledgment}
The author acknowledges partial support from UNAM DGAPA-PAPIIT grant IN116723 (2023-2025). He also wishes to thank an anonymous referee for very careful reading and constructive criticism.

\end{document}